\documentclass[letterpaper,11pt,leqno]{amsart} 

\usepackage[T1]{fontenc}
\usepackage[utf8]{inputenc}
\usepackage[portrait,margin=3.5cm]{geometry} 
\usepackage[mathscr]{eucal}
\usepackage[colorlinks=true, bookmarksdepth=2]{hyperref}
\usepackage{amsmath,amssymb,amsthm,amsfonts,amsbsy,latexsym,dsfont,color,graphicx,enumitem, upgreek, xcolor, xfrac}
\usepackage[normalem]{ulem}
\usepackage[numeric,initials,nobysame]{amsrefs} 
\usepackage[foot]{amsaddr}

\definecolor{aquamarine}{HTML}{0035BE}

\hypersetup{
    citecolor=blue,%
    linkcolor=brickred,
}
\definecolor{mahogany}{cmyk}{0, 0.77, 0.87, 0}
\definecolor{salmon}{cmyk}{0, 0.53, 0.38, 0}
\definecolor{melon}{cmyk}{0, 0.46, 0.50, 0}
\definecolor{yellowgreen}{cmyk}{0.44, 0, 0.74, 0}
\definecolor{brickred}{cmyk}{0, 0.89, 0.94, 0.28}
\definecolor{OliveGreen}{cmyk}{0.64, 0, 0.95, 0.40}
\definecolor{RawSienna}{cmyk}{0, 0.72, 1.0, 0.45}
\definecolor{ZurichRed}{rgb}{1, 0, 0}

\newtheorem{theorem}{Theorem}[section]
\newtheorem{corollary}[theorem]{Corollary}
\newtheorem{proposition}[theorem]{Proposition}
\newtheorem{lemma}[theorem]{Lemma}
\newtheorem{problem}[theorem]{Problem}

\newtheorem{remark}[theorem]{Remark}

\newtheorem{question}[theorem]{Question}
\newtheorem{conjecture}[theorem]{Conjecture}

\theoremstyle{remark}

\numberwithin{equation}{section}

\newcommand{\ext}{\mathrm{ext}}
\newcommand{\dis}{\mathrm{dis}}

\newcommand{\Tdis}{T_{\dis}}

\newcommand{\N}{\mathbb{N}}

\newcommand{\R}{\mathbb{R}}
\newcommand{\Z}{\mathbb{Z}}

\newcommand{\bH}{\mathbb{H}}

\newcommand{\bK}{\mathbb{K}}

\newcommand{\bR}{\mathbb{R}}
\newcommand{\bS}{\mathbb{S}}

\newcommand{\cF}{\mathcal{F}}

\newcommand{\cP}{\mathcal{P}}

\newcommand{\cU}{\mathcal{U}}

\newcommand{\rE}{\mathrm{E}}

\newcommand{\fM}{\mathfrak{M}}

\newcommand{\wh}{\widehat}
\newcommand{\wt}{\widetilde}
\newcommand{\ds}{\displaystyle}
\newcommand{\varep}{\varepsilon}
\newcommand{\ol}{\overline}

\newcommand{\norm}[1]{\| #1 \|}

\newcommand{\Paren}[1]{\left( #1 \right)}

\newcommand{\Abs}[1]{\left| #1 \right|}
\newcommand{\brk}[1]{\langle #1 \rangle}
\newcommand{\parder}[2]{\frac{\partial #1 }{\partial #2 }}

\newcommand{\ind}{\mathds{1}}

\DeclareMathOperator{\sign}{sign}

\newcommand{\dht}{\mathcal{H}}
\newcommand{\drt}{\mathcal{R}}

\newcommand{\ph}{\varphi}
\newcommand{\eps}{\varepsilon}

\newcommand{\expr}[1]{\left( #1 \right)}
\newcommand{\tfloor}[1]{\lfloor #1 \rfloor}

\begin{document}
\dedicatory{In memory of Richard F. Gundy (1933-2023)} 
\title[Discrete singular integrals]{Sharp $\ell^p$ inequalities for discrete singular integrals on the lattice $\Z^d$}

\author{Rodrigo
Ba\~nuelos$^1$}
\address{$^1$Department of Mathematics \\ Purdue University \\ 150 N.\@ University Street, West Lafayette, IN 47907}
\email{banuelos@purdue.edu}

\author{Daesung Kim$^2$}
\address{$^2$School of Mathematics \\ Georgia Institute of Technology \\ 686 Cherry Street, Atlanta, GA 30332}
\email{dkim3009@gatech.edu}

\author{Mateusz Kwa\'snicki$^3$}
\address{$^3$Faculty of Pure and Applied Mathematics \\ Wrocław University of Science and Technology \\ ul.\@ Wybrzeże Wyspiańskiego 27 \\ 50-370 Wrocław, Poland}
\email{mateusz.kwasnicki@pwr.edu.pl}

\thanks{R.~Ba\~nuelos is supported in part by NSF Grant 1403417-DMS. M.~Kwa\'snicki is supported by the Polish National Science Centre (NCN) grant 2019/33/B/ST1/03098.}
\date{\today}

\begin{abstract}  This paper investigates higher dimensional versions of the longstanding conjecture verified in \cite{BanKwa} that the $\ell^p$-norm of the discrete Hilbert transform on the integers is the same as the $L^p$-norm of the Hilbert transform on the real line.  
It computes the $\ell^p$-norms  of a family of discrete operators on the lattice $\Z^{d}$, $d\geq 1.$ They are discretizations of a new class of singular integrals on $\R^d$ that have the same kernels as the classical Riesz transforms near zero and similar behavior at infinity.   
The discrete operators have the same $p$-norms as the classical Riesz transforms on $\R^d$.  They are constructed  as conditional expectations of martingale transforms of Doob h-processes conditioned to exit the upper--half space $\R^d\times \R_{+}$ only on the lattice $\Z^d$. 
The paper also presents a discrete analogue of the classical method of rotations which gives the norm of a different variant of discrete Riesz transforms on $\Z^d$. Along the way a new proof is given based on Fourier transform techniques of the key identity used to identify the norm of the discrete Hilbert transform in \cite{BanKwa}. Open problems are stated. 
\end{abstract}

\maketitle
\setcounter{tocdepth}{1}
\tableofcontents

\section{Introduction and Statement of Results} 
The probabilistic representation \`a la Gundy--Varopoulos \cite{GV79} of the classical Riesz transforms and other singular integrals and Fourier multipliers as conditional expectations (projections) of stochastic integrals, in combination with the sharp martingale inequalities of Burkholder \cite{Bur84} and their versions for orthogonal martingales \cite{BanWang} and non-symmetric transforms \cite{Choi, BanOse1}, has proven to be a powerful tool in obtaining sharp, or near sharp, $L^p$-bounds for these operators in a variety of geometric settings. A particular feature of these techniques is that they give $L^p$-bounds independent of the geometry of the ambient space.  Most notably, the bounds are independent of  the dimension. For example, in \cite{BanWang} such techniques were used to show that the $L^p$-norm, $1<p<\infty$, of the Riesz transforms on $\bR^d$, $d>1$, is the same as that of the Hilbert transform on  $\bR$ found by S. Pichorides \cite{Pic72} and B. Cole (see \cite{Gam}), and to obtain the first explicit bounds for the Beurling--Ahlfors operator. The former was first proved using the method of rotations in \cite{IwaMar}. For some history on norm estimates for the Beurling--Ahlfors transform  motivated by the  celebrated 1982 conjecture of T. Iwaniec \cite{Iwa82} and the current best known bound, see \cite{BanJan} and the  
survey article \cite{Ban}.  

One advantage of the martingale approach is that it immediately extends to geometric and analytic settings well beyond $\R^d$, including the infinite dimensional case of Wiener space and other semigroups such as those of L\'evy processes and discrete Laplacian on groups. The interest on dimension free estimates for Riesz transforms and other operators in harmonic analysis was initiated by the results and questions raised in Stein \cite{SteSome} and Meyer \cite{Mey1}. 
 For some of the now vast literature on dimension free and sharp bounds  for Riesz transforms and Fourier  multipliers in a variety of geometric and analytic settings and many uses of probabilistic tools, we refer the reader to \cite{Ban, BanBog, BBL20, BO15, BO18, OY21, GesMonSak, BanMen,  Gun89, BBLS2021, CarSam, Tor, ADP20, Pet2, Pet3, BanOse1, LiX, NazVol, Ose2, Lus-Piq3, Lus-Piq4,  Nao, Pis, DraVol1, DraVol2, DraVol3, DraVol4, CarDra1, CarDra2, Bak1} and references contained therein.

 \subsection{The probabilistic discrete Hilbert transform on $\Z$}\label{subPDH}
In his celebrated 1928 paper \cite{Riesz}, M.~Riesz solved a problem of considerable interest to the analysis community at the time by showing that the Hilbert transform 
\begin{equation}\label{ConRiesz} 
    Hf(x)=p.v. \frac{1}{\pi}\int_{\R} \frac{f(x-y)}{y} dy, \quad f\in L^p(\bR),
\end{equation} 
is a bounded operator on $L^p(\bR)$, $1<p<\infty$. For some history on this problem and Riesz's solution in 1925 before its publication in 1928, we refer the reader to M.~Cartwright's article ``Manuscripts of Hardy, Littlewood, Marcel Riesz and Titchmarsh,'' \cite{Car1982}. In his paper Riesz also showed that the boundedness of $H$ on $L^p(\bR)$ implies the boundedness of the discrete version $H_{\dis}$ on $\ell^p(\Z)$, where the latter is defined by the convolution with $\{\frac{1}{\pi m}, m\in\Z\setminus\{0\}\}$. That is, 
\begin{equation}\label{disRiesz} 
    H_{\dis}f(n)=\frac{1}{\pi}\sum_{m\in\Z\setminus\{0\}}\frac{f(n-m)}{m},\quad f\in\ell^p(\Z).
\end{equation} 
In fact, Riesz showed that the operator norms satisfy 
\begin{equation}\label{Riesz28}
\|H\|_{L^p\to L^p}\leq  \|H_{\dis}\|_{\ell^p\to \ell^p},  \quad \|H_{\dis}\|_{\ell^p\to \ell^p}\leq C\|H\|_{L^p\to L^p},
\end{equation}
for some constant $C$.

The discrete Hilbert transform was introduced {in 1909} by D.~Hilbert who verified its boundedness on $\ell^2(\Z)$. Proving that the operator norms of $H$ and $H_{\dis}$, $1<p<\infty$,  are the same had been a long-standing open problem motivated in part by an erroneous proof of E.~C.~Titchmarsh in 1926, \cite{Titc26, Titc27}. 

Throughout this paper we will use the notation
\[
p^* = \max\left\{p,\, \frac{p}{p - 1}\right\},\quad 1<p<\infty.
\]
With this notation, the  Pichorides \cite{Pic72} and Cole \cite{Gam} results say that 
\begin{align}\label{pichcole}
\|H\|_{L^p\to L^p}=
\cot\left(\frac{\pi}{2 p^*}\right)=
    \begin{cases}
        \tan\left(\frac{\pi}{2p}\right), & 1<p\le 2 \\
        \cot\left(\frac{\pi}{2p}\right), & 2\leq p<\infty.  
    \end{cases}
\end{align} 
 When $p$ is of the form $2^j$ or ${2^j/(2^j-1)}$, $j \in \N$, the result had been known for several years. For such $p$, it is shown in \cite{Laeng} that \eqref{pichcole} holds for 
 ${H_{\dis}}$.   For further history and references related to this problems, 
 see \cite{Laeng, BanKwa1, HarLitPol, CarSam, Tor, ADP20, HunMukWhe, Pet2, Pet3}.  In \cite{BanKwa}, it is proved that for all $1<p<\infty$, 
\begin{equation}\label{BanKua}
 \|H_{\dis}\|_{\ell^p\to \ell^p}\leq \cot(\pi / (2 p^*))
\end{equation} 
and hence, together with the first inequality in \eqref{Riesz28},  
\begin{align}\label{sharpH-dis}
\|H_{\dis}\|_{\ell^p\to \ell^p}=\|H\|_{L^p\to L^p}.
\end{align} 

The proof of \eqref{BanKua} rests  on a modification of the Gundy--Varopoulos techniques using occupation time formulas as in \cite{Ban86} for the Doob $h$-process where the harmonic function $h$ corresponds to the periodic Poisson kernel. This construction leads to yet a third operator $T_{\bH}$, which we call the \textit{probabilistic discrete Hilbert transform.} 
\begin{align}\label{ProbH}
T_{\bH}f(n)=\sum_{m\in\Z\setminus\{0\}} \bK_{\bH}(m)f(n-m), 
\end{align}
where  
\begin{align}\label{kerneld-1} 
    \bK_{\bH}(m)= \frac{1}{\pi m}\left(1+\int_0^\infty\frac{2y^3}{(y^2+\pi^2 m^2)\sinh^2(y)}\, dy\right).
    \end{align}

Although, as \eqref{kerneld-1} clearly shows, this  construction does not lead to an exact representation of the discrete Hilbert transform, unlike the classical case for the continuous Hilbert transform on $\R$ and the Gundy--Varopoulos extension for Riesz transforms on $\R^d$,  the martingale inequalities from \cite{BanWang} show that for $1<p<\infty$,
\begin{align}\label{ProbH1}
 \|T_{\bH}\|_{\ell^p\to \ell^p}\leq \cot(\pi / (2 p^*)). 
\end{align} 
Furthermore, it is shown in \cite{BanKwa} that there exists a probability kernel $\cP$ on $\Z$ such that 
\begin{align}\label{ConvFormula}
H_{\dis}f(n)=(T_{\bH}\ast \cP)f(n).
\end{align}  
Here $\ast$  denotes the convolution. From this,  
$$
\|H_{\dis}\|_{\ell^p\to \ell^p}\leq  \|T_{\bH}\|_{\ell^p\to \ell^p}\leq \cot(\pi / (2 p^*))
$$
and \eqref{sharpH-dis} follows.   

In this paper we give a new proof of \eqref{ConvFormula}; see Section \ref{FMPDH}, Theorem \ref{BanKwaMain}. The proof is based on  Fourier transform techniques. 

Before discussing the higher dimensional case, we point out that there are other variants of the discrete Hilbert transform on $\Z$. These include the one introduced by Titchmarsh in \cite{Titc26},  called the \textit{Riesz--Titchmarsh operator},  for which the conjecture that their $\ell^p$ norms coincide with that of $H$ remains open.  For a detailed discussion of the different variants of discrete Hilbert transforms and a proof of the conjecture for  $p = 2 n$ or its conjugate $p = \frac{2 n}{2 n - 1}$, $n \in \N$, for the Riesz--Titchmarsh operator, we refer the reader to \cite{BanKwa1}.  While the argument in \cite{BanKwa1} does not use probability, it does use the fact that \eqref{sharpH-dis} holds for  all $1<p<\infty$.  

\subsection{The probabilistic discrete Riesz transforms on $\Z^d$.}\label{SubPDR} In  their celebrated paper \cite{CZ} Calder\'on and Zygmund  remark that Riesz's argument for the second inequality in \eqref{Riesz28} applies to show that  the boundedness of singular integrals on $L^p(\R^d)$ implies the boundedness of their discrete versions defined on $\Z^d$ in the same manner as the discrete Hilbert transform.    More precisely, let 
$$
Tf(z)=p.v.\int_{\R^d}K(z)f(x-z)dz, \quad f:\R^d\to \R 
$$
be a Calder\'on-Zygmund singular integral ($K$ as in \eqref{eq:ConK-def}). Then the discrete operator $T_{\dis}$ defined by 
\begin{align}\label{disCZ} 
T_{\dis}(f)(n)=\sum_{m\in\Z^d\setminus\{0\}}K(m)f(n-m),\quad f: \Z^d \to \R 
\end{align} 
is bounded on $\ell^p(\Z^d)$ and 
$$\|T_{\dis}\|_{\ell^p\to \ell^p}\leq C\|T\|_{L^p\to L^p},
$$  for some constant $C$; see  \cite[pg.~138]{CZ}. In particular, if we consider the Riesz transforms on $\R^d$ 
\begin{equation}\label{classicalCRT-1}
	R^{(k)} f(x) =p.v.\int_{\R^d} K^{(k)}(z)f(x-z)\,dz, \quad k=1,2,\dots d,
\end{equation}
with 
\begin{equation}\label{RieszK} 
    K^{(k)}(z) =c_d \frac{z_{k}}{|z|^{d+1}},  \quad c_d=\frac{\Gamma(\frac{d+1}{2})}{\pi^{\frac{d+1}{2}}}
\end{equation}
and their discrete versions 
\begin{align}\label{DRT-1} 
    R^{(k)}_{\dis}f(n)=c_d\sum_{m\in\Z^d\setminus\{0\}} \frac{m_{k}}{|m|^{d+1}}f(n-m), 
\end{align} 
it holds that for all $1<p<\infty$,  
\begin{align*}
\|R^{(k)}_{\dis}\|_{\ell^p\to \ell^p}\leq C\|R^{(k)}\|_{L^p\to L^p}, \quad 1<p<\infty. 
\end{align*} 
for some constant $C$.

These operators are the classical extensions of the Hilbert transform to several dimensions. It is shown in \cite{IwaMar} using the method or rotations, and in  \cite{BanWang} using martingale inequalities, that  
\begin{align}\label{SCRiesz} 
 \|R^{(k)}\|_{L^p\to L^p}=\|H\|_{L^p\to L^p}=\cot(\pi / (2 p^*))
 \end{align} 
From  Proposition \ref{thm:RCZ} below it follows that 
\begin{align}\label{prop1}
\|R^{(k)}_{\dis} \|_{\ell^p\to \ell^p}\leq \|R^{(k)}\|_{L^p\to L^p}+C_d=\cot(\pi / (2 p^*))+C_d, 
\end{align}
where $C_d$ is a constant depending only on $d$. Furthermore, in Theorem  \ref{homeker} a sharp lower bound for singular integrals with homogeneous kernels is proved from which it follows that 
\begin{align}\label{prop2}
\cot(\pi / (2 p^*))=\|R^{(k)}\|_{L^p\to L^p}\leq \|R^{(k)}_{\dis} \|_{\ell^p\to \ell^p}.
\end{align}
This  gives  
\begin{align}\label{SharpDiscRiesz} 
\cot(\pi / (2 p^*))\leq \|R^{(k)}_{\dis} \|_{\ell^p\to \ell^p}
\leq \cot(\pi / (2 p^*))+C_d
\end{align}

Given the success of  probabilistic techniques in deriving sharp estimates for Riesz transforms in general geometric settings, as discussed above, together with the interest on discrete analogues of many classical operators in harmonic analysis as studied in \cite{MSW, Pie1, Pie2, Pie3, Pie 4, Pie5, Pierce, SW2000, SteWai99} (and many other references contained therein), together with the results in \cite{BanKwa},  the following questions naturally arise:  
\begin{enumerate}
\item 
Can the construction of the probabilistic operators in \cite{BanKwa} be carried out in higher dimension to obtain a collection of operators on $\Z^{d}$, $d> 1$, which are closely related to the classical Riesz transforms in \eqref{DRT-1}, that have $\ell^p(\Z^d)$-norms independent of $d$ and that could be used to show that $\|R_{\dis}^{(k)}\|_{\ell^p\to\ell^p}$ is independent of $d$ or even that it equals $\cot(\pi / (2 p^*))$, as suggested by \eqref{SharpDiscRiesz}?  
\item 
Do the probabilistic  discrete operators arise as the discretization of Calder\'on-Zygmund singular integrals on $\R^d$--as is the case of \eqref{DRT-1}?   
\item Given the important role of the method of rotations in obtaining sharp or near sharp estimates for Riesz transforms and related singular integrals on  $\R^d$ (see for example \cite{IwaMar}), it is natural to ask:  is there a discrete ``method of rotations'' that would reduce the $\ell^p(\Z^d)$ boundedness of $R_{{\dis}
}^{(k)}$ to that of $H_{{\dis}}(\Z)$? \end{enumerate}

The goal of this paper is to address these questions as we now explain.  The periodic Poisson kernels $ h(x,y)$, $x\in \R^d, y>0$, is obtained by summing the translations of  the classical Poisson kernel $p_n(x,y)=p(x-n,y)$ over $n\in \Z^d$ as in \eqref{doob-h}.  Using this positive harmonic function we construct a Doob $h$-process in the upper half-space which exits only on the lattice $\Z^d$. Taking conditional expectations of martingale transforms of stochastic integrals built with this process, we construct a large collection of operators on $\Z^{d}$, denoted by $T_A$, where $A$ is any  $(d+1)\times (d+1)$ matrix with variable entries and finite quadratic norm $\|A\|$ as defined in \eqref{normofA}.   Applying the martingale inequalities in \cite{Bur84} and \cite{BanWang}, leads to the operator bounds $\|T_A\|_{\ell^p\to \ell^p}\leq \|A\|(p^*-1)$ for all $A$ and $ \|T_A\|_{\ell^p\to \ell^p}\leq \|A\|\cot(\pi/(2 p^*))$,  if in addition $A$ has the property that $Av$ is orthogonal to $v$, for all $v\in \R^{d+1}$.  The kernels for the  $T_A$'s, (not necessarily of convolution type) are computed in Theorem \ref{thm:TaIntRep}.    
Of particular interest in this construction is the case of the sequence of matrices  $\bH^{(k)}=(a^{(k)}_{ij})$, $k=1, 2, \dots d$, given by 
\begin{equation}\label{HRiesz}
	a^{(k)}_{ij}=
	\begin{cases}
	    -1, & i=k, j=d+1\\
		1, & i=d+1, j=k\\
		0, & \text{otherwise}.  
	\end{cases}
\end{equation}
These matrices have the orthogonality property, $\bH^{(k)}v\cdot v=0$ for all $v\in \R^{d+1}$, and $\|\bH^{(k)}\|\leq 1$.  Hence the corresponding operators $T_{\bH^{(k)}}$ satisfy 
\begin{align}\label{THUB}
\|T_{\bH^{(k)}}\|_{\ell^p\to\ell^p}\leq \cot(\pi / (2 p^*)), \quad k=1, 2, \dots, d. 
\end{align} 
 In this case the kernels (see \eqref{byparts} in the proof of Theorem \ref{DisProbVSDiscRies}) are given by 
\begin{eqnarray}\label{ProbDisR} 
\bK_{\bH^{(k)}}(m)
    =-4\int_{\R^{d}}\int_{0}^{\infty}\frac{1}{h(x, y)}\parder{{p_0}
}{x_k}(x, y)\parder{}{y}(yp_m)(x, y)\,dydx,
 \end{eqnarray}
and 
\begin{equation}
    T_{\bH^{(k)}}(f)(n)=\sum_{m\in\Z^d\setminus\{0\}}\bK_{\bH^{(k)}}(n-m)f(m).
\end{equation}

The choice of the $\bH^{(k)}$'s  is motivated  by the fact that using the  standard Brownian motion in $\R^{d+1}_{+}$ killed upon exiting on $\R^d$,  these are the matrices used for the probabilistic representation of the classical Riesz transforms. This was originally done by Gundy and Varopoulos \cite{GV79} (see Section \ref{Conjecture} below) and subsequently by many other authors in different geometric settings.  Indeed, the Gundy-Varopoulos construction  gives operators whose kernels are as in \eqref{ProbDisR} with $h$ replaced by 1 and $n\in \Z^d$  by $z\in \R^d$. In fact,  
\begin{align}
-4\int_{\R^{d}}\int_{0}^{\infty}\parder{p_0(x, y)}{x_k}\parder{}{y}(yp_z(x, y))\,dydx=\frac{c_d z_k}{|z|^{d+1}}, \quad z\in \R^d\setminus\{0\}, 
\end{align} 
where as before $p_z(x, y)=p(x-z, y)$ is the Poisson kernel. For this computation, see \eqref{h=1case} below.

We call the operators $T_{\bH^{(k)}}$, $k=1, 2, \dots, d$, the  \textit{probabilistic discrete  Riesz transforms}. When $d=1$ there is only one matrix $\bH^{(1)}$ that we simply denote by $\bH$ and the corresponding operator  $T_{\bH}$  is the \textit{probabilistic discrete Hilbert transform  \eqref{ProbH}} studied in \cite{BanKwa}. In this case the kernel can be further simplified  to \eqref{kerneld-1}. This simplification is due to the fact that when $d=1$ we have the simple closed formula for periodic Poisson kernel  $h$ given by  \eqref{Explicitd=1}. This gives a very nice and quite simple formula for $1/h$ (see \eqref{hinverse}) from which many explicit computations can be done. Such a closed formula is not available for $d>1$.  Instead, various properties of the function $h(x, y)$ (as in Lemma \ref{lem:hlimit}) are used for many of the computations in this paper. 

 While at present there is no analogue to \eqref{ConvFormula} for $d>1$ (see Question \ref{convZd} below), after a scaling and limiting arguments (Lemmas \ref{lem:DiscreteAsymptotics} and  \ref{lem:DiscreteApproximation}), the probabilistic Riesz transforms are closely related to the classical Riesz transforms.  This relationship can be exploited to prove the lower bound 
\begin{align}\label{THLB} 
\cot(\pi / (2 p^*))\leq \|T_{\bH^{(k)}}\|_{\ell^p\to\ell^p}. 
\end{align} 
This together with \eqref{THUB} shows that the $\ell^p$-norm of $T_{\bH^{(k)}}$ are not only independent of the dimension but equal the $L^p$-norm of $R^{(k)}$, see Theorem  \ref{bestpro-dis-riesz}.  

Given our discussion above the following kernels on $\R$ and $\R^d$, $d>1$, respectively, naturally arise:  
\begin{align}
\bK_{\bH}(z)=\frac{1}{\pi z}\left(1+\int_0^\infty\frac{2y^3}{(y^2+\pi^2 z^2)\sinh^2(y)}\, dy\right)\,\ind_{\{|z|\geq 1\}}(z)+\frac{1}{\pi z}\ind_{\{|z|<1\}}(z),\quad z\in \R\setminus\{0\}, 
\end{align}
\begin{align}
    \bK_{\bH^{(k)}}(z)
	&=\left(-4\int_{\R^{d}}\int_{0}^{\infty}\frac{1}{h(x, y)}\parder{{p_0} }{x_k}\parder{}{y}(yp_z)\,dydx\right)\ind_{\{|z|\geq 1\}}(z)\\
	&+c_d\frac{z_k}{|z|^{d+1}}\ind_{\{|z|<1\}}(z), \quad  z\in \R^d\setminus\{0\}, \,\,\, d>1, \,\,\, k=1, 2, \dots,  d.\nonumber
\end{align}
It is proved in Section \ref{ProbContRiesz}, Theorem \ref{hormander},  that these kernels satisfy the H\"ormader condition.  The principal value convolution operators  $Tf(z)=\bK_{\bH}\ast f (z)$ and $T^{(k)}f(z)=\bK_{\bH^{(k)}}\ast f (z)$ are bounded on $L^p(\R)$ and $L^p(\R^d)$, $d>1$, $1<p<\infty$.   Furthermore,  it is proved in Section \ref{ProbContRiesz} (see Theorem \ref{LpCont}) that 
\begin{align*}
\cot(\pi/(2p^\ast))\leq \|T\|_{L^p\to L^p} &\leq \cot(\tfrac{\pi}{2p^{\ast}})+0.09956,
\end{align*}
and 
\begin{align*}
\cot(\pi/(2p^\ast))\leq \|T^{(k)}\|_{L^p\to L^p} & \leq\cot(\tfrac{\pi}{2p^{\ast}})+ C_d
\end{align*} 
 where $C_d$ depends only on $d$.  
 
The probabilistic discrete Hilbert and Riesz transforms on the lattice $\Z^d$ are the discretization of these Calder\'on-Zygmund operators on $\R^d$. We call the operators $T$ and $T^{(k)}$ the {\it probabilistic continuous Hilbert and Riesz transforms}, respectively.   We conjecture (Conjecture \ref{conticon}) that their  $L^p$-norms are not only independent of the dimension but that in fact they are equal to $\cot(\pi/(2p^\ast))$. 
 
 \begin{remark} It is important to note that these discrete singular integrals  \'a la Hilbert, Riesz and  Calder\'on and Zygmund,  do not arise as ``genuine'' Riesz transforms of semigroups associated with discrete or semi-discrete Laplacians for which many results exist and to which the ``usual'' Gundy--Varopoulos construction as outlined in Section \ref{Conjecture} applies. For some of these applications we refer the reader to  \cite{Pet2, ArcDomPet2, ADP20}.
\end{remark}

Concerning  a ``method of rotations'' in the discrete setting, Section~\ref{MethodRotations} presents an analogue motivated by differently defined discrete Riesz transforms $\drt^{(k)}$ on $\Z^d$.   It gives a representation of these discrete Riesz transforms in terms of the discrete Hilbert transform $H_{\dis}$ from which  the sharp estimates for their norms is obtained, Theorem \ref{thm:rot:higher}.

We summarize the main results of this paper in the following  Theorem.  

\begin{theorem}\label{summary} 
Let $R^{(k)}$ be the classical Riesz transforms in \eqref{classicalCRT-1}, $R_{\dis}^{(k)}$ the discrete Riesz transforms in \eqref{DRT-1}, $T_{\bH^{(k)}}$ the probabilistic discrete Riesz transforms in \eqref{ProbDisR} and $\drt^{(k)}$ the Riesz transforms obtained by the method of rotations in \eqref{MethodRot}. Fix $1<p<\infty$.   
\begin{itemize}
\item[(i)] For 
 $d\geq 2$ and $k=1, \dots, d$, 
\begin{align}
   \|T_{\bH^{(k)}}\|_{\ell^p\to \ell^p}=\|\drt^{(k)}\|_{\ell^p\to \ell^p}= \|R^{(k)}\|_{L^p\to L^p}=\cot(\tfrac{\pi}{2p^*}),  
\end{align}
and 
\begin{align}
   \cot(\tfrac{\pi}{2p^*})\leq \|R_{\dis}^{(k)}\|_{\ell^p\to \ell^p}\leq \cot(\tfrac{\pi}{2p^{\ast}})+C_d. 
  \end{align}
   The constant $C_d$ depends only on $d$.
    \item[(ii)] The singular integral operators $T$ and $T^{(k)}$ on $\R$ and $\R^d$, $d>1$,  with kernels $\bK_{\bH}$ and $\bK_{\bH^{(k)}}$, satisfy the H\"ormader condition. Further, 
 \begin{align}\label{normT}
\cot(\pi/(2p^\ast))\leq \|T\|_{L^p\to L^p} &\leq \cot(\pi/(2p^\ast))+0.09956
\end{align} 
and 
\begin{align}\label{normTk}
\cot(\pi/(2p^\ast))\& \leq \|T^{(k)}\|_{L^p\to L^p} & \leq\cot(\tfrac{\pi}{2p^{\ast}})+C_d
\end{align} 
 The constant $C_d$ depends only on $d$.

\item[(iii)]  
 When $d=1$, 
 \begin{align}
 R^{(1)}=H,\quad  R_{\dis}^{(1)}=\drt^{(1)}=H_{\dis}, \quad T_{\bH^{(1)}}=T_{\bH}. 
 \end{align}  The $p$-norm of these operators all equal $\cot(\tfrac{\pi}{2p^*})$.
  
\end{itemize} 
\end{theorem} 

\begin{remark}
We note that (iii) is the  content of \cite{BanKwa}.  We state it here to more clearly give a global picture for all $d\geq 1$. The contribution of this paper to (iii) is the new proof of the key identity \eqref{ConvFormula} based on Fourier transform techniques which has the potential for extending to higher dimensions.  (See also Remark \ref{byparts} for a derivation of \eqref{kerneld-1}.)
This motivates Question \ref{convZd}. If this question is  answered in the affirmative it would resolve Conjecture \ref{SharpRiesz}.  
\end{remark}

It is perhaps interesting  to end this introduction with the following insightful comment from Stein \cite{SteSome} who ends his paper with the remark: \begin{quote} {\it ``The above results raise the following general question. Can one find an appropriate infinite-dimensional formulation of (that part of) harmonic analysis in $\R^n$ which displays in a natural way the above uniformity in $n$? A related question is to study the ``limit as $n\to \infty$'' of the above results, insofar as such limits may have a meaning. One might guess that a further understanding of these questions would involve, among other things, notions from probability theory: i.e. Brownian motion and possibly some variant of the central limit theorem.''}\end{quote}   

The now vast literature on probabilistic techniques for problems related to this topic, some of it cited in this paper, proves that his ``guess''
was indeed correct. 
   
\section{Organization of the paper}
\begin{itemize}[leftmargin=*]
\item Section \ref{PerPK} introduces the \textit{periodic Poisson kernel} $h$ on $\R^d$, $d\geq 1$, from which we will define the Doob $h$-process used throughout the paper and derive some of its basic properties.  
\vskip.05in 

\item Section \ref{DisHarExt} defines the $h$-harmonic extension of the function $f:\Z^d\to \R$ and discusses a connection to an interesting problem of Magyar, Stein, and Wainger in ~\cite{MSW}, see Remark \ref{MagSteWeRMK1}.
 \vskip.05in 

\item Section \ref{general-d} defines the Doob $h$-process associated with the function $h$, the martingale transforms and recalls the relevant martingale inequalities. 

\item Section \ref{Martranpro} defines the projection operators, denoted by $T_{A}$, as the conditional expectations of martingale transforms. {Theorems \ref{thm:TaIntRep} and \ref{Proj}} compute their kernels and prove their boundedness on $\ell^p(\Z^d)$ with the same bounds as those in the martingale inequalities.
\vskip.05in 

\item Section \ref{CZO} considers the discrete Calder\'on--Zygmund operators given by \eqref{disCZ} and proves their boundedness properties.  This is done in Proposition \ref{thm:RCZ} and Theorem \ref{homeker}. From these, (ii) of Theorem \ref{summary} follows. {Conjecture \ref{SharpRiesz}}, and the weaker {Problem \ref{WeakerProb}}, on the norm of the discrete Riesz transforms are formulated in this section. 
\vskip.05in 

\item Section \ref{PDR} shows that the  $\ell^p(\Z^d)$-norms of the  probabilistic discrete Riesz transforms on $\Z^d$ are the same as the $L^p(\R^d)$-norms of the classical Riesz transforms.  This is done in Theorem \ref{bestpro-dis-riesz}. From this it follows that the $p$-norms of the classical Riesz transforms on $L^p(\R^d)$, the classical Hilbert transform on $L^p(\R)$, the discrete Hilbert transform on $\ell^p(\Z)$, and probabilistic discrete Riesz transforms on $\ell^p(\Z^d)$ are all equal to $\cot(\pi / (2 p^*))$, verifying  (i) in Theorem \ref{summary} for all the the operators with the exception of $\drt^{(k)}$. 
\vskip.05in

\item Section \ref{FMPDH} computes the Fourier transform of the probabilistic discrete Hilbert transform. This allows for a new proof of the ``Key Lemma 1.3'' in \cite{BanKwa} which shows that the discrete Hilbert transform $H_{\dis}$ is the convolution of the probabilistic discrete Hilbert transform with a probability kernel, see Theorem \ref{BanKwaMain}. The proof here is  based on the Fourier transform techniques. The natural question for $\Z^d$, $d\geq 2$ is stated at the end of this section, see  Question \ref{convZd}. 
\vskip.05in 

\item Section \ref{ProbContRiesz} shows that replacing the discrete variable $n\in \Z^d$ by the continuous variable $z\in \R^d$ in the kernel for the probabilistic discrete Riesz transform{s}, and after a modification near zero that does  not affect the discrete operator, give singular integral operators  satisfying the H\"ormader condition and are bounded on $L^p(\R^d)$, $1<p<\infty$ with their norms satisfying \eqref{normT} and \eqref{normTk}.  See Theorems \ref{almostCZ}, \ref{LpCont}, \ref{hormander}, and Corollary \ref{ContDecom}.
\vskip.05in 

\item Section \ref{MethodRotations} discusses a {``method of rotation''} by constructing certain discrete Riesz transforms on $\Z^d$ motivated by the classical ones and verifying that these Riesz transforms have the same norms as the discrete Hilbert transform $H_{\dis}$ and the probabilistic discrete Riesz transforms $T_{\bH^{(k)}}$, see Theorem \ref{thm:rot:higher}. 
\vskip.05in

\item Section \ref{Numerical} presents some numerical calculations comparing the relative sizes of the kernels for the discrete Riesz transforms, the probabilistic discrete Riesz transforms and the discrete Riesz transforms constructed in the method of rotations.  
\end{itemize}  

\section{Notation} 
The Fourier transform of a function $f$ on $\R^d$ is denoted by $\wh{f}$, where 
$$ \wh{f}(\xi)=\int_{\R^{d}}f(x)e^{-2\pi i x\cdot \xi}dx \quad \text{for }\xi \in\R^d. $$ 
For a function $f:\Z^d\to\R$, the Fourier transform is denoted by $\cF(f)$, where
\begin{align*}
    \cF(f)(\xi) = \sum_{n\in\Z^d}f(n)e^{-2\pi i n\cdot \xi} \quad \text{for } \xi\in Q:=[-\tfrac12, \tfrac12)^d.  
\end{align*}
Here, $Q$ is often called the fundamental cube. 

The standard notations $\|f\|_{L^p}$ and $\|f\|_{\ell^p}$ are used for the $p$-norm of functions in $L^p(\R^d)$ and $\ell^p(\Z^d)$, respectively. $\|T\|_{L^p\to L^p}$ will denote the operator norm of $T:L^p(\R^d) \to L^p(\R^d)$, and  similarly $\|T\|_{\ell^p\to \ell^p}$ for the  operator norm of $T:\ell^p(\Z^d)\to \ell^p(\Z^d)$. 

The gradient and Laplacian of functions $u(x, y)$ on the upper half-space $\R^d \times \R_+=\{(x, y): x\in \R^d, y>0\}$ are denoted by 
$$ \nabla u 
=\left(\frac{\partial u}{\partial x_1}, \cdots, \frac{\partial u}{\partial x_d}, \frac{\partial u}{\partial y}\right)
=\left(\nabla_{x}u, \,\, \parder{u}{y}\right) $$
and 
$$ \Delta u 
=\sum_{i=1}^d \frac{\partial^2 u}{\partial x{_i}^2}+\frac{\partial^2 u}{\partial y^2}=\Delta_{x}u+\parder{u^{2}}{y^{2}}, $$ 
respectively. By abuse of notation, for $f:\R^d\to \R$ we will still use $\Delta f=\sum_{i=1}^d \frac{\partial^2 f}{\partial x{_i}^2}$ to denote its Laplacian on $\R^d$. 

Throughout the paper, 
$C_1, C_2, \ldots$ or $C_d$ will denote constants that depend only on $d$ and whose value may change from line to line. 

\section{Preliminaries}
\subsection{The periodic Poisson kernel}\label{PerPK}
Let $d\geq 1$. The Poisson kernel for the upper half-space $\R^d \times \R_+$ is given by
\begin{align}\label{PoissonK}
    p(x,y) =\frac{c_{d}y}{(|x|^{2}+y^{2})^{\frac{d+1}{2}}}, \quad (x, y)\in\R^{d}\times \R_{+}, \quad c_d=\Gamma{\biggl(}\frac{d+1}{2}{\biggr)}\pi^{-\frac{d+1}{2}}.
\end{align}
For $x,z\in\R^{d}$ and $y\in \R_{+}$, we set $p_z(x,y)=p(x-z,y)$. Since $\Delta p_{z}(x,y)=0$ and $p_{z}(x,y)> 0$ for all $(x,y)\in\R^{d}\times \R_{+}$ and $z\in \R^d$, we see that the function $h$ defined by 
\begin{equation}\label{doob-h}
    h(x,y)=\sum_{n\in \Z^d}p_{n}(x,y)
\end{equation}
is also positive and harmonic. In addition, it is periodic in $x$ in the sense that $h(x+m,y)=h(x,y)$ for all $m\in\Z^d$. We call the function $h(x, y)$ the periodic Poisson kernel. The following properties of $h(x, y)$ will be used frequently in the sequel. 

\begin{lemma}\label{lem:hlimit}
We have $\lim_{y\to\infty}h(x,y)=1$ uniformly in $x\in\R^{d}$. In particular, for each $y_0>0$, there exist constants $C_1, C_2>0$ such that $C_1\leq h(x,y)\leq C_2$, for all $x\in\R^{d}$ and $y\geq y_0$. 
\end{lemma}
\begin{proof}
Recall that  $Q={[-\frac12, \frac12)^d}
.$ For $x\in\R^{d}$ and $y>0$, we have
\begin{equation*}
	1 =\int_{\R^{d}}p(x-z,y)\,dz =\sum_{n\in \Z^d}\int_{n+Q}p(x-z,y)\,dz =\sum_{n\in \Z^d}\int_{Q}p_{n}(x-z,y)\,dz.
\end{equation*}
We will estimate the quantity
\begin{equation*}
	p_{n}(x,y) - \int_{Q}p_{n}(x-z,y)\,dz = \int_{Q}(p_{n}(x,y) - p_{n}(x-z,y))\,dz .
\end{equation*}
All constants $C_1, C_2, \ldots$ below are positive and depend only on $d$. Observe that we have
\begin{equation*}
	|\nabla_x p(x, y)| \leq C_1 \frac{y}{(|x|^2 + y^2)^{\frac{d}{2} + 1}}
\end{equation*}
for all $(x, y) \in \R^{d} \times \R_{+}$. Furthermore, if $x \in \R^{d}$, $y \geq 1$, $z \in Q$ and $n \in \Z^{d}$, then
\begin{equation*}
	|x-n-z|^2 + y^2 \geq C_2 (|x - n|^2 + y^2) .
\end{equation*}
It follows that
\begin{align*}
	|\nabla_x p_n(x-z, y)|
	& \leq C_1 \, \frac{y}{(|x-n-z|^2 + y^2)^{\frac{d}{2} + 1}} \\
	& \leq C_3 \, \frac{y}{(|x - n|^2 + y^2)^{\frac{d}{2} + 1}} \\
	& \leq C_4 \, \frac{p_n(x, y)}{(|x - n|^2 + y^2)^{\frac{1}{2}}} \leq C_4 \, \frac{p_n(x, y)}{y} \, .
\end{align*}
Thus,
\begin{equation*}
    |p_{n}(x,y) -p_{n}(x-z,y)| \leq C_4 \, \frac{p_n(x, y)}{y} \times |z| \leq C_5 \, \frac{p_n(x, y)}{y} .
\end{equation*}
We conclude that
\begin{align*}
	|h(x,y)-1|
	& = \Abs{\sum_{n\in \Z^d} \int_{Q}(p_{n}(x,y)dz-p_{n}(x-z,y)) dz} \\
	& \leq \sum_{n\in \Z^d} \int_{Q}|p_{n}(x,y)dz-p_{n}(x-z,y)| dz \\
	& \leq \sum_{n\in \Z^d} \int_{Q} C_5 \, \frac{p_n(x, y)}{y} \, dz = \frac{C_5 h(x,y)}{y} \, .
\end{align*}
This proves the first statement of the lemma. The second assertion (ii) follows from (i) and the fact that $h(x,y)$ is positive, continuous, and periodic in $x$.
\end{proof}

Clearly, $h(x, y) \geq p(x, y)$, and thus for some constant $C_1$ depending only on $d$ we have $h(x, y) \geq C_1 y$ whenever $x \in Q$ and $y \in (0, 1)$. Since $h(x, y)$ is periodic in $x$ with period $1$, the same estimate holds for all $x \in \R^{d}$, and by combining this inequality with Lemma~\ref{lem:hlimit}, we find that
\begin{equation}\label{eq:hupperbound1}
    h(x, y) \geq C_2 \min\{1, y\}
\end{equation}
for all $(x, y) \in \R^{d} \times \R_{+}$. Similarly, we have $h(x, y) \geq p(x, y) \geq C_3 \frac{y}{\sqrt{x^2 + y^2}}$ when $x^2 + y^2 < 1$ and $y \in (0, 1)$. For all other $(x, y)$ the last estimate is weaker than~\eqref{eq:hupperbound1} (up to a constant factor), so we conclude that
\begin{equation}\label{eq:hupperbound2}
    h(x, y) \geq C_4 \, \frac{y}{\sqrt{x^2 + y^2}}
\end{equation}
holds for all $(x, y) \in \R^{d} \times \R_{+}$.

We recall the Poisson summation formula.  
\begin{proposition}[\cite{Graf}*{Theorem 3.1.17}]
Suppose that $f,\wh{f}\in L^{1}(\R^{d})$ and 
\begin{equation*}
    |f(x)|+|\wh{f}(x)|\leq C(1+|x|)^{-n-\delta}
\end{equation*}
for some $C,\delta>0$. Then $f$ and $\wh{f}$ are continuous and 
\begin{equation}\label{PoissSum}
    \sum_{n\in\Z^{d}}f(x+n)=\sum_{n\in\Z^{d}}\wh{f}(n)e^{2\pi i n\cdot x}
\end{equation}
for all $x\in\R^{d}$.
\end{proposition}

Applying to the function $f(x)=p(x,y)$ for which $\wh{f}(n)=e^{-2\pi |n|y}$, we obtain that 
\begin{align}\label{FourierInv} 
	h(x,y) = \sum_{n\in\Z^{d}}p_0(x+n,y) = \sum_{n\in\Z^{d}}e^{-2\pi |n|y}e^{2\pi in\cdot x}. 
\end{align}
From the fact that the Poisson kernel for the unit disc $\cP_{r}(\theta)$ is given by 
\begin{align*}
	\cP_{r}(\theta) = \sum_{n\in\Z}r^{|n|}e^{in\theta} = \frac{1-r^{2}}{1+r^{2}-2r\cos\theta},\quad 0\leq r\leq 1,\quad  0\leq \theta\leq 2\pi, 
\end{align*}
we see that for $d=1$, 
\begin{align}\label{Explicitd=1}
	h(x,y) = \frac{1}{2\pi}\cP_{e^{-2\pi y}}(2\pi x) =\frac{\sinh(2\pi y)}{\cosh(2\pi y)-\cos (2\pi x)}.
\end{align}
This explicit expression for $h$ when $d=1$ is well known, see for example \cite{IwaKow}*{pg.70}. It was derived in \cite{BanKwa}*{Lemma 3.1} by a different argument and used there for many calculations. In particular, for $d=1$, such a formula permits explicit calculations for various quantities involving  the function $\frac{1}{h(x, y)}$, see \cite{BanKwa}. In Section \ref{FMPDH} we will use this in the calculation of the Fourier transform of the kernel $\bK_{\bH}$ for the probabilistic discrete Hilbert transform. For $d\geq 2$, while we can express $h(x, y)$ in various other forms besides \eqref{doob-h} and \eqref{FourierInv}, it does not seem possible to write such a convenient closed formula that will facilitate calculations with $\frac{1}{h(x, y)}$ in a similar manner.  

\subsection{$h$-harmonic extension}\label{DisHarExt}
Let $f:\Z^d \to\R$ be a function of compact support, that is, $f(n)=0$ all but finitely many $n\in \Z^d$. Define  
\begin{align}\label{eq:def_harmonic}
	u_{f}(x,y)=\sum_{n\in \Z^d}f(n)\frac{p_{n}(x,y)}{h(x,y)}.
\end{align}
Note that $u_{f}$ is $h$-harmonic.  That is, $\Delta(hu_f)=0$. Equivalently, $u_f(x, y)$ is harmonic in the upper half-space relative to the operator 
\begin{align*}
	\frac{1}{2}\Delta+\frac{\nabla h(x,y)\cdot \nabla}{h(x,y)}.
\end{align*}
The following proposition provides information on the boundary values of $u_f$.

\begin{proposition}
    For each $x\in\R^d$, $p(x,y)/h(x,y)$  converges as $y\to 0$. Let $\Psi(x)$ be the limit, and $f\in\ell^p(\Z^d)$ be compactly supported. Define $f_\ext(x) := \sum_{n\in\Z^d} f(n)\Psi(x-n)$. Then, $f_\ext(n) = f(n)$ for all $n\in\Z^d$ and $\|f_\ext\|_{L^p}\leq \|f\|_{\ell^p}$. 
\end{proposition} 
\begin{proof}
Suppose $x=0$. Note that $p_n(0,y)/h(0,y)\le 1$ and 
\begin{align*}
    h(0,y) 
    &= p(0,y) + \sum_{n\neq 0} p(n,y) = p(0,y) + \sum_{n\neq 0} \frac{c_d y}{\left(|n|^2+y^2\right)^{\frac{d+1}{2}}}\\
    & \le p(0,y) + c_d y \sum_{n\neq 0} |n|^{-(d+1)}.
\end{align*}
Since the sum on the right hand side is finite, we have
\begin{align*}
    \frac{1}{1+C_dy^{d+1}} \le \frac{p(0,y)}{h(0,y)}\le 1.
\end{align*}
Thus, if $x=0$, the limit exists and $\Psi(0)=1$. 

Suppose $x=n\in\Z^d\setminus\{0\}$. Then, $\lim_{y\to 0}p(n,y) = 0$ and 
\begin{align*}
    \frac{1}{h(n,y)} \le \frac{1}{p(0,y)} = \frac{1}{c_d}y^d\to 0.
\end{align*}
Thus, the limit exists and $\Psi(n)=0$. For $x\in\R^d\setminus\Z^d$, we have $\lim_{y\to 0} p(x,y)/y = c_d |x|^{-d-1}$ and
\begin{align*}
    \lim_{y\to 0}\frac{h(x,y)}{y} =c_d\sum_{n\in\Z^d}|x-n|^{-(d+1)}.
\end{align*}
Since the sum is finite for each $x\in\R^d\setminus\Z^d$, we conclude that $\Psi(x)$ is well-defined. Note that
\begin{align*}
    \int_{\R^d}\frac{p(x,y)}{h(x,y)}\, dx = \sum_{n\in\Z^d}\int_Q \frac{p_n(x,y)}{h(x,y)}\, dx = 1
\end{align*}
and
\begin{align*}
    \sum_{n\in\Z^d}\frac{p(x-n,y)}{h(x-n,y)} = \sum_{n\in\Z^d}\frac{p_n(x,y)}{h(x,y)} = 1.  
\end{align*}
By Fatou's lemma, we have
\begin{align*}
    \int_{\R^d}\Psi(x)\, dx \leq 1, \qquad \sum_{n\in\Z^d}\Psi(x-n) \le 1.
\end{align*}
Thus, it follows from H\"older's inequality that
\begin{align*}
    \|f_\ext\|_{L^p}^p
    = \int_{\R^d} |f_\ext (x)|^p\, dx
    \le \left(\sum_{n\in\Z^d}|f(n)|^p\int_{\R^d}\Psi(x-n)\, dx \right)\left(\sum_{n\in\Z^d}\Psi(x-n)\right)^{p-1} 
    \le \|f\|_{\ell^p}^p.
\end{align*}
\end{proof}

\begin{remark}\label{MagSteWeRMK1}
Due to this Proposition, we call $u_f$ the discrete harmonic extension of $f$.
\end{remark}  

When  $d=1$, we can use \eqref{Explicitd=1} to compute that 
\begin{align*}
	{\Psi_1(x):=}\Psi(x)=
	\begin{cases}
		\frac{\sin^2(\pi x)}{\pi^2 x^2}, & x\ne 0,\\
		1, & x=0.
	\end{cases}
	\end{align*}
{For an arbitrary $d = 1, 2, \dots$, s}et $\Psi_d(x)=\prod_{i=1}^d{\Psi_1}(x_i)$ and consider the function 
$$\wt{f}_{\ext}(x)=\sum_{n\in \Z^d}f(n)\Psi_d(x-n).$$ 
It was proved in~\cite{MSW} that for all $d\geq 1$, there exists a dimensional constant $A_d\geq 1$ such that 
\begin{align}\label{MSW}
    \frac{1}{A_d}\|f\|_{\ell^p}\leq \|\wt{f}_{\ext}\|_{L^p}\leq A_d\|f\|_{\ell^p}.
\end{align} 
The compact support of the Fourier transform of $\Psi_d(x)$, which we are not able to verify in our case for $\Psi(x)$ when  $d>1$, is crucial for the first inequality in~\eqref{MSW}. The bounds in~\eqref{MSW} were used in~\cite{MSW}*{Proposition 2.1} to show that the $L^p$-norm of a continuous Fourier multiplier operator $T$, when the multiplier is bounded and of compact support, controls the $\ell^p$-norm of its discrete version $T_{\dis}$ with a constant depending on $d$. More precisely, suppose $$\wh{K}(\xi)=\int_{\R^d} K(x)e^{-2\pi ix\cdot \xi} dx$$ is bounded and supported on the fundamental cube $Q={[-\frac12, \frac12)^d}$. Define the Fourier multiplier $\wh{Tf}(\xi)=\wh{K}(\xi)\wh{f}(\xi)$ and its discrete version by
\begin{align*}
    T_{\dis}f(n)=\sum_{m\in \Z^d} K(m)f(n-m).
\end{align*}
Fix $1\leq p\leq \infty$. If $T$ is bounded on $L^p(\R^d)$, then $T_{\dis}$ is bounded on $\ell^p(\Z^d)$ and 
\begin{align}\label{MSW-Ineq}
    \|T_{\dis}\|_{\ell^p\to\ell^p}\leq C_d\|T\|_{L^p\to L^p}, 
\end{align} 
where $C_d=3^d A_d^2$. 

\begin{remark}\label{MagSteWeRMK}
The problem raised in~\cite{MSW}*{Remark (1), pg.\,193}, \begin{quote}``it would be interesting to know if $C_d$ can be taken to be independent of $d$, or for that matter if $C_d=1$,''\end{quote} has been shown not to be the case, at least for $p$ near 1.  See 
\cite{Kov} 
for details.  

\end{remark} 
 
\subsection{Doob $h$-process and martingale transforms}\label{general-d}
For the function $h(x,y)$ defined in~\eqref{doob-h}, let $Z_t$ be a solution of the stochastic differential equation 
\begin{align*}
	dZ_{t}=dB_{t}+\frac{\nabla h(Z_{t})}{h(Z_{t})}dt,
\end{align*}
where $(B_t)_{t\geq 0}$ is the $(d+1)$-dimensional Brownian motion starting from $(x_{0},y_{0})\in\R^d\times\R_+$. The lifetime of $Z_{t}$ in the upper half-space is defined by $\tau=\inf\{t\geq0: Y_{t}=0\}$. The lifetime  $\tau$ is finite with probability one and the process $Z_t$ only exits the upper half-space on $\Z^d \times \{0\}$. Indeed, at its lifetime $Z_t$ approaches the point $n\in \Z^d$ with probability $\frac{p_n(x,y)}{h(x,y)}$, where $(x, y)$ is the starting point of $Z_t$. For the basic properties and stochastic calculus for the Doob $h$-processes we refer the reader to \cite[Chapter 3]{Bass}.

We denote by  $\ell^{p}_{c}(\Z^d)$ the space of all compactly supported functions in $\ell^{p}(\Z^d)$. For $f\in \ell^p_c(\Z^d)$, we define 
\begin{align*}
    M_{t}=M^{f}_{t}=u_{f}(Z_{t}),  \quad \text{ for } t\in(0,\tau).
\end{align*}
By It\^o's formula, $M^{f}_{t}$ is a martingale and satisfies
\begin{align}\label{eq:def_M_t}
	M^{f}_{t}
	&= M^{f}_{0}+\int_{0}^{t}\nabla u_{f}(Z_{s})\cdot dZ_{s} +\frac{1}{2}\int_{0}^{t}\Delta u_{f}(Z_{s})\, ds \\
	&= M^{f}_{0}+\int_{0}^{t}\nabla u_{f}(Z_{s})\cdot dZ_{s} -\int_{0}^{t}\frac{\nabla h(Z_{s})\cdot \nabla u_{f}(Z_{s})}{h(Z_{s})}\, ds \nonumber\\
	&= M^{f}_{0}+\int_{0}^{t}\nabla u_{f}(Z_{s})\cdot dB_{s},\nonumber 
\end{align}
where $(B_t)_{t\geq0}$ is the $(d+1)$-dimensional Brownian motion.

Let $\fM_{(d+1)}(\R)$ be the space of all $(d+1)\times (d+1)$ real matrices and denote its norm by 
$$ \|A\|=\sup\{\|Av\|:v\in\R^{d+1}, \|v\|\leq 1\}.$$
By abuse of notation, for a matrix-valued function $A:\R^d\times \R_+\to\fM_{(d+1)}(\R)$, that is, $A(x,y)=(a_{ij}(x,y))_{1\leq i,j\leq d+1}$ for all $x\in\R^d$ and $y>0$, we define 
\begin{align}\label{normofA}
	\|A\| &=\sup_{(x,y)\in\R^d\times\R_+}\|A(x,y)\|\\
	& = \sup\{\|A(x,y)v\|:v\in\R^{d+1}, \|v\|\leq 1, (x,y)\in\R^d\times\R_+\}\nonumber.
\end{align}
We say a matrix $A(x, y)=(a_{ij}(x, y))\in\fM_{(d+1)}(\R)$ is {\it orthogonal}
if $\brk{A(x, y)v,v}=\sum_{ij}a_{ij}v_i v_j=0$ for all $(x, y)\in \R^d\times \R_+$ and all $v\in\R^{d+1}$. Let $A(x,y)$ be a matrix-valued function on $\R^d\times \R_+$ and $f\in\ell^p_c(\Z^d)$. The martingale transform of $(M^f_{t})_{t\geq0}$ with respect to $A(x,y)$ is defined by
\begin{align*}
	(A\star M^f)_{t}
	&:= \int_{0}^{t}A(Z_{s})\nabla u_f(Z_{s})\cdot dB_{s}\\
    &= \int_{0}^{t}A(Z_{s})\nabla u_f(Z_{s})\cdot dZ_{s}-\int_{0}^{t}\frac{A(Z_s)\nabla u_f(Z_{s}) \cdot \nabla h(Z_{s})}{h(Z_{s})}\, ds. 
\end{align*}
From the martingale inequalities in \cite{Bur84} (general $A$) and \cite{BanWang} (orthogonal $A$), respectively, we have the following
\begin{theorem}\label{thm:MartTrans}
Let $1<p<\infty$ and recall that $p^\ast = \max\{p,\frac{p}{p-1}\}$. 
\begin{itemize}
\item[(i)] Let $A$ be a matrix-valued function with $\|A\|<\infty$. Then we have
\begin{align*}
    \| A\star M^f\|_p\leq (p^\ast -1)\|A\|\|M^f\|_p.
\end{align*}
\item[(ii)]  If additionally $A$ is orthogonal, then
\begin{align*}
    \| A\star M^f\|_p\leq \cot\left(\frac{\pi}{2 p^\ast}\right) \|A\|\|M^f\|_p.  
\end{align*}
\end{itemize} 
\end{theorem} 

\section{Discrete operators arising from martingale transforms}
\subsection{The projection operators $T_A$ and their $\ell^p$ boundedness}\label{ProjOP} 
For the rest of this paper we fix our starting point $(x_0, y_0)$ to be $(0, w)$, $w>0$. For $A:\R^d\times \R_+\to \fM_{(d+1)}(\R)$ with $\|A\|<\infty$,  $f\in \ell^p_c(\Z^d)$ and $n\in\Z^d$, we define 
\begin{align}\label{Martranpro}
    T_A^{w}(f)(n)  =\rE_{(0,w)}\Big[(A\star M^{f})_{\tau}\Big| X_{\tau}=n\Big].
\end{align}

We call these operators \textit{``projections of martingale transforms.''}  Our next goal is twofold. Firstly, we show that when $w\to\infty$, they give rise to a family of operators, denoted as  $T_A$, which are bounded on $\ell^p_c(\Z^d)$, $1<p<\infty$, with the same $\ell^p$-bounds as those given in Theorem \ref{thm:MartTrans}. In particular, these bounds are independent of $d$. Secondly, we compute their kernels. Although these constructions are in the style of Gundy--Varopoulos \cite{GV79}, we follow the approach in \cite{Ban86}*{Section 2} using the occupation time formula in terms of the Green's functions to compute their kernels.   

For each $n\in \Z^d$, we consider the processes $(Z_{t})_{t\geq0}$ starting at $(0,w)$ and conditioned to exit the upper half-space at $(n,0)$ and denote it by $Z_t^n$. Then $Z_t^n$ is just Brownian in the upper half-space with drift $\frac{\nabla p_n}{p_n}.$ Let us denote the Brownian motion which arises as the martingale part of $Z_t^n$ by $B^n_t$, and the expectation of $(Z^{n}_{t})_{t\geq0}$ by $\rE^{n}_{(0,w)}$. Then $T_A^{w}(f)$ can be written as
\begin{align*}
    T^{w}_A(f)(n)	&= \rE_{(0,w)}\Big[(A\star M^{f})_{\tau}\Big|X_{\tau}=n\Big] \\
    &= \rE_{(0,w)}\Big[\int_{0}^{\tau}A(Z_{s})\nabla u_{f}(Z_{s})\cdot dZ_{s}-\int_{0}^{\tau}\frac{A(Z_s)\nabla u_{f}(Z_{s})\cdot \nabla h(Z_{s})}{h(Z_{s})}ds\Big|X_{\tau}=n\Big] \\
    &= \rE_{(0,w)}^{n}\Big[\int_{0}^{\tau}A(Z_{s}^n)\nabla u_{f}(Z_{s}^{n})\cdot dZ_{s}^{n}-\int_{0}^{\tau}\frac{A(Z_s^n)\nabla u_{f}(Z_{s}^{n})\cdot \nabla h(Z_{s}^{n})}{h(Z_{s}^{n})}ds\Big] \\
	&= \rE_{(0,w)}^{n}\Big[\int_{0}^{\tau}A(Z^n_{s})\nabla u_{f}(Z_{s}^{n})\cdot dB^{{n}}_{s} +\int_{0}^{\tau}\frac{A(Z_s^n)\nabla u_{f}(Z_{s}^{n})\cdot \nabla p_n(Z_{s}^{n})}{p_n(Z_{s}^{n})}ds\\
    &\qquad	-\int_{0}^{\tau}\frac{A(Z_s^n)\nabla u_{f}(Z_{s}^{n})\cdot \nabla h(Z_{s}^{n})}{h(Z_{s}^{n})}ds\Big] \\
    &= \rE_{(0,w)}^{n}\Big[ \int_{0}^{\tau}A(Z_s^n)\nabla u_{f}(Z_{s}^{n})\cdot \left(\frac{\nabla p_n(Z_{s}^{n})}{p_n(Z_{s}^{n})}-\frac{\nabla h(Z_{s}^{n})}{h(Z_{s}^{n})}\right) ds\Big].
\end{align*}

Next, we use the occupation time formula to write this expectation as an integral over $\R^d\times\R_+$. Let us denote the Green's function for the upper half-space with pole $(0, w)$ by $G_{w}(x, y)$. Then 
\begin{align*}
	G_{w}(x,y)
    = \begin{cases}
        \ds\frac{1}{2\pi}\log\Paren{\frac{x^2+(y+w)^2}{x^2+(y-w)^2}}, & d=1, \\
        \ds\frac{\Gamma(\frac{d-1}{2})}{2\pi^{\frac{d+1}{2}}}\Paren{\frac{1}{(|x|^{2}+|w-y|^{2})^{\frac{d-1}{2}}}-\frac{1}{(|x|^{2}+|w+y|^{2})^{\frac{d-1}{2}}}}, & d\geq 2.
    \end{cases}
\end{align*} 
Since the occupation time measure for the process $Z_t^n$ is given by 
$$ \frac{p_{n}(x,y)G_{w}(x,y)}{p_{n}(0,w)},$$
it follows from the occupation time formula that
\begin{align*}
    T^{w}_A(f)(n)  
    &= \int_{\R^{d}}\int_{\R_{+}}\frac{p_{n}(x,y)G_{w}(x,y)}{p_{n}(0,w)} A(x,y)\nabla u_{f}(x,y)\cdot \left(\frac{\nabla p_n(x,y)}{p_n(x,y)}-\frac{\nabla h(x,y)}{h(x,y)}\right)\, dydx\\
    &= \int_{\R^{d}}\int_{\R_{+}}\frac{h(x,y)G_{w}(x,y)}{p_{n}(0,w)} A(x,y)\nabla u_{f}(x,y)\cdot \left(\frac{\nabla p_n(x,y)}{h(x,y)}-\frac{p_{n}(x,y) \nabla h(x,y)}{h(x,y)^2}\right)\, dydx\\
    &= \int_{\R^{d}}\int_{\R_{+}}\frac{h(x,y)G_{w}(x,y)}{p_{n}(0,w)} A(x,y) \nabla u_{f}(x,y)\cdot \nabla \left(\frac{p_n(x,y)}{h(x,y)}\right)\, dydx.
\end{align*} 
    
For $m,n\in \Z^d$, we define the kernel for the operator $T^{w}_A$ by 
\begin{equation}\label{kernel1}
    \bK^{w}_A(n,m) :=T^{w}_A(\delta_{m})(n) 
\end{equation}
where $\delta_m(n)=1$ if $m=n$ and otherwise 0. Note that $u_{\delta_{m}}(x,y)=\frac{p_{m}(x,y)}{h(x,y)}$ and hence we have 
\begin{equation}\label{kernel2}
    T^w_A(f)(n)=\sum_{m\in\Z^d}\bK^{w}_A(n,m)f(m), 
\end{equation}  
where 
\begin{align}\label{kernel3} 
	\bK^{w}_A(n,m) =\int_{\R^{d}}\int_{\R_{+}}\frac{G_{w}(x,y)}{p_{n}(0,w)}h(x,y) A(x,y)\nabla \left(\frac{p_m}{h}\right)\cdot \nabla \left(\frac{p_n}{h}\right)\, dydx.
\end{align}

With the kernels $\bK^{w}_A(n,m)$ defined for all $w>0$, we like to compute the limit as $w\to \infty$ and study their properties. For each $m,n\in\Z^d$ and $f\in\ell^p_c(\Z^d)$, we define 
\begin{equation}
    \bK_A(n,m)=\lim_{w\to\infty}\bK^{w}_A(n,m)\label{kernel}
\end{equation}
and 
\begin{equation}
    T_A(f)(n)=\sum_{m\in\Z^d}\bK_A(n,m)f(m)\label{disoper}.
\end{equation}
The following theorem shows that $T_A$ is well-defined and gives an explicit expression for it.  

\begin{theorem}\label{thm:TaIntRep}
Let $A(x,y)$ be a matrix-valued function with $\|A\|<\infty$ and $m,n\in\Z^d$. Then, $\bK^w_A(n,m)$ converges as $w\to\infty$ and 
\begin{align}\label{kernelrep}
    \bK_A(n, m) =\lim_{w\to\infty}\bK^w_A(n,m) = \int_{\R^{d}}\int_{\R_{+}} 2y h(x,y) A(x,y)\nabla \left(\frac{p_m}{h}\right)\cdot \nabla \left(\frac{p_n}{h}\right)\, dydx.
\end{align}
\end{theorem}

\begin{remark}\label{rmk:kernelsymmetry}
    Note that if $A(x,y)=A$ is a constant matrix, then the change of variables $x=-z+m-n$ gives $K_{A}(n,m)= -K_{A^T}(n,m)$.
\end{remark}

\begin{lemma}\label{lem:greenlimit}
Let $d\geq 1$, $n\in \Z^d$, and $(x,y)\in\R^{d}\times\R_{+}$. Then we have 
\begin{equation*}
    \lim_{w\to\infty}\frac{G_{w}(x,y)}{p_{n}(0,w)}=2y.
\end{equation*}

{Suppose that $d \geq 2$. Then
\begin{equation*}
    \frac{G_{w}(x,y)}{p_{n}(0,w)} \leq \begin{cases}
     C_d y & y \le \tfrac{w}{2} , \\
     C_d w (|\tfrac{x}{w}|^2 + (\tfrac{y}{w} - 1)^2)^{-(d - 1)/2} & y > \tfrac{w}{2} .
    \end{cases}
\end{equation*}
when $n\in \Z^d$ and $w\geq |n|$.

For $d = 1$, we have
\begin{equation*}
    \frac{G_{w}(x,y)}{p_{n}(0,w)} \leq 2yg\left( \frac{y}{w} \right) 
\end{equation*}
if $n \in \Z^d$ and $w \geq |n|$, where 
\begin{align*}
    g(t)=
        \frac{1}{2t} \log(1+4t(t-1)^{-2}) .
\end{align*}}

\end{lemma}
\begin{proof}
The case $d=1$ was proven in~\cite{BanKwa}*{Lemma 3.3}. For $d\geq 2$, we use the mean value theorem to get
\begin{align*}
    \frac{1}{(|x|^{2}+|w-y|^{2})^{\frac{d-1}{2}}}-\frac{1}{(|x|^{2}+|w+y|^{2})^{\frac{d-1}{2}}}
    &= 2(d-1)wy(|x|^{2}+|w-y|^{2}+4wy\varep)^{-\frac{d+1}{2}}\\
    &= 2(d-1)yw^{-d}(1+|{\tfrac{x}{w}}|^{2}+|{\tfrac{y}{w}}|^{2}+2(2\varep-1){\tfrac{y}{w}})^{-\frac{d+1}{2}}
\end{align*}
for some $\varep\in(0,1)$. The first part follows now from
\begin{align*}
    \frac{G_{w}(x,y)}{p_{n}(0,w)}
    & = \frac{\Gamma(\frac{d-1}{2})}{2\pi^{\frac{d+1}{2}}} \frac{2(d-1)yw^{-d}(1+|\frac{x}{w}|^{2}+|\frac{y}{w}|^{2}+2(2\varep-1)\frac{y}{w})^{-\frac{d+1}{2}}}{c_{d}w^{-d}(1+\frac{|n|^{2}}{|w|^{2}})^{-\frac{d+1}{2}}} \\
    & = \frac{2y(1+\frac{|n|^{2}}{|w|^{2}})^{\frac{d+1}{2}}}{(1+|\frac{x}{w}|^{2}+|\frac{y}{w}|^{2}+2(2\varep-1)\frac{y}{w})^{\frac{d+1}{2}}} ,
\end{align*}
{because the right-hand side clearly converges to $2 y$ as $w \to \infty$. To prove the second part of the statement, we observe that}
\begin{align*}
    \frac{G_{w}(x,y)}{p_{n}(0,w)}
    = \frac{2y(1+\frac{|n|^{2}}{|w|^{2}})^{\frac{d+1}{2}}}{(1+|\frac{x}{w}|^{2}+|\frac{y}{w}|^{2}+2(2\varep-1)\frac{y}{w})^{\frac{d+1}{2}}}
    \leq 2y\Big(\frac{2}{|\frac{y}{w}-1|^{2}}\Big)^{\frac{d+1}{2}} \le 2^{1 + 3 (d + 1) / 2} y ,
\end{align*}
when $w \geq |n|$ {and $y \le \tfrac{w}{2}$. Furthermore, when $y > \tfrac{w}{2}$, we have
\begin{align*}
    \frac{G_{w}(x,y)}{p_{n}(0,w)}
    & \le \frac{\Gamma(\frac{d - 1}{2})}{2 \pi^{(d + 1) / 2} c_d} \, \frac{(|n|^2 + w^2)^{(d + 1) / 2}}{w (|x|^2 + (y - w)^2)^{(d - 1) / 2}} \\
    & \le \frac{1}{d - 1} \, \frac{(2 w)^{d + 1}}{w^d (|\frac{x}{w}|^2 + (\frac{y}{w} - 1)^2)^{(d - 1) / 2}}
\end{align*}
when $w \geq |n|$.
}
\end{proof}

\begin{lemma}\label{lem:gradbound}
Let $d\geq 1$, $n\in \Z^d$, $x\in\R^{d}$, and $y>0$, then we have
\begin{equation*}
    \frac{|\nabla p_{n}(x,y)|}{p_{n}(x,y)}\leq \frac{d}{y},\quad
    \frac{|\nabla h(x,y)|}{h(x,y)}\leq \frac{d}{y}.
\end{equation*}
\end{lemma}
\begin{proof}
Direct computation gives
\begin{align}
    \parder{}{x_{i}}p_{n}(x,y)
    &= -\frac{c_{d}(d+1)(x_{i}-n_{i})y}{(|x-n|^{2}+y^{2})^{\frac{d+3}{2}}}
    =-\frac{(d+1)(x_{i}-n_{i})}{|x-n|^{2}+y^{2}}p_{n}(x,y),\label{Poisshor} \\
    \parder{}{y}p_{n}(x,y)
    &= \frac{c_{d}(|x-n|^{2}-dy^{2})}{(|x-n|^{2}+y^{2})^{\frac{d+3}{2}}}
    = \frac{(|x-n|^{2}-dy^{2})}{y(|x-n|^{2}+y^{2})}p_{n}(x,y)\label{Poissver}.
\end{align}
This in turn gives 
\begin{align*}
    \frac{|\nabla p_{n}(x,y)|^{2}}{p_{n}(x,y)^{2}}
    &= \frac{1}{y^{2}(|x-n|^{2}+y^{2})^{2}}\Big((d+1)^{2}y^{2}|x-n|^{2}+(|x-n|^{2}-dy^{2})^{2}\Big) \\
    &= \frac{1}{y^{2}(|x-n|^{2}+y^{2})^{2}}\Big(|x-n|^{4}+(d^{2}+1)y^{2}|x-n|^{2}+d^{2}y^{4}\Big) \\
    &= \frac{1}{y^{2}(|x-n|^{2}+y^{2})^{2}}\Big((|x-n|^{2}+d^{2}y^{2})(|x-n|^{2}+y^{2})\Big) \\
    &= \frac{|x-n|^{2}+d^{2}y^{2}}{y^{2}(|x-n|^{2}+y^{2})} \\
    &\leq  \frac{d^{2}}{y^{2}}.
\end{align*}
Let $N>0$ and set $h^{N}(x,y)=\sum_{n\in \Z^d, |n|\leq N}p_{n}(x,y)$. 
Note that $h^{N}(x,y)\to h(x,y)$ and $\nabla h^{N}(x,y)\to \nabla h(x,y)$ as $N\to \infty$ uniformly on compact sets in $\R^{d}\times \R_{+}$. 
We then have
\begin{align*}
    |\nabla h^{N}(x,y)|
    &\leq  \sum_{n\in \Z^d, |n|\leq N}|\nabla p_{n}(x,y)| \\
    &\leq  \frac{d}{y}\sum_{n\in \Z^d, |n|\leq N} p_{n}(x,y) \\
    &\leq  \frac{d}{y}h(x,y).
\end{align*}
Letting $N\to \infty$, we get the desired result.
\end{proof}

\begin{proof}[Proof of Theorem~\ref{thm:TaIntRep}]
Let 
\begin{align*}
    j_{n,m}(x,y,w)
    =\frac{G_{w}(x,y)}{p_{n}(0,w)}h(x,y)A(x,y)\nabla \left(\frac{p_m}{h}\right)\cdot \nabla \left(\frac{p_n}{h}\right),
\end{align*}
then $\bK^w_A(n,m) = \iint j_{n,m}(x,y,w)\, dxdy$. Let $w\geq |n|$ and define 
\begin{align*}
    j_{n,m}^{(1)}(x,y,w) &= j_{n,m}(x,y,w)\ind_{\{0<y<w/2\}},\\
    j_{n,m}^{(2)}(x,y,w) &= j_{n,m}(x,y,w)\ind_{\{y\ge w/2\}}.
\end{align*}
We claim that
\begin{align*}
    \lim_{w\to\infty}\iint j_{n,m}^{(1)}(x,y,w)\, dxdy & = \int_{\R^{d}}\int_{\R_{+}} 2y h(x,y) A \nabla \left(\frac{p_m}{h}\right)\cdot \nabla \left(\frac{p_n}{h}\right)\, dydx,\\
    \lim_{w\to\infty}\iint j_{n,m}^{(2)}(x,y,w)\, dxdy &= 0.  
\end{align*} 
By Lemma~\ref{lem:gradbound}, we have
\begin{equation*}
	|\nabla (\log\frac{p_{m}(x,y)}{h(x,y)})|
	\leq \frac{|\nabla p_{m}(x,y)|}{p_{m}(x,y)} +\frac{|\nabla h(x,y)|}{h(x,y)} 
	\leq \frac{2d}{y}
\end{equation*}
and
\begin{equation*}
    \Abs{A(x,y)\nabla \left(\log\frac{p_{m}(x,y)}{h(x,y)}\right)\cdot \nabla\left(\log\frac{p_{n}(x,y)}{h(x,y)}\right)}\leq \frac{C_{d,A}}{y^{2}}.
\end{equation*}
Note that if $0<y<w/2$, then $g(y/w)\leq C$ for some $C>0$. By Lemma~\ref{lem:greenlimit}, $$G_w(x,y)/p_n(0,w)\leq 
C y.$$ Thus, we have 
\begin{align*}
    j_{n,m}^{(1)}(x,y,w)
    &=\frac{G_{w}(x,y)}{p_{n}(0,w)} \frac{p_m p_n}{h}A \nabla \left(\log \frac{p_m}{h}\right)\cdot  \nabla \left(\log \frac{p_n}{h}\right) \\
    &\leq \frac{C_{d,A}p_m(x,y) p_n(x,y)}{yh(x,y)}.
\end{align*} 
If $y\ge 1$, then it follows from Lemma~\ref{lem:hlimit} that 
\begin{align*}
	\frac{p_{m}(x,y)p_{n}(x,y)}{h(x,y)}
    \leq C p_m(x,y)p_n(x,y)
    \leq C_{d,m,n}\frac{y^2}{(|x|^2+y^2)^{d+1}}.
\end{align*}	
Since 
\begin{align*}
    \int_1^\infty \int_{\R^d} \frac{y}{(|x|^2+y^2)^{d+1}}\, dxdy<\infty,
\end{align*}
it follows from the dominated convergence theorem that
\begin{align*}
    \lim_{w\to\infty}\int_1^\infty \int_{\R^d} j_{n,m}^{(1)}(x,y,w)\, dxdy
    = \int_1^\infty\int_{\R^{d}} 2y h(x,y) A(x,y)\nabla \left(\frac{p_m}{h}\right)\cdot \nabla \left(\frac{p_n}{h}\right)\, dxdy.
\end{align*} 
Suppose $0<y<1$ and $n\neq m$. Using $p_m(x,y) + p_n(x,y)\leq h(x,y)$, we get 
\begin{align*}
	\frac{p_{m}(x,y)p_{n}(x,y)}{h(x,y)}
	&\leq \frac{c_{d}y}{(|x-n|^{2}+y^{2})^{\frac{d+1}{2}}+(|x-m|^{2}+y^{2})^{\frac{d+1}{2}}} \\
	&\leq \frac{2^{\frac{d-1}{2}}c_{d}y}{(|x-n|^{2}+|x-m|^{2}+2y^{2})^{\frac{d+1}{2}}} \\
	&{=} \frac{c_{d}y}{2(|x-(\frac{n+m}{2})|^{2}+y^{2}+|\frac{n+m}{2}|^{2})^{\frac{d+1}{2}}} \\
	&\leq C_{d,n,m}\frac{y}{(|x|^{2}+1)^{\frac{d+1}{2}}}.
\end{align*}	
Since 
\begin{align*}
    \int_0^1\int_{\R^d} \frac{1}{(|x|^{2}+1)^{\frac{d+1}{2}}}\, dxdy<\infty,
\end{align*}
we obtain that 
\begin{align*}
    \lim_{w\to\infty}\int_0^1 \int_{\R^d} j_{n,m}^{(1)}(x,y,w)\, dxdy
    = \int_0^1\int_{\R^{d}} 2y h(x,y) A(x,y)\nabla \left(\frac{p_m}{h}\right)\cdot \nabla \left(\frac{p_n}{h}\right)\, dxdy
\end{align*} 
when $n\neq m$. 

Suppose $0<y<1$ and $n=m$. For $f=\delta_n$, it follows from It\^o's formula that
\begin{align*}
    M_t^f 
    = u_f(Z_t)
    = u_f(0,w)+\int_0^t \nabla u_f(Z_s)\cdot dB_s
    =\frac{p_n(0,w)}{h(0,w)}+\int_0^t \nabla\left(\frac{p_n(Z_s)}{h(Z_s)}\right)\cdot dB_s.
\end{align*}
Note that $\rE_{(0,w)}[M_\tau^f|X_\tau=n]=1$. 
Applying~\eqref{kernel3} with $A(x,y)=\mathrm{Id}$, we get
\begin{align*}
    \rE_{(0,w)}\left[\int_0^\tau \nabla\left(\frac{p_n(Z_s)}{h(Z_s)}\right)\cdot dB_s\bigg|X_\tau=n\right]
    = \int_0^\infty \int_{\R^d} \frac{G_w(x,y)}{p_n(0,w)}h(x,y) \left|\nabla\left(\frac{p_n}{h}\right)\right|^2\, dxdy,
\end{align*}
which leads to
\begin{align}\label{Ito}
    \int_0^\infty \int_{\R^d} \frac{G_w(x,y)}{p_n(0,w)}h(x,y) \left|\nabla\left(\frac{p_n}{h}\right)\right|^2\, dxdy = 1-\frac{p_n(0,w)}{h(0,w)}.
\end{align}
Fix $N>0$. Note that it follows from the proof of Lemma~\ref{lem:greenlimit} that there exists a constant $C$ depending only on $d$ such that for large $w$,
\begin{align*}
    \frac{G_w(x,y)}{p_n(0,w)}\ge C y
\end{align*}
for all $0<y<1$ and $|x|\le N$. Thus, we obtain
\begin{align*}
    \int_0^1\int_{|x|\le N} y h(x,y) \left|\nabla\left(\frac{p_n}{h}\right)\right|^2\, dxdy \leq C.
\end{align*}
By Lemma~\ref{lem:greenlimit} and the dominated convergence theorem, we get
\begin{align*}
    \lim_{w\to\infty}\int_0^1\int_{|x|\le N}j_{n,n}^{(1)}(x,y,w)\, dxdy 
    = \int_0^1\int_{|x|\le N} 2y h(x,y) A(x,y)\nabla \left(\frac{p_n}{h}\right)\cdot \nabla \left(\frac{p_n}{h}\right)\, dxdy.
\end{align*}
Since $|j_{n,n}^{(1)}{(\cdot, \cdot, w)}|$ is integrable over $\R^d{{} \times (0, 1)}$, we have
\begin{align*}
    \int_0^1\int_{|x|\ge N}j_{n,n}^{(1)}(x,y,w)\, dxdy \to 0
\end{align*}
as $N\to \infty$. Using the previous argument, we see that 
\begin{align*}
    |2y h(x,y) A(x,y)\nabla \left(\frac{p_n}{h}\right)\cdot \nabla \left(\frac{p_n}{h}\right)|
    \le C\frac{p_n(x,y)^2}{yh(x,y)}
    \le C\frac{p_n(x,y)}{y}
    = \frac{C}{(|x-n|^2+y^2)^{(d+1)/2}}.
\end{align*}
Since the integral of $(|x-n|^2+y^2)^{-(d+1)/2}$ over $|x|\ge N$ converges to 0 as $N\to\infty$, we get
\begin{align*}
    \lim_{w\to\infty}\int_0^1\int_{\R^d}j_{n,n}^{(1)}(x,y,w)\, dxdy 
    = \int_0^1\int_{\R^d} 2y h(x,y) A(x,y)\nabla \left(\frac{p_n}{h}\right)\cdot \nabla \left(\frac{p_n}{h}\right)\, dxdy
\end{align*}
as desired. 

For the other integral, we have
\begin{align*}
    j_{n,m}^{(2)}(x,y,w)
	&\leq C_{d,n,m,A}\frac{y^{2}}{(|x|^{2}+y^{2})^{d+1}} \cdot {\frac{C_d w}{(|\frac{x}{w}|^2 + |\frac{y}{w} - 1|^2)^{(d - 1) / 2}}} 
\cdot \frac{1}{y^{2}} \\ 
	&{{} = C_{d,n,m,A} \frac{1}{w^{2 d + 1} (|\frac{x}{w}|^2 + |\frac{y}{w}|^2)^{d + 1} (|\frac{x}{w}|^2 + |\frac{y}{w} - 1|^2)^{(d - 1) / 2}}} .
\end{align*}
Thus,
\begin{align*}
	{\int_0^{\infty}\int_{\R^{d}} j_{n,m}^{(2)}(x,y,w) \,dxdy}
	& { \le \frac{C_{d,n,m,A}}{w^{2 d + 1}} \int_{\frac{w}{2}}^{\infty}\int_{\R^{d}} \frac{1}{(|\frac{x}{w}|^2 + |\frac{y}{w}|^2)^{d + 1} (|\frac{x}{w}|^2 + |\frac{y}{w} - 1|^2)^{(d - 1) / 2}} \, dxdy} \\
	& {= \frac{C_{d,n,m,A}}{w^d} \int_{\frac{1}{2}}^{\infty}\int_{\R^{d}} \frac{1}{(|v|^2 + s^2)^{d + 1} (|v|^2 + |s - 1|^2)^{(d - 1) / 2}} \, dvds}
\end{align*}
and the right-hand side clearly converges to zero as $w \to \infty$.
\end{proof}

\begin{theorem}\label{Proj} 
Let $f\in\ell^p(\Z^d)$, $1<p<\infty$,  and $A(x,y)$ be a matrix-valued function with $\|A\|<\infty$. Then 
\begin{equation}\label{projTH-gen}
    \norm{T_A(f)}_{\ell^{p}} \leq (p^\ast-1)\|A\|\norm{f}_{\ell^{p}}. 
\end{equation}
If in addition, $A(x,y)$ is orthogonal for all $(x,y)\in\R^d\times \R_+$, then 
\begin{equation}\label{projTH-gen-2}
    \norm{T_A(f)}_{\ell^{p}} \leq \cot\left(\frac{\pi}{2p^{\ast}}\right)\|A\|\norm{f}_{\ell^{p}}.	
\end{equation}
\end{theorem} 
\begin{proof} 
Suppose $f:\Z^d\to\R$ is compactly supported. By the sharp martingale inequality (Theorem~\ref{thm:MartTrans}) and Jensen's inequality for conditional expectations, we have 
\begin{align}\label{eq:sharpmgineq1-gen}
	\sum_{n\in \Z^d}|T^{w}_A(f)(n)|^{p}\frac{p_{n}(0,w)}{h(0,w)}
	&= \rE_{(0,w)}[|T^{w}_A(f)(X_{\tau})|^{p}] \\\nonumber
    &= \rE_{(0,w)}[|\rE_{(0,w)}[(A \ast M^f)_{\tau}|X_{\tau}]|^{p}] \\\nonumber
    &\le \rE_{(0,w)}[\rE_{(0,w)}[|(A \ast M^f)_{\tau}|^{p}|X_{\tau}]] \\\nonumber
	&= \rE_{(0,w)}[|(A \star M^{f})_{\tau}|^{p}] \\\nonumber
	&\leq (p^\ast -1)^{p}\|A\|^p\rE_{(0,w)}[|M^{f}_{\tau}|^{p}] \\\nonumber
	&= (p^\ast -1)^{p}\|A\|^p\sum_{n\in \Z^d}|f(n)|^{p}\frac{p_{n}(0,w)}{h(0,w)}.
\end{align}
Since
\begin{equation*}
	p_{n}(0,w)
	=\frac{c_{d}w}{(|n|^{2}+w^{2})^{\frac{d+1}{2}}}
	=\frac{c_{d}}{w^{d}}\frac{1}{(\frac{|n|^{2}}{w^{2}}+1)^{\frac{d+1}{2}}}
    \le \frac{c_{d}}{w^{d}},
\end{equation*}
we get
\begin{align*}
    \frac{1}{c_d}\sum_{n\in \Z^d}|T^{w}_A(f)(n)|^{p} w^d p_{n}(0,w)
	&\le (p^\ast -1)^{p}\|A\|^p\sum_{n\in \Z^d}|f(n)|^{p}.
\end{align*}
Recall that $T_A(f)(n)$ is the (pointwise) limit of $T_A^w(f)(n)$. By Fatou's lemma, we get 
\begin{align*}
	\sum_{n\in \Z^d}|T_A(f)(n)|^{p}
    &\leq \liminf_{w\to\infty} \frac{1}{c_d}\sum_{n\in \Z^d}|T^{w}_A(f)(n)|^{p} w^d p_{n}(0,w)\\
	&\leq (p^\ast -1)^{p}\|A\|^p\sum_{n\in \Z^d}|f(n)|^{p},
\end{align*}
which proves \eqref{projTH-gen}. The proof of  \eqref{projTH-gen-2} follows from the same argument using the second part of Theorem~\ref{thm:MartTrans}.  
\end{proof} 
The following Littlewood--Paley inequality is the analogue in our current setting of the inequalities in \cite[Corollaries 3.42 and 3.9.2]{Ban}.  

\begin{corollary}
Let $p\in(1,\infty)$, $q=p/(p-1)$, $f\in \ell^p(\Z^d)$, and $g\in \ell^q(\Z^d)$, then
\begin{align*}
    \int_0^\infty\int_{\R^d} 2yh(x,y)|\nabla u_f(x,y)||\nabla u_g(x,y)|\, dxdy
    \le (p^\ast -1) \norm{f}_{{\ell^{p}}}\norm{g}_{{\ell^{q}}}.
\end{align*}
\end{corollary}
\begin{proof}
Let $A(x,y)$ be a matrix-valued function with $\|A\|<\infty$. Assume that $f$ and $g$ have compact supports. By~\eqref{kernelrep} and~Theorem~\ref{Proj}, we have
\begin{align*}
    \Big| \sum_{n\in\Z^d} T_A f(n)g(n)\Big|
    &= \Big|\sum_{n,m\in\Z^d} \left( \int_{\R^{d}}\int_{\R_{+}} 2y h(x,y) A(x,y)\nabla \left(\frac{p_m}{h}\right)\cdot \nabla \left(\frac{p_n}{h}\right)\, dydx \right) f(m)g(n)\Big|\\
    &= \Big|\int_{\R^{d}}\int_{\R_{+}} 2y h(x,y) A(x,y)\nabla u_f(x,y) \cdot \nabla u_g(x,y)\, dydx\Big|\\
    &\le (p^\ast -1)\|A\|\norm{f}_{{\ell^{p}}}\norm{g}_{{\ell^{q}}}.
\end{align*}
Define
\begin{align*}
    A_{ij}(x,y) =\frac{\partial_j u_f(x,y) \partial_i u_g(x,y)}{|\nabla u_f(x,y)||\nabla u_g(x,y)|}
\end{align*}
for $1\le i,j\le d+1$. Here, we used the notations $\partial_i=\partial_{x_i}$ for $i=1,2,\cdots,d$ and $\partial_{d+1}=\partial_y$. Since $\|A\|\le 1$ and $A\nabla u_f \nabla u_g =|\nabla u_f ||\nabla u_g|$, the proof is complete.
\end{proof}

\begin{remark}
Let $a:\Z^d\to (0,\infty)$. We define $h_a(x,y) = \sum_{n\in\Z^d} a(n)p_n(x,y)$ and
\begin{align*}
    u^a_f(x,y) = \sum_{n\in\Z^d} f(n)\frac{a(n)p_n(x,y)}{h^a(x,y)}.
\end{align*}
For $f:\Z^d\to \R$ we define its weighted $\ell^p(a)$ norm by
$$\norm{f}_{\ell^{p}(a)}=\left(\sum_{n\in \Z^d}|f(n)|^p a(n)\right)^{1/p}.$$ 
Following the same argument with $h_a(x,y)$, and under the assumption that 
$$\sum_{n\neq 0}a(n) |n|^{-d-1}<\infty,$$
which guarantees that the harmonic function $h_a(x,y)$ is well-defined, we can define a discrete operator $T^{h_a}_A$ associated with $h_a$ and obtain the weighted norm inequalities:
\begin{equation}\label{weight1}
    \norm{T^{h_a}_A f}_{\ell^{p}(a)} \leq (p^\ast-1)\|A\|\norm{f}_{\ell^{p}(a)},  
\end{equation}
for any $A$.  If in addition  $A$ is orthogonal then, 
\begin{equation}\label{weight2}
    \norm{T^{h_a}_A f}_{\ell^{p}(a)}\leq \cot\left(\frac{\pi}{2p^{\ast}}\right)\|A\|\norm{f}_{\ell^{p}(a)}.	
\end{equation}
\end{remark}

\section{Discrete Calder\'on--Zygmund operators}
\subsection{Discrete Calder\'on--Zygmund operators and norm estimates}\label{CZO}  
Let $T$ be an operator acting on the Schwartz space of rapidly decreasing functions on $\R^d$. We say $T$ is a Calder\'on--Zygmund operator if it is bounded in $L^2$ and can be written as 
\begin{align}\label{CZK-1} 
    Tf(x) =p.v. \int_{\R^d} K(x,z) f(z)\, dz
\end{align}
where $K$ is continuously differentiable off the diagonal with the bounds 
\begin{align}\label{CZK-2}
    |K(x,z)|\leq \frac{\kappa}{|x-z|^{d}},\quad
    |\nabla_x K(x,z)|\leq \frac{\kappa}{|x-z|^{d+1}},\quad
    |\nabla_z K(x,z)|\leq \frac{\kappa}{|x-z|^{d+1}},
\end{align}
for $x\neq z$, for some universal constant $\kappa$.

The Calder\'on--Zygmund operator $T$ as above are bounded in $L^p$, for $1<p<\infty$ (see~\cite[Chapter 8]{Graf}). Here we will consider Calder\'on--Zygmund operators which are of convolution type. That is, their kernels are of the form $K(x,z)=K(x-z)$ satisfying 
\begin{align}\label{eq:ConK-def}
    K\in C^1(\R^d\setminus \{0\}), \quad |K(z)|\leq \kappa|z|^{-d}, \quad |\nabla K(z)|\leq \kappa|z|^{-(d+1)}, 
\end{align}
for some universal constant $\kappa$.
 
For these operators, Calder\'on and Zygmund \cite{CZ}  defined their discrete analogues  by
\begin{align}\label{eq:Tdis-def}
    \Tdis (f)(n)=\sum_{m\in\Z^d\setminus\{n\}}K(n-m)f(m),\quad f\in\ell^p(\Z^d).
\end{align} 
As already mentioned in the introduction, M.~Riesz \cite{Riesz} showed that in dimension~1, the boundedness of $H$ on $L^p(\R)$ implies the boundedness of $H_{\dis}$ on $\ell^p(\Z)$. In the  ``Added in proof'' section of their paper Calder\'on and Zygmund observed that the boundedness of $T$  on $L^p(\bR^d)$ leads to the boundedness of $\Tdis$  on $\ell^p(\Z^d)$. In fact, Calder\'on and Zygmund simply  remark{ed} (\cite[pg.~138]{CZ}) that ``for $n=1$ this remark is due to M. Riesz, and the proof in the case of general $n$ follows a similar pattern'' (here their $n=d$) and no further details are provided. For the sake of completeness and because we wish to keep track of constants, we provide the proof here. Recall that the truncated operator $T_\varep$ is defined by $T_\varep(f)=K_\varep\ast f$ where $K_\varep(x)=K(x)\ind_{\{|x|\ge \varep\}}$, $T_{\varep}$ satisfies $\|T_{\varep}f \|_{L^p}\leq C_p\|f\|_{L^p}$, where $C_p$ is independent of $\varep$ and $\lim_{\varep\to 0}T_{\varep}f$ exists in $L^p(\R^d)$ and a.e. We denote the limit operator by $T$.

\begin{proposition}\label{thm:RCZ}
Let $T$ be the Calder\'on--Zygmund operator with convolution kernel $K(x)$ satisfying \eqref{eq:ConK-def}. Then, $\Tdis$ is bounded on $\ell^p(\Z^d)$, $1<p<\infty$. Furthermore, we have 
$$\|T_{\dis}\|_{\ell^p\to\ell^p}\leq \|T_1\|_{L^p\to L^p}+C(d, \kappa).$$ 
\end{proposition}
\begin{proof}
As asserted by Calder\'on and Zygmund, the proof follows the argument of Riesz. Let $p\in(1,\infty)$ and $q$ be the conjugate exponent, that is, $\frac1p+\frac1q=1$. Let $f\in\ell^p(\Z^d)$ and $g\in\ell^q(\Z^d)$.  Define $F\in L^p(\R^d)$ and $G\in L^q(\R^d)$ by $F(x)=f(n)$ and $G(x)=g(n)$ for $x\in n+Q$, $Q=[-\frac12,\frac12)^d$. Then,
\begin{align*}
    \int_{\R^d} T_{1}(F)(x)G(x)\, dx
    &= \int_{\R^d}\int_{\R^d} K_{1}(y-x)F(x)G(y)\, dxdy\\
    &= \sum_{n,m\in\Z^d}\left(\int_{m+Q}\int_{n+Q}K_{1}(y-x)\, dxdy\right)f(n)g(m)\\
    &= \sum_{n\in\Z^d}\Tdis(f)(n) g(n) + \sum_{n,m\in\Z^d}\wt{K_1}(n-m) f(n)g(m),
\end{align*}
where 
\begin{align*}
    \wt{K_{1}}(n-m)
    &=\int_{m+Q}\int_{n+Q}K_{1}(y-x)\, dxdy-K_{1}(m-n)\\
    &=\int_{Q}\int_{Q}(K_{1}(m-n+s-t)-K_{1}(m-n))\, dtds.
\end{align*}

Using $|\nabla K_{1}(x-z)|\leq \kappa|x-z|^{-(d+1)}$, we have
\begin{align*}
    |\wt{K_{1}}(n-m)|\leq C(d, \kappa) |n-m|^{-(d+1)},     
\end{align*}
for $|m-n|$ large enough. Since $|m|^{-(d+1)}$ is summable, we have
\begin{align*}
    \left | \sum_{n,m\in\Z^d}\wt{K_1}(n,m) f(n)g(m)\right |
    &\leq \left( \sum_{n,m\in\Z^d}|\wt{K_1}(n,m)| |f(n)|^p \right)^{1/p} \left( \sum_{n,m\in\Z^d}|\wt{K_1}(n,m)| |g(m)|^q \right)^{1/q}\\
    &\leq C(d, \kappa)\|f\|_{\ell^p}\|g\|_{\ell^q} 
\end{align*}
and this gives 
$$\|T_{\dis}\|_{\ell^p\to \ell^p}\leq \|T_{1}\|_{L^p\to L^p}+C(d, \kappa),$$
where the constant $C(d, \kappa)$ depends on $d$ and $\kappa$ but not on $p$.  
\end{proof}

Riesz's argument was modified  in \cite{HunMukWhe} to prove a discrete $A_p$-weighted version of the celebrated Hunt-Muckenhoupt-Wheeden weighted norm inequality for the Hilbert transform.  See also \cite[Theorem 4.6]{Laeng}.  

In~\cite{Titc26}, Titchmarsh gave (with a slightly different version of $H$) a different proof of Riesz's theorem by first showing that $H_{\dis}$ is bounded on $\ell^p$ and from this that $H$ is bounded in $L^p$ and that in fact $\|H\|_{L^p\to L^p}\leq \|H_{\dis}\|_{\ell^p\to\ell^p}$. We show next that a similar result holds for singular integrals that commute with dilations. More precisely, consider singular integrals with kernels of the form $K(x)=\frac{\Omega(x)}{|x|^{d}}$, where $\Omega$ is homogeneous of degree zero; $\Omega(rx)=\Omega(x),$ for all $r>0$. We assume that $\Omega$ satisfies the necessary hypothesis (see for example \cite[Theorem 3]{Stein70}) so that the singular integral is bounded on $L^p$.  That is, (i) $\Omega$ is bounded, (ii) Dini continuous, and (iii) its integral on the sphere is $0$.

\begin{theorem}\label{homeker} 
Suppose $K(x)=\frac{\Omega(x)}{|x|^{d}}$ satisfying (i)--(iii) 
and $\Omega(-x)=-\Omega(x)$ for all $x\in\R^d\setminus\{0\}$. For $1<p<\infty$, we have $ \|T\|_{L^p\to L^p}\leq \|T_{\dis}\|_{\ell^p\to \ell^p}$.
\end{theorem}
\begin{proof}
We define the continuous-discrete operator $\wt{T}_{\dis}$ on $L^p(\R^d)$ by
\begin{align*}
    \wt{T}_{\dis}(F)(x) = \sum_{n\in\Z^d\setminus\{0\}}K(n)F(x-n), \quad F\in L^p, \, x\in\R^d. 
\end{align*} 
Let $Q={[-\frac12, \frac12)^d}
$. For $f\in \ell^p(\Z^d)$, let $F(x)=\sum_{n\in\Z^d} f(n) \ind_Q(x-n)$. Then, $\|F\|_{{L^p}}=\|f\|_{{\ell^p}}$ and
\begin{align*}
    \wt{T}_\dis (F)(n)
    &= \sum_{m\in\Z^d\setminus\{0\}}K(m) \sum_{l\in\Z^d}f(l) \ind_Q(n-m-l) \\
    &= \sum_{m\in\Z^d\setminus\{0\}} K(m) f(n-m)\\
    &= T_\dis(f)(n),
\end{align*}
which implies $\|T_{\dis}\|_{\ell^p\to \ell^p}\leq  \|\wt{T}_\dis\|_{L^p\to L^p}$. On the other hand, for any $F\in L^p$ we have, 
\begin{align*}
	\|\wt{T}_\dis F\|_{L^p}^p
    &= \int_{\R^d} \Big| \sum_{m\in\Z^d\setminus\{0\}} K(m) F(x-m)\Big|^p dx \\
    &= \sum_{n\in\Z^d\setminus\{0\}}\int_{Q} \Big| \sum_{m\in\Z^d\setminus\{0\}} K(m) F(x+n-m)\Big|^p dx \\
	&\leq \|T_\dis\|^p_{{\ell^p\to\ell^p}}\sum_{n\in\Z^d\setminus\{0\}}\int_{Q} |F(x+n)|^p dx \\
    &= \|T_\dis\|^p_{{\ell^p\to\ell^p}}\|F\|_{L^p}^p.
\end{align*}
Thus in fact, $\|\wt{T}_\dis\|_{{L^p\to L^p}}=\|T_\dis\|_{{\ell^p\to\ell^p}}$. 

Let $\varep>0$ and define $\tau_\varep F(x):=\varep^{\frac{d}{p}}F(\varep x)$. Then $\|\tau_\varep F\|_p=\|F\|_p$, for $F\in L^p(\R^d)$. Let $\wt{T}_\dis^\varep F(x)=\tau_{\frac{1}{\varep}}\wt{T}_\dis \tau_{\varep} F(x)$. Now suppose $F$ is smooth with compact support. Then 
\begin{align*}
	\wt{T}_\dis^\varep F(x)
    &= \sum_{m\in\Z^d\setminus\{0\}}K(m) F(x-\varep m) \\
    &= \sum_{m\in\Z^d\setminus\{0\}}\frac{\Omega(\varep m)}{|\varep m|^{d}}F(x-\varep m) \varep^{d}\\
    &= \frac{\varep^d}{2}\sum_{m\in\Z^d\setminus\{0\}}\frac{\Omega(\varep m)}{|\varep m|^{d}}(F(x-\varep m)-F(x+\varep m)).
\end{align*}
For each $r>0$, we have
\begin{align*}
    \lim_{\varep\to 0} \frac{\varep^d}{2}\sum_{\substack{m\in\Z^d\setminus\{0\},\\ |\varep m|>r}}\frac{\Omega(\varep m)}{|\varep m|^{d}}(F(x-\varep m)-F(x+\varep m))
    =\frac12 \int_{|y|>r}\frac{\Omega(y)}{|y|^d}(F(x-y)-F(x+y))\, dy.
\end{align*}
On the other hand, since $\Omega$ is bounded and $F$ is smooth on compact support, it follows that
\begin{align*}
    \Bigg|\frac{\varep^d}{2}\sum_{\substack{m\in\Z^d\setminus\{0\},\\ |\varep m|\le r}}\frac{\Omega(\varep m)}{|\varep m|^{d}}(F(x-\varep m)-F(x+\varep m))\Bigg|
    \le C \varep^d \sum_{\substack{m\in\Z^d\setminus\{0\},\\ |\varep m|\le r}} \frac{1}{|\varep m|^{d-1}}
    =  C r.
\end{align*}
Similarly, 
\begin{align*}
    \Bigg| \int_{|y|\le r}\frac{\Omega(y)}{|y|^d}(F(x-y)-F(x+y))\, dy\Bigg|
    \le C \int_{|y|\le r}|y|^{1-d}\, dy = Cr.
\end{align*}
Therefore, we get
\begin{align}\label{RiemSumConv}
    \lim_{\varep \to 0}\wt{T}_\dis^\varep F(x)
    &= \lim_{\varep \to 0}\frac{\varep^d}{2}\sum_{m\in\Z^d\setminus\{0\}}\frac{\Omega(\varep m)}{|\varep m|^{d}}(F(x-\varep m)-F(x+\varep m))\\\nonumber
    &=\frac12 \lim_{r\to 0}\int_{|y|>r}\frac{\Omega(y)}{|y|^d}(F(x-y)-F(x+y))\, dy\label{RiemSumConv}\\
    &=T(F)(x).\nonumber
\end{align}
By Fatou's lemma, we get
\begin{equation*}
    \|T(F)\|_{L^p} \leq \liminf_{\varep\downarrow0}\|\wt{T}_\dis^{\varep}(F)\|_{L^p}\leq \|\wt{T}_\dis(F)\|_{L^p},
\end{equation*}
which finishes the proof.
\end{proof}

The ``continuous-discrete operator'' versions have been used in several places to bound the norm of the continuous version by that of its discrete versions, see for example \cite{Laeng, Pierce}.  

\subsection{A conjecture on the $\ell^p$-norms of the discrete Riesz transforms}\label{Conjecture} 
The canonical examples of Calder\'on--Zygmund operators that satisfy the assumptions of both Proposition \ref{thm:RCZ} and Theorem \ref{homeker} are the classical Riesz transforms on $\R^d$ already defined in \eqref{classicalCRT-1}.

The Riesz transforms arise naturally from the Poisson semigroup and its connection to the Laplacian.  That is, if we let $P_yf(x)$ be the convolution of the function $f$ with the Poisson kernel $p(\cdot\, , y)$ as in \eqref{PoissonK}, then in fact, 
\begin{align}\label{SemigroupRiesz} 
    R^{(k)}f(x)=\int_0^{\infty} \frac{\partial}{\partial x_k}P_y f(x)\, dy=\frac{\partial }{\partial x_k}(-\Delta)^{-1/2}f(x).
\end{align}
With this interpretation, the Riesz transforms can be defined in a variety of analytic and geometric settings, including manifolds, Lie groups, and Wiener space. We briefly recall here the Gundy--Varopoulos \cite{GV79} representation of $R^{(k)}$, referring the reader to \cite{Ban86, Ban} for details and applications. Let $B_t$ be the standard Brownian motion  in the upper half-space $\R_{+}^{d+1}$  starting at the point $(0, y)$ and $\tau$ its exit time. Consider the conditional expectations operators 
\begin{equation}\label{GunVorRiesz1}
    \rE_{(0, y)}{\Big[}\int_0^{\tau} \bH^{(k)}\nabla U_f(B_s)\cdot dB_s \Big| B_{\tau}=x{\Big]},
\end{equation} 
where $\bH^{(k)}$ are the matrices in \eqref{HRiesz}, $U_f(x, y)=P_yf(x)$. Just as we do in Section \ref{ProjOP} for the discrete case, under the assumption that $f$ is sufficiently smooth, the quantity in \eqref{GunVorRiesz1} converges pointwise to $R^{(k)}f(x)$, as $y\to \infty$.

We remark here that verifying the convergence of \eqref{GunVorRiesz1} to the Riesz transforms is much simpler than the corresponding convergence results in Section \ref{ProjOP}; see for example \cite[p. 417]{Ban}. It is also worth mentioning here that the original Gundy--Varopoulos paper used the so called ``background radiation'' process in the construction. That the background radiation is not needed was shown in \cite{Ban86}.

It is proved in Laeng~\cite{Laeng09} that the $L^p$-norm of the truncated Hilbert transform $H_\varep$ 
\begin{align*}
    H_\varep f(x)=\frac{1}{\pi} \int_{|z|>\varep} \frac{f(x-z)}{z}\,dz,
\end{align*}
coincides with that of the Hilbert transform $H$. That is, $\|H_\varep\|_{p\to p}= \cot\left(\frac{\pi}{2p^\ast}\right)$ for every $\varep>0$.

The methods of rotations can be used to obtain the same for the truncated Riesz transforms. That is, let $K_\varep^{(k)}(z)=K^{(k)}(z)\ind_{\{|z|> \varep\}}$ and $R^{(k)}_\varep f(x)=K_\varep^{(k)}*f(x)$ be the truncated Riesz transform. Then,  
\begin{equation}\label{tranRiesz}
    \|R^{(k)}_\varep\|_{L^p\to L^p} = \cot\left(\frac{\pi}{2p^\ast}\right),
\end{equation} 
for all $\varep>0$. 
Combining \eqref{SCRiesz} with Fatou's Lemma, it suffices to show the upper inequality 
\begin{align}\|R^{(k)}_\varep\|_{L^p\to L^p} \leq \cot\left(\frac{\pi}{2p^\ast}\right).
\end{align} 
This follows from the method of rotations (\cite{Graf}*{Equation (4.2.17)}) applied with  $\Omega(y)=c_d\frac{y_k}{|y|}$ and the additional observation that 
\begin{align*}
    \int_{\bS^{d-1}}|\Omega(\theta)|\, d\theta
    =c_d \int_{\bS^{d-1}}|\theta_k|\, d\theta
    =c_d \frac{2\pi^{\frac{d-1}{2}}}{\Gamma\left(\frac{d+1}{2}\right)}
    =\frac{2}{\pi}.
\end{align*}

It follows from this and  Propositions  \ref{thm:RCZ} and Theorem  \ref{homeker}  that \eqref{prop1} holds.  That is, we have the following 

\begin{corollary}\label{bestweknow} 
For $1<p<\infty$, $k=1, 2, \dots, d$, we have 
$$\cot\left(\frac{\pi}{2p^{\ast}}\right)\leq \|R_{\dis}^{(k)}\|_{\ell^p\to \ell^p}\leq \cot\left(\frac{\pi}{2p^{\ast}}\right)+C_d,$$  
where $C_d$ is a dimensional constant. 
\end{corollary}

\begin{conjecture}\label{SharpRiesz} 
For all $d>1$, $1<p<\infty$, $k=1, \dots, d$, 
\begin{equation}\label{DRTConj} 
	\|R^{(k)}_\dis\|_{\ell^p\to \ell^p}= \cot\left(\frac{\pi}{2p^{\ast}}\right).
\end{equation}
\end{conjecture}

A weaker but also interesting result would be  

\begin{problem}\label{WeakerProb}  Show that 
\begin{equation}
\|R^{(k)}_\dis\|_{\ell^p} \leq C_p, 
\end{equation}
where $C_p$ is independent of $d$.   Even more, it would be interesting to show that the vector of discrete Riesz transforms also has  a bound independent of the dimension, as it was first proved by Stein in \cite{SteSome} for the continuous version. That is, if 
\begin{align}R_{\dis}f(n)=\left(\sum_{k=1}^d |R^{(k)}_\dis f(n)|^2\right)^{1/2}, 
\end{align}
then 
\begin{align}\label{vector} \|R_{\dis}f\|_{\ell^p}\leq C_p\|f\|_{\ell^p},
\end{align}
where $C_p$ is independent of $d$. 
\end{problem} 

From the representation of the kernel $\bK_A(n,m)$ in~\eqref{kernelrep} and Remark~\ref{rmk:kernelsymmetry}, we have $T_{\bH^{(k)}} = 2T_{A^{(k)}}$ where $A^{(k)}_{ij}=1$ for $i=d+1$, $j=k$ and otherwise 0. By~\cite[Theorem 4.2]{BanWang}, we obtain that
\[
    \left\|\left(\sum_{k=1}^d |T_{\bH^{(k)}}f|^{2}\right)^{1/2}\right\|_{\ell^p}\le 2(p^\ast -1)\|f\|_{\ell^p}.
\]

As for the best constant \eqref{vector} for the vector, this is a well known problem even in the continuous case.  The original proof with bound independent of dimension in \cite{SteSome}, and different subsequent proofs shortly after, did not even give the sharp asymptotic behavior which is  $\mathcal{O}(1/p-1)$, as $p\to 1$ and $\mathcal{O}(p)$, as $p\to\infty$.  The first proof providing this behavior was given in \cite{BanGoodLamb} with subsequent improvements with explicit constants in \cite[Corollary 4.21]{BanWang} and \cite[Theorem 5.1]{IwaMar}.

These problem are also motivated by the remark in \cite{MSW}*{pg.\,193} already discussed in connection to inequality \eqref{MSW-Ineq}.

\section{Discrete Riesz transforms and their probabilistic counterparts}
\subsection{Probabilistic Discrete Riesz Transforms and their norms}\label{PDR} 
The proof of the Conjecture \ref{SharpRiesz} for $d=1$ in \cite{BanKwa} rests on the probabilistic construction of the operators in Section \ref{general-d} for $d=1$ applied to the operator $T_{\bH}^{w}$ as in \eqref{projTH-gen-2} with the matrix 
\begin{equation}\label{matrixH} 
    \bH=  \left[\begin{matrix} 0 & -1 \\ 1 & 0 \end{matrix}\right],  
\end{equation} 
which is orthogonal and of norm 1. Motivated by this and the Gundy--Varopoulos \cite{GV79} probabilistic representation of the Riesz transforms $R^{(k)}$ on $\R^d$, $d\geq 1$ as in \eqref{GunVorRiesz1} and its many variants studied over the years (see for example \cite{Ban, BBL20} and the many references therein), we consider the operators $T_{\bH^{(k)}}$, where for  each $k=1,2,\cdots,d$, the $(d+1)\times (d+1)$ matrix  $\bH^{(k)}$ is given by \eqref{HRiesz}. As already noted $\bH^{(k)}$ is orthogonal, $\|\bH^{(k)}\|=1$.  This gives the operators $T_{\bH^{(k)}}$, $k=1,\dots d$, which we call the \textit{``probabilistic discrete Riesz transforms''}. By Theorem \ref{thm:TaIntRep}, their kernels are given by 
\begin{equation}
\bK_{\bH^{(k)}}(n,m)=\int_{\R^{d}}\int_0^{\infty}2y h(x,y)\bH^{(k)}\nabla \left(\frac{p_m(x, y)}{h(x, y)}\right) \cdot  \nabla \left(\frac{p_n(x,y)}{h(x,y)}\right) \,dydx.
\end{equation} 
Using the fact that $h(x+m,y)=h(x,y)$ for all $m\in\Z^d$, a change of variables shows that 
\begin{equation}\label{ExplicitK-H}
    \bK_{\bH^{(k)}}(n,m)=\bK_{\bH^{(k)}}(n-m)
    =\int_{\R^{d}}\int_0^{\infty}2y h(x,y)\bH^{(k)}\nabla \left(\frac{p_0(x, y)}{h(x, y)}\right) \cdot  \nabla \left(\frac{p_{n-m}(x, y)}{h(x, y)}\right) \,dydx.
\end{equation}

Note that when  $d=1$ the matrices in \eqref{HRiesz} reduce to the matrix $\bH$ in \eqref{matrixH} which gives the probabilistic discrete Hilbert transform $T_{\bH}$.  

The kernels $\bK_{\bH^{(k)}}$ enjoy the following  properties that will be used below in several computations. 
For $n=(n_1,n_2,\cdots,n_d)\in\Z^d$, define 
\begin{align*}
	\wt{n}=(n_1,n_2,\cdots,-n_k,\dots n_d).
\end{align*}
It follows from \eqref{ExplicitK-H} (or from~\eqref{eq:rep_jn} below) that $\bK_{\bH^{(k)}}(n)=-\bK_{\bH^{(k)}}(\wt{n})$. Thus we have $\bK_{\bH^{(k)}}(n)=0$, if $n_k=0$, and in particular, $\bK_{\bH^{(k)}}(0)=0$. It also follows that $\bK_{\bH^{(k)}}(n)={-}\bK_{\bH^{(k)}}(-n)$. These properties of $\bK_{\bH^{(k)}}$ will be used below in several computations.   Because $\bK_{\bH^{(k)}}(0)=0$ we often write the sum in the convolution over $\Z$ and not over $\Z\setminus\{0\}$.

By~\eqref{projTH-gen-2} of Theorem~\ref{Proj} and \eqref{kernelrep} of Theorem~\ref{thm:TaIntRep}, we obtain the following
\begin{theorem}\label{ProbRbound} 
Suppose $f\in\ell^p(\Z^d)$, $1<p<\infty$. Set 
\begin{equation}\label{probRiesz} 
    T_{\bH^{(k)}}(f)(n)=\sum_{m\in\Z^d}\bK_{\bH^{(k)}}(n-m)f(m).
\end{equation}
Then,
\begin{align}\label{cot-bound-ProbRiesz}
	\|T_{\bH^{(k)}}f\|_{\ell^p}\leq \cot(\frac{\pi}{2p^{\ast}})\|f\|_{\ell^{p}}, \quad k=1, 2, \dots, d.
\end{align}	
\end{theorem}

In the following proposition, we derive a different integral representation for the kernel $\bK_{\bH^{(k)}}(n)$ which will provide a relationship between the operators  $T_{\bH^{(k)}}$ and the CZ discrete Riesz transforms $R^{(k)}_{\dis}$. As we shall see, this representation allows us to prove that the $\ell^p$-bound in~\eqref{cot-bound-ProbRiesz} is best possible; see Theorem~\ref{bestpro-dis-riesz}.

\begin{theorem}\label{DisProbVSDiscRies}
We have
\begin{align}\label{eq:rep_jn}
    \bK_{\bH^{(k)}}(n)
    &=\left(\int_{\bR^d}\int_0^{\infty} \frac{U_n(x,y)}{h(x,y)}\, dydx\right) \ind_{\{|n|\geq 1\}}(n)\\
    &=\left(4\int_{\bR^d}\int_0^{\infty} \frac{S_n(x,y)}{h(x,y)}\, dydx-3\int_{\R^d}\int_0^{\infty} \frac{T_n(x,y)}{h(x,y)}\, dydx\right)\ind_{\{|n|\geq 1\}}(n)\nonumber\\
	&=-4\int_{\R^{d}}\int_{0}^{\infty}\frac{1}{h}\parder{{p_0}}{x_k}\parder{}{y}(yp_n)\,dydx
\end{align}
where 
\begin{align*}
	S_n(x,y)&=\frac{2c_d^2(d+1)x_k y^2}{(|x|^2+y^2)^{\frac{d+3}{2}}(|x-n|^2+y^2)^{\frac{d+1}{2}}},\\
	T_n(x,y)&=\frac{4c_d^2(d+1)^2 x_k y^4}{3(|x|^2+y^2)^{\frac{d+3}{2}}(|x-n|^2+y^2)^{\frac{d+3}{2}}},
\end{align*}
and $U_n(x,y)=4S_n(x,y)-3T_n(x,y)$.
\end{theorem}

\begin{proof} 
{
By~\eqref{ExplicitK-H}, we have
\begin{align*}
    &\bK_{\bH^{(k)}}(n)\\
    &=\int_{\R^{d}}\int_0^{\infty}2y h\bH^{(k)}\nabla \left(\frac{p_0}{h}\right) \cdot  \nabla \left(\frac{p_{n}}{h}\right) \,dydx\\
    &=\int_{\R^{d}}\int_0^{\infty}2y h\left(
        \parder{}{x_k} \left(\frac{p_0}{h}\right)   \parder{}{y} \left(\frac{p_{n}}{h}\right) 
        -\parder{}{y} \left(\frac{p_0}{h}\right)   \parder{}{x_k} \left(\frac{p_{n}}{h}\right) 
        \right)
        \,dydx\\
    &=
    \int_{\R^{d}}\int_0^{\infty}2y h\left(
        \frac{1}{h^2}\parder{p_0}{x_k}\parder{p_n}{y}
        -\frac{p_0}{h^3}\parder{h}{x_k}\parder{p_n}{y}
        -\frac{p_n}{h^3}\parder{p_0}{x_k}\parder{h}{y}
        +\frac{p_0p_n}{h^4}\parder{h}{x_k}\parder{h}{y}
        \right)
        \,dydx\\
    &\qquad - \int_{\R^{d}}\int_0^{\infty}2y h\left(
        \frac{1}{h^2}\parder{p_0}{y}\parder{p_n}{x_k}
        -\frac{p_0}{h^3}\parder{h}{y}\parder{p_n}{x_k}
        -\frac{p_n}{h^3}\parder{p_0}{y}\parder{h}{x_k}
        +\frac{p_0p_n}{h^4}\parder{h}{y}\parder{h}{x_k}
        \right)
        \,dydx\\
    &=
    \int_{\R^{d}}\int_0^{\infty}2y h\left(
        \frac{1}{h^2}\parder{p_0}{x_k}\parder{p_n}{y}
        -\frac{p_0}{h^3}\parder{h}{x_k}\parder{p_n}{y}
        +\frac{p_0}{h^3}\parder{h}{y}\parder{p_n}{x_k}
        \right)
        \,dydx\\
    &\qquad - \int_{\R^{d}}\int_0^{\infty}2y h\left(
        \frac{1}{h^2}\parder{p_0}{y}\parder{p_n}{x_k}
        -\frac{p_n}{h^3}\parder{p_0}{y}\parder{h}{x_k}
        +\frac{p_n}{h^3}\parder{p_0}{x_k}\parder{h}{y}
        \right)
        \,dydx.
\end{align*}
By the change of variable $x=n-x'$, we simplify
\begin{align*}
    \bK_{\bH^{(k)}}(n)
	=4\int_{\R^{d}}\int_{0}^{\infty}\left(
        \frac{y}{h}\parder{{p_0} }{x_k}\parder{p_n}{y} 
        +\frac{y{p_0} }{h^2}\parder{h}{y}\parder{p_n}{x_k} 
        -\frac{y{p_0} }{h^2}\parder{h}{x_k}\parder{p_n}{y} 
    \right)\,dydx.
\end{align*}
It follows from integration by parts that
\begin{align*}
    \int_0^\infty 
        \frac{y{p_0} }{h^2}\parder{h}{y}\parder{p_n}{x_k} 
    \, dy
    &= -\int_0^\infty 
    h\parder{}{y}\left(
        \frac{y{p_0} }{h^2}\parder{p_n}{x_k} 
        \right)
    \, dy\\
    &= -\int_0^\infty 
    \left(
        \frac{{p_0} }{h}\parder{p_n}{x_k} 
        +\frac{y }{h}\parder{p_0}{y}\parder{p_n}{x_k} 
        -\frac{2y{p_0} }{h^2}\parder{h}{y}\parder{p_n}{x_k} 
        +\frac{y{p_0} }{h}\frac{\partial^2 p_n}{\partial y \partial x_k} 
        \right)
    \, dy\\
    &= \int_0^\infty 
    \left(
        \frac{{p_0} }{h}\parder{p_n}{x_k} 
        +\frac{y }{h}\parder{p_0}{y}\parder{p_n}{x_k} 
        +\frac{y{p_0} }{h}\frac{\partial^2 p_n}{\partial y \partial x_k} 
        \right)
    \, dy.
\end{align*}
Similarly, 
\begin{align*}
    -\int_{\R}
        \frac{y{p_0} }{h^2}\parder{h}{x_k}\parder{p_n}{y} 
    \, dx_k
    &= \int_{\R}
    h\parder{}{x_k}\left(
        \frac{y{p_0} }{h^2}\parder{p_n}{y} 
        \right)
    \, dx_k\\
    &= \int_{\R}
    \left(
        \frac{y }{h}\parder{p_0}{x_k}\parder{p_n}{y} 
        -\frac{2y{p_0} }{h^2}\parder{h}{x_k}\parder{p_n}{y} 
        +\frac{y{p_0} }{h}\frac{\partial^2 p_n}{\partial y \partial x_k} 
        \right)
    \, dx_k\\
    &=- \int_{\R}
    \left(
        \frac{y }{h}\parder{p_0}{x_k}\parder{p_n}{y} 
        +\frac{y{p_0} }{h}\frac{\partial^2 p_n}{\partial y \partial x_k} 
        \right)
    \, dx_k.
\end{align*}
Combining these, we get 
\begin{align*}
    \bK_{\bH^{(k)}}(n)
    &=4\int_{\R^{d}}\int_{0}^{\infty}\left(
        \frac{y}{h}\parder{{p_0} }{x_k}\parder{p_n}{y} 
        +\frac{y{p_0} }{h^2}\parder{h}{y}\parder{p_n}{x_k} 
        -\frac{y{p_0} }{h^2}\parder{h}{x_k}\parder{p_n}{y} 
    \right)\,dydx\\
    &=4\int_{\R^{d}}\int_{0}^{\infty}\bigg(
        \frac{y}{h}\parder{{p_0} }{x_k}\parder{p_n}{y} 
        +\frac{{p_0} }{h}\parder{p_n}{x_k} 
        +\frac{y }{h}\parder{p_0}{y}\parder{p_n}{x_k} 
        +\frac{y{p_0} }{h}\frac{\partial^2 p_n}{\partial y \partial x_k} \\
    &\qquad\qquad -\frac{y }{h}\parder{p_0}{x_k}\parder{p_n}{y} 
        -\frac{y{p_0} }{h}\frac{\partial^2 p_n}{\partial y \partial x_k} 
    \bigg)\,dydx\\
    &=4\int_{\R^{d}}\int_{0}^{\infty}\bigg(
        \frac{{p_0} }{h}\parder{p_n}{x_k} 
        +\frac{y }{h}\parder{p_0}{y}\parder{p_n}{x_k} 
    \bigg)\,dydx\\
    &
    =
    4\int_{\R^{d}}\int_{0}^{\infty} 
    \frac{1}{h}\parder{p_n}{x_k} \parder{}{y}\left(yp_0\right)
    \,dydx
    =
    -4\int_{\R^{d}}\int_{0}^{\infty} 
    \frac{1}{h}\parder{p_0}{x_k} \parder{}{y}\left(yp_n\right)
    \,dydx
\end{align*}
We note that on the numerator of~\eqref{Poissver} we have $(|x-n|^2-dy^2)$. Splitting this into $|x-n|^2+y^2$ and $-(d+1)y^2$, then the first term can be absorbed in $S_n(x,y)$ and the other in $T_n(x,y)$. Thus, integration by parts with~\eqref{Poisshor} and~\eqref{Poissver} gives that 

\begin{align}\label{byparts} 
    \bK_{\bH^{(k)}}(n)
	& =8c_d^2(d+1)\int_{\R^d}\int_0^\infty \frac{1}{h(x,y)}\frac{x_k y^2}{(|x|^2+y^2)^{\frac{d+3}{2}}(|x-n|^2+y^2)^{\frac{d+1}{2}}} \, dydx\\
	&\qquad\qquad -4c_d^2 (d+1)^2 \int_{\R^d}\int_0^\infty \frac{1}{h(x,y)}\frac{x_k y^4}{(|x|^2+y^2)^{\frac{d+3}{2}}(|x-n|^2+y^2)^{\frac{d+3}{2}}} \, dydx. \nonumber
\end{align}
}
\end{proof}

\begin{remark} 
When $d=1$, \eqref{Explicitd=1} gives  
$h(x,y) =\frac{\sinh(2\pi y)}{\cosh(2\pi y)-\cos(2\pi x)},$
and 
 \begin{align*} 
	S_n(x,y)
	&=\frac{4}{\pi^2}\frac{xy^2}{(x^2+y^2)^2 ((x-n)^2+y^2)},\\
	T_n(x,y)
	&=\frac{16}{3\pi^2}\frac{x y^4}{(x^2+y^2)^2 ((x-n)^2+y^2)^2},
\end{align*}
From this we have that for each $n\in\Z\setminus\{0\}$ \begin{align*}
    \frac{\pi^2}{16} K_{\bH}(n)
	&= \frac{\pi^2}{4}\iint\frac{S_n(x,y)}{h(x,y)}\,dxdy-\frac{3\pi^2}{16}\iint\frac{T_n(x,y)}{h(x,y)}\,dxdy \\
	&=\int_0^\infty	\coth(2\pi y)y^2 \left(\int_{\R}\frac{x}{(x^2+y^2)^2 ((x-n)^2+y^2)}\,dx\right) dy\\
	&\qquad -\int_0^\infty	\frac{y^2}{\sinh(2\pi y)} \left(\int_{\R}\frac{x\cos(2\pi x)}{(x^2+y^2)^2 ((x-n)^2+y^2)}\,dx\right) dy\\
	&\qquad -\int_0^\infty	\coth(2\pi y)y^4 \left(\int_{\R}\frac{x}{(x^2+y^2)^2 ((x-n)^2+y^2)^2}\,dx\right) dy\\
	&\qquad +\int_0^\infty	\frac{y^4}{\sinh(2\pi y)} \left(\int_{\R}\frac{x\cos(2\pi x)}{(x^2+y^2)^2 ((x-n)^2+y^2)^2}\,dx\right) dy.
\end{align*}
The inner  integrals with respect to $x$ can all be computed using  partial fraction decompositions.
Indeed, for $n\in\Z\setminus\{0\}$ we have:  

\begin{align*}
    \int_{\R}\frac{x}{(x^2+y^2)^2 ((x-n)^2+y^2)}\,dx
    =\frac{2\pi n }{y(n^2+4y^2)^2}, 
\end{align*}

\begin{align*}
    \int_{\R}\frac{x\cos(2\pi x)}{(x^2+y^2)^2 ((x-n)^2+y^2)}\,dx
    =\frac{2\pi n  e^{-2\pi y}}{y(n^2+4y^2)^2}-\frac{2\pi^2 e^{-2\pi y}}{n(n^2+4y^2)}, 
\end{align*}

\begin{align*}
	\int_{\R}\frac{x}{(x^2+y^2)^2 ((x-n)^2+y^2)^2}\,dx
	=\frac{\pi n (n^2+20y^2)}{2y^3(n^2+4y^2)^3}, 
\end{align*}

\begin{align*}
	\int_{\R}\frac{x\cos(2\pi x)}{(x^2+y^2)^2 ((x-n)^2+y^2)^2}\,dx
	=\frac{e^{-2\pi y}\pi n(n^2+20y^2)}{2y^3(n^2+4y^2)^3} + \frac{e^{-2\pi y}\pi^2 (n^2-4y^2)}{ny^2(n^2+4y^2)^2}. 
\end{align*}

From these we obtain 	
\begin{align*}
  K_{\bH}(n)
    &= \frac{1}{\pi n}\left( 1+ 16\pi \int_0^\infty	\frac{y^2 e^{-2\pi y}(3n^2+4y^2)}{\sinh(2\pi y)(n^2+4y^2)^2}\,dy \right)\nonumber\\
    &=\frac{1}{\pi n}\left( 1+ \int_0^\infty \frac{2 y^3}{(y^2 + \pi^2 n^2) \sinh^2 y} \, dy \right),
\end{align*}
where the last equality follows from integrations by parts with 
\begin{align*}
	f(y)
    =\frac{e^{-y}}{\sinh(y)},\qquad g(y)=\frac{2y^3}{y^2+\pi^2 n^2}. 
\end{align*}
This leads to the expression 
\begin{align}\label{kerneld} 
    \bK_{\bH}(n)&= \frac{1}{\pi n}\left(1+\int_0^\infty\frac{2y^3}{(y^2+\pi^2 n^2)\sinh^2(y)}\, dy\right)\,\ind_{\Z \setminus \{0\}}(n)\\
    &=\frac{1}{\pi n}\left(1+\int_0^\infty\frac{2y^3}{(y^2+\pi^2 n^2)\sinh^2(y)}\, dy\right)\,\ind_{\{|n|\geq 1\}}(n).\nonumber
\end{align}
This formula was computed in \cite{BanKwa} with a slightly different approach. We will return to this formula below in Section \ref{ProbContRiesz}.  
\end{remark} 

Recall that $R^{(k)}_{\dis}$ 
are the CZ discrete Riesz transform given by
\begin{align}\label{ClaRiesz} 
    R^{(k)}_{\dis}f(m) =c_d\sum_{\substack{n\in\Z^d\setminus\{0\}}} \frac{n_{k}}{|n|^{d+1}}f(m-n) =\sum_{n\in Z^d} 
    K_{R^{(k)}_{\dis}}(n)
    f(m-n),
\end{align} 
where 
\begin{align*}
    K_{R^{(k)}_{\dis}}(n)
    =c_d\frac{{n_k}}{|n|^{d+1}} \ind_{\Z^d\setminus \{0\}}(n),\qquad c_d=\frac{\Gamma(\frac{d+1}{2})}{\pi^{\frac{d+1}{2}}}.
\end{align*}

The following alternative representation for  $K_{R^{(k)}_{\dis}}$  will be frequently used below. 
\begin{proposition}\label{DisRieszKern1}
We have
\begin{align*}
    K_{R^{(k)}_{\dis}}(n)
    =\left(\int_{\R^d}\int_0^{\infty} U_n(x,y)\, dydx\right)\ind_{\{|n|\geq 1\}}(n)&=\left(\int_{\R^d}\int_0^{\infty} S_n(x,y)\, dydx\right)\ind_{\{|n|\geq 1\}}(n)\\
   & =\left(\int_{\R^d}\int_0^{\infty}T_n(x,y)\, dydx\right)\ind_{\{|n|\geq 1\}}(n).
\end{align*}
\end{proposition}

\begin{proof}
Let $N=\frac{d+3}{2}$. By the definition of Gamma function, 
\begin{align*}
	&\int_{0}^{\infty}\int_{\R^{d}}\frac{x_{k}y^{2}}{(|x|^{2}+y^{2})^{N}(|x-n|^{2}+y^{2})^{N-1}}\,dxdy\\
	&= \frac{1}{\Gamma(N)\Gamma(N-1)}\int_{0}^{\infty}\int_{\R^{d}}\int_{0}^{\infty}\int_{0}^{\infty}
		x_{k}y^{2}u^{N-1}v^{N-2}e^{-u(|x|^{2}+y^{2})-v(|x-n|^{2}+y^{2})}\,dudvdxdy\\
	&= \frac{1}{\Gamma(N)\Gamma(N-1)}\int_{0}^{\infty}\int_{\R^{d}}\int_{0}^{\infty}\int_{0}^{\infty}
		x_{k}y^{2}u^{N-1}v^{N-2}e^{-(u+v)|x-\frac{v}{u+v}n|^{2}-\frac{uv}{u+v}|n|^2 -(u+v)y^{2}}\,dudvdxdy.
\end{align*}
Since we have
\begin{align*}
	\int_{\R^d}x_ke^{-(u+v)|x-\frac{v}{u+v}n|^{2}}\, dx
	&= \int_{\R^d}\left(x_k+\frac{v}{u+v}n_k\right)e^{-(u+v)|x|^{2}}\, dx\\
	&= \frac{v}{u+v}n_k\int_{\R^d}e^{-(u+v)|x|^{2}}\, dx\\
	&= \frac{\pi^{\frac{d}{2}}v}{(u+v)^{1+\frac{d}{2}}}n_k
\end{align*}
and
\begin{equation*}
	\int_0^\infty y^2 e^{-(u+v)y^2}\, dy
	= (u+v)^{-\frac{3}{2}}\frac{\sqrt{\pi}}{4},
\end{equation*}
it follows from Fubini's theorem that
\begin{align*}
	&\int_{0}^{\infty}\int_{\R^{d}}\frac{x_{k}y^{2}}{(|x|^{2}+y^{2})^{N}(|x-n|^{2}+y^{2})^{N-1}}\,dxdy\\
	&= \frac{\pi^{\frac{d+1}{2}}n_k}{4\Gamma(N)\Gamma(N-1)}\int_{0}^{\infty}\int_{0}^{\infty} \frac{u^{N-1}v^{N-1}}{(u+v)^{N+1}}e^{-\frac{uv}{u+v}|n|^2 }\,dudv.
\end{align*}
Substituting $t = \frac{u}{u + v}$, $s = (u + v) |n|^2$, $du dv = |n|^{-4} s ds dt$, we obtain
\begin{align*}
	&{\int_{0}^{\infty}\int_{\R^{d}}\frac{x_{k}y^{2}}{(|x|^{2}+y^{2})^{N}(|x-n|^{2}+y^{2})^{N-1}}\,dxdy}\\
	&= \frac{\pi^{\frac{d+1}{2}}}{4\Gamma(N)\Gamma(N-1)} \frac{n_k}{|n|^{d+1}} \int_{0}^{1}\int_{0}^{\infty}s^{N-2}t^{N-1}(1-t)^{N-1}e^{-st(1-t)}\,dsdt\\
	&= \frac{1}{2(d+1)c_d} \frac{n_k}{|n|^{d+1}}.
\end{align*}
Similarly, we have
\begin{align*}
	&\int_{0}^{\infty}\int_{\R^{d}} \frac{x_k y^4}{(|x|^2+y^2)^N(|x-n|^2+y^2)^N} \,dxdy\\
	&= \frac{1}{\Gamma(N)^{2}}\int_{0}^{\infty}\int_{\R^{d}}\int_{0}^{\infty}\int_{0}^{\infty}x_k y^4u^{N-1}v^{N-1}e^{-(u+v)|x-\frac{v}{u+v}n|^{2}-\frac{uv}{u+v}|n|^2 -(u+v)y^{2}}\,dudvdxdy\\
	&= \frac{3\pi^{\frac{d+1}{2}}}{8\Gamma(N)^{2}}\frac{n_k}{|n|^{d+1}} \int_{0}^{\infty}\int_{0}^{\infty}\frac{u^{N-1}v^{N}}{(u+v)^{N+2}}e^{-\frac{uv}{u+v} }\,dudv\\
	&=\frac{3}{4c_d(d+1)^2}\frac{n_k}{|n|^{d+1}}.
\end{align*}
\end{proof}

We are now ready to prove that the upper bound in Theorem~\ref{ProbRbound} is sharp. 

\begin{theorem}\label{bestpro-dis-riesz} 
The $\ell^p$-bound of $T_{\bH^{(k)}}$ in Theorem~\ref{ProbRbound} is best possible. That is, for all $d\geq 1$,  $k=1, \dots d$, $\|T_{\bH^{(k)}}\|_{\ell^p\to \ell^p} = \cot(\frac{\pi}{2p^{\ast}})$.  
\end{theorem}

The result will follow from the next two lemmas. 
\begin{lemma}\label{lem:DiscreteAsymptotics} 
With the notation introduced earlier in this section, we have
$$
	\lim_{|n| \to \infty} |n|^{d} |\bK_{\bH^{(k)}}(n) - K_{R_{\dis}^{(k)}}(n) | = 0.
$$
\end{lemma}

\begin{proof}
Recall that
$$
	\bK_{\bH^{(k)}}(n)
	= \int_{\R^{d}} \int_0^\infty \frac{U_n(x, y)}{h(x, y)} \, dy dx , \qquad K_{R_{\dis}^{(k)}}(n)
	= \int_{\R^{d}} \int_0^\infty U_n(x, y) dy dx ,
$$
where
$$
	U_n(x, y)
	= -4 \parder{p_0}{x_k}(x, y) \, \parder{}{y} (y p_n(x, y)) \, .
$$
Observe that $p_0(x, y) = \varep^{d + 1} p_0(\varep x, \varep y)$ and $p_n(x, y) = \varep^{d + 1} p_{\varep n}(\varep x, \varep y)$. Therefore,
$$
	U_n(x, y)
	= \varep^{2 d + 1} U_{\varep n}(\varep x, \varep y) ;
$$
here and below we abuse the notation and we allow $n$ in $p_n$ and $U_n$ to be an arbitrary vector in $\R^{d}$. It follows that
\begin{align*}
	\bK_{\bH^{(k)}}(n)
	& = \int_{\R^{d}} \int_0^\infty \frac{U_n(x, y)}{h(x, y)} \, dy dx \\
	& = \varep^{2 d + 1} \int_{\R^{d}} \int_0^\infty \frac{U_{\varep n}(\varep x, \varep y)}{h(x, y)} \, dy dx \\
	& = \varep^{d} \int_{\R^{d}} \int_0^\infty \frac{U_{\varep n}(x, y)}{h(\frac{1}{\varep} x, \frac{1}{\varep} y)} \, dy dx .
\end{align*}
If we choose $\varep = \tfrac{1}{|n|}$, we find that
$$
	|n|^{d} \bK_{\bH^{(k)}}(n)
	= \int_{\R^{d}} \int_0^\infty \frac{U_{n / |n|}(x, y)}{h(|n| x, |n| y)} \, dy dx .
$$
Similarly,
$$
	|n|^{d} K_{R_{\dis}^{(k)}}(n) = \int_{\R^{d}} \int_0^\infty U_{n / |n|}(x, y) dy dx .
$$
Therefore,
\begin{equation}\label{eq:unformlyintegrableconvergence}
	|n|^{d} |\bK_{\bH^{(k)}}(n) - K_{R_{\dis}^{(k)}}(n) |
	\le \int_{\R^{d}} \int_0^\infty |U_{n / |n|}(x, y)| \times \biggl|\frac{1}{h(|n| x, |n| y)} - 1\biggr| dy dx .
\end{equation}
By the estimate~\eqref{eq:hupperbound2}, we have
\begin{equation*}
	\frac{y}{\sqrt{x^2 + y^2}} \, {\biggl(}\frac{1}{h(|n| x, |n| y)} {{}+{}} 1{\biggr)} \leq C_1
\end{equation*}
for a constant $C_1$ that depends only on the dimension $d$, and by Lemma~\ref{lem:hlimit}, the left-hand side converges pointwise to zero as $|n| \to \infty$. On the other hand, by the explicit expression for $U_{n / |n|}$ given in Theorem~\ref{DisProbVSDiscRies}, we have
\begin{align*}
	\frac{\sqrt{x^2 + y^2}}{y} \, |U_{n / |n|}(x, y)| 
	& \le \frac{C_2 \sqrt{x^2 + y^2}}{y} \, \frac{|x_k| y^2}{(|x|^2 + y^2)^{\frac{d + 3}{2}} (|x - \frac{n}{|n|}|^2 + y^2)^{\frac{d + 1}{2}}} \\
	& \qquad + \frac{C_3 \sqrt{x^2 + y^2}}{y} \, \frac{|x_k| y^4}{(|x|^2 + y^2)^{\frac{d + 3}{2}} (|x - \frac{n}{|n|}|^2 + y^2)^{\frac{d + 3}{2}}} \\
	& = C_2 \, \frac{|x_k| y}{(|x|^2 + y^2)^{\frac{d + 2}{2}} (|x - \frac{n}{|n|}|^2 + y^2)^{\frac{d + 1}{2}}} \\
	& \qquad + C_3 \, \frac{|x_k| y^3}{(|x|^2 + y^2)^{\frac{d + 2}{2}} (|x - \frac{n}{|n|}|^2 + y^2)^{\frac{d + 3}{2}}} \\
	& \le C_4 \, \frac{y}{(|x|^2 + y^2)^{\frac{d + 1}{2}} (|x - \frac{n}{|n|}|^2 + y^2)^{\frac{d + 1}{2}}}
\end{align*}
for some constants $C_2$, $C_3$ and $C_4$ that again depend only on $d$. We have thus shown that
\begin{equation}\label{eq:Unbound}
	|U_{n / |n|}(x, y)| \times {\biggl(}\frac{1}{h(|n| x, |n| y)} {{}+{}} 1{\biggr)} \leq C_1 C_4 \, \frac{y}{(|x|^2 + y^2)^{\frac{d + 1}{2}} (|x - \frac{n}{|n|}|^2 + y^2)^{\frac{d + 1}{2}}} \, ,
\end{equation}
and additionally, as $|n| \to \infty$, the left-hand side converges point-wise to zero. Observe that if we denote by $\mathbb{O}_n$ an orthogonal transformation of $\R^{d}$ which maps $\frac{n}{|n|}$ to $e_1 = (1, 0, 0, \ldots, 0)$, then the above estimate takes form
\begin{equation*}
	|U_{n / |n|}(\mathbb{O}_n x, y)| \times \biggl|\frac{1}{h(|n| \mathbb{O}_n x, |n| y)} - 1\biggr| \leq C_1 C_4 \, \frac{y}{(|x|^2 + y^2)^{\frac{d + 1}{2}} (|x - e_1|^2 + y^2)^{\frac{d + 1}{2}}} \, ,
\end{equation*}
and the right-hand side no longer depends on $n$. Since the right-hand side is integrable, by the dominated convergence theorem we find that
\begin{align*}
	& \lim_{|n| \to \infty} \int_{\R^{n}} \int_0^\infty |U_{n / |n|}(x, y)| \times \biggl|\frac{1}{h(|n| x, |n| y)} - 1\biggr| dx dy \\
	& \qquad = \lim_{|n| \to \infty} \int_{\R^{n}} \int_0^\infty |U_{n / |n|}(\mathbb{O}_n x, y)| \times \biggl|\frac{1}{h(|n| \mathbb{O}_n x, |n| y)} - 1\biggr| dx dy = 0 .
\end{align*}
The desired result now follows from~\eqref{eq:unformlyintegrableconvergence}.
\end{proof}
Let us consider the continuous-discrete operator
$$
    \wt{T}_{\bH^{(k)}} F(x) = \sum_{n \in \Z^{d} \setminus \{0\}} F(x + n) \bK_{\bH^{(k)}}(n). 
$$
Recall that $\bK_{\bH^{(k)}(0)}=0$.  By the argument of Theorem~\ref{homeker}, the norm of the operator $\wt{T}_{\bH^{(k)}}$ on $L^{p}(\R^{d})$ is equal to the norm of the operator $T_{\bH^{(k)}}$ on $\ell^{p}(\Z^{d})$.

For $\varep > 0$, $1< p<\infty$ and a function $F$ on $\R^{d}$, denote $\tau_{\varep} F(x) = \varep^{d/p} F(\varep x)$ and $\wt{T}_{\bH^{(k)}}^\varep = \tau_{1/\varep} \wt{T}_{\bH^{(k)}} \tau_{\varep}$. Observe that $\|\tau_{\varep} F\|_p = \|F\|_p$, and hence the norm of the operator $\wt{T}_{\bH^{(k)}}^\varep$ on $L^{p}(\R^{d})$ does not depend on $\varep$.

We claim that as $\varep \to 0^+$, the operators $\wt{T}_{\bH^{(k)}}^\varep$ approximate the continuous Riesz transform in \eqref{classicalCRT-1}.  More precisely we have 

\begin{lemma} \label{lem:DiscreteApproximation}
Suppose that $F$ is a smooth and compactly supported function on $\R^d$. Then
$$
	\lim_{\varep \to 0^+} \wt{T}_{\bH^{(k)}}^\varep F(x) = R^{(k)} F(x),
$$
for every $x \in \R^{d}$.
\end{lemma}
\begin{proof}
We write
\begin{equation}\label{eq:DiscreteApproximation}\begin{aligned}
	\wt{T}_{\bH^{(k)}}^\varep F(x)
	& = \sum_{n \in \Z^{d} \setminus \{0\}} F(x + \varep n) \bK_{\bH^{(k)}}(n) \\
	& = \frac{1}{2} \sum_{n \in \Z^{d} \setminus \{0\}} (F(x + \varep n) - F(x - \varep n)) \bK_{\bH^{(k)}}(n) \\
	& = \frac{1}{2} \sum_{n \in \Z^{d} \setminus \{0\}} (F(x + \varep n) - F(x - \varep n)) K_{R_{\dis}^{(k)}}(n) \\
	& \qquad + \frac{1}{2} \sum_{n \in \Z^{d} \setminus \{0\}} (F(x + \varep n) - F(x - \varep n)) (\bK_{\bH^{(k)}}(n) - K_{R_{\dis}^{(k)}}(n)) .
\end{aligned}\end{equation}
We treat the two terms on the right-hand side separately.

Since $ K_{R_{\dis}^{(k)}}(n) = \varep^d K_{R^{(k)}}(\varep n)$, the first term in the right-hand side of~\eqref{eq:DiscreteApproximation} is just the Riemann sum
$$
	\frac{\varep^d}{2} \sum_{n \in \Z^{d} \setminus \{0\}} (F(x + \varep n) - F(x - \varep n)) K_{R^{(k)}}(\varep n)
$$
of the integral
$$
	R^{(k)} F(x) = \frac{1}{2} \int_{\R^d} (F(x + y) - F(x - y)) K_{R^{(k)}}(y) dy .
$$
By \eqref{RiemSumConv} of Theorem~\ref{homeker} applied to the kernels for $R^{(k)}$ we have, 
\begin{align*}
	& \lim_{\varep \to 0^+} \frac{\varep^d}{2} \sum_{n \in \Z^{d} \setminus \{0\}} (F(x + \varep n) - F(x - \varep n)) K_{R^{(k)}}(\varep n) \\
	& \qquad = \frac{1}{2} \int_{\R^d} (F(x + y) - F(x - y)) K_{R^{(k)}}(y) dy .
\end{align*}

Thus, it remains to prove that the other term in the right-hand side of~\eqref{eq:DiscreteApproximation} converges to zero. To this end, we apply Lemma~\ref{lem:DiscreteAsymptotics}. First, there is a constant $C_1$ (which depends on $F$) such that $|F(x + y) - F(x - y)| \le C_1 |y|$. If $R > 0$ is large enough, so that $F(x + y) = 0$ whenever $|y| \ge R$, we have
$$\begin{aligned}
	& \frac{1}{2} \biggl| \sum_{n \in \Z^{d} \setminus \{0\}} (F(x + \varep n) - F(x - \varep n)) (\bK_{\bH^{(k)}}(n) - K_{R_{\dis}^{(k)}}(n)) \biggr| \\
	& \qquad \le \frac{C_1}{2} \sum_{n \in \Z^{d} \setminus \{0\}} \ind_{B(0, R)}(\varep n) |\varep n| \, |\bK_{\bH^{(k)}}(n) - K_{R_{\dis}^{(k)}}(n) |.
\end{aligned}$$
Given any $\delta > 0$, by Lemma~\ref{lem:DiscreteAsymptotics} there is $r > 0$ such that $|\bK_{\bH^{(k)}}(n) - K_{R_{\dis}^{(k)}}(n) | < \delta |n|^{-d}$ when $|n| \geq r$. Thus, denoting $C_2 = \sup_{n \ne 0} |\bK_{\bH^{(k)}}(n) - K_{R_{\dis}^{(k)}}(n) |$, we find that
$$\begin{aligned}
	& \frac{1}{2} \biggl| \sum_{n \in \Z^{d} \setminus \{0\}} (F(x + \varep n) - F(x - \varep n)) (\bK_{\bH^{(k)}}(n) - K_{R_{\dis}^{(k)}}(n)) \biggr| \\
	& \qquad \le \frac{C_1 \delta}{2} \sum_{n \in \Z^{d} \setminus \{0\}} \ind_{B(0, R)}(\varep n) |\varep n| \, |n|^{-d} + \frac{C_1 C_2}{2} \sum_{n \in \Z^{d} \setminus \{0\}} \ind_{B(0, r)}(n) |\varep n| \\
	& \qquad \le \frac{C_1 \delta \varep}{2} \sum_{n \in \Z^{d} \setminus \{0\}} \ind_{B(0, R / \varep)}(n) |n|^{1 - d} + \frac{C_1 C_2 \varep}{2} \sum_{n \in \Z^{d} \setminus \{0\}} \ind_{B(0, r)}(n) |n| .
\end{aligned}$$
The first sum in the right-hand side is bounded by $\frac{C_3}{\eps}$ for an appropriate constant $C_3$, and the second one does not depend on $\eps$. Therefore,
$$
	{\limsup_{\eps \to 0^+} \frac{1}{2} \biggl| \sum_{n \in \Z^{d} \setminus \{0\}} (F(x + \varep n) - F(x - \varep n)) (\bK_{\bH^{(k)}}(n) - K_{R_{\dis}^{(k)}}(n)) \biggr| \le \frac{C_1 C_3 \delta}{2} .}
$$
Since $\delta > 0$ is arbitrary, we conclude that 
$$
	\lim_{\varep \to 0^+} \frac{1}{2} \biggl| \sum_{n \in \Z^{d} \setminus \{0\}} (F(x + \varep n) - F(x - \varep n)) (\bK_{\bH^{(k)}}(n) - K_{R_{\dis}^{(k)}}(n)) \biggr| = 0 ,
$$
and the proof is complete.
\end{proof}

Applying Fatou's lemma, exactly as in the proof of Theorem~\ref{homeker}, proves that $\|R^{(k)}\|_{L^p\to L^p}\leq \|T_{\bH^{(k)}}\|_{\ell^p\to\ell^p}$. This and the equality  \eqref{SCRiesz} give the assertion of Theorem~\ref{bestpro-dis-riesz}.

\begin{remark} 
From the martingale inequality in \cite{BanWang}*{Theorem 1} we also obtain the following version of Ess\'en's inequality for the probabilistic discrete Riesz transforms 
\begin{align}\label{eq:Essen}
    \left\|\left(|T_{\bH^{(k)}} f|^2+|f|^2\right)^{\sfrac12}\right\|_{\ell^p}\leq \sqrt{1+\left(\cot\Big(\frac{\pi}{2p^{\ast}}\Big)\right)^2}\|f\|_{\ell^p}. 
\end{align} 
Let $F\in L^p(\R^d)$ be such that $F(x)=f(n)$ for $x\in n+Q$, $n\in\Z^d$. Then, it follows from Lemma~\ref{lem:DiscreteApproximation} with Fatou's lemma that
\begin{align*}
    \left\|\left(|\widetilde{T}_{\bH^{(k)}} F|^2+|F|^2\right)^{\sfrac{1}{2}}\right\|_{L^p} 
    \geq \left\|\left(|R^{(k)} F|^2+|F|^2\right)^{\sfrac12}\right\|_{L^p}.
\end{align*}
Let $\tau_\varep F(x) = \varep^{d/p}F(\varep x)$. Since $\tau_{1/\varep}R^{(k)}\tau_\varep= R^{(k)}$ and $\|\tau_\varep F\|_p=\|F\|_p$, we have
\begin{align*}
    \left\|\left(|R^{(k)} F|^2+|F|^2\right)^{\sfrac12}\right\|_{L^p}
    = \left\|\left(|R^{(k)} (\tau_\varep F)|^2+|\tau_\varep F|^2\right)^{\sfrac12}\right\|_{L^p}.
\end{align*}
Since any function $G\in L^p(\R^d)$ can be approximated by $\tau_\varep F$ where $F(x)$ is of the form $F(x)=\sum_{n\in\Z^d}f(n)\ind_{n+Q}(x)$ for $f\in\ell^p(\Z^d)$, with the help of the Fatou's lemma, we see that the inequality~\eqref{eq:Essen} is also sharp.  
\end{remark} 

\begin{remark} 
Similarly, using the matrices $A_{jk}$ as in \cite{BanWang}*{pg.~595} would lead to what one  may call ``probabilistic discrete  second order Riesz transforms'' with $\ell^p$-norms bounded above by $(p^*-1)$. Notice, however,  that even if we had  the analogues of the above Lemmas for these operators (which we do not currently have), the $(p^*-1)$ bound will not be sharp. Instead one would expect the sharp bound to be  $\frac{1}{2}(p^*-1)$ when $j\ne k$ and the Choi constant when $j=k$, see \cite{GesMonSak, BanOse1}.  Similar questions could be asked about the probabilistic discrete Beurling--Ahlfors operator,  its sharp norm on $\ell^p$  and the relationships to the discrete Beurling--Ahlfors operator which Calder\'on and Zygmund highlight in their discussion on discrete singular integrals, see \cite[pg. 138]{CZ}.  Based on Iwaniec's conjecture \cite{Iwa82} that the norm of the Beurling--Ahlfors operator on $L^p(\R^2)$ is $(p^*-1)$, $1<p<\infty$, one would conjecture that the CZ discrete Beurling--Ahlfors operator should also have norm $(p^*-1)$ on $\ell^p(\Z^2)$.  We have not explored these questions. 
\end{remark} 

\section{Fourier multiplier of the probabilistic discrete Hilbert transform}\label{FMPDH} 
In this section, we focus on the case $d=1$ and compute the Fourier transform of the probabilistic discrete Hilbert transform $T_{\bH}$ whose kernel is given by 
\begin{align}\label{disProbRiesz} 
    \bK_{\bH}(n) =-4\int_{\R}\int_{0}^{\infty}\frac{1}{h}\parder{p}{x}\parder{}{y}(yp_n)\,dydx.
\end{align}

This representation for the kernel of the probabilistic discrete Hilbert transform $T_{\bH}$ together with the computation from Proposition \ref{DisRieszKern1} makes it clear that there is a connection between this operator and the discrete Hilbert transform $H_{\dis}$. However,  this by itself does not yet give the bound $\|H_{\dis}\|_{\ell^p\to \ell^p}\leq \cot(\frac{\pi}{2p^*})$. In order to derive this bound from the bound of  $\|T_{\bH}\|_{\ell^p\to \ell^p}$, as already noted in the introduction, one needs to show that up to convolution with a probability kernel, the discrete Hilbert transform equals the probabilistic discrete Hilbert transform.  This crucial fact was derived in \cite{BanKwa}*{Lemma 1.3} using explicit computations to construct such a kernel. In what follows we provide a completely different proof of this fact, based on the formula from Lemmas \ref{DPG-forlula1}, \ref{DPG-formula2}, which gives an explicit formula for the Fourier transform of such a kernel. Although not clear at all at this point, it may be possible that such an approach based on the Fourier transform (as opposed to the complex variables approach in \cite{BanKwa}) could lead to similar results for the discrete Riesz transforms $R^{(k)}_{\dis}$ in $d>1$.  

Again from  \eqref{Explicitd=1}  
$h(x,y) =\frac{\sinh(2\pi y)}{\cosh(2\pi y)-\cos(2\pi x)} \, ,$
which gives 
\begin{align}\label{hinverse}
 \frac{1}{h(x, y)}=\coth(2\pi y) -\frac{1}{\sinh(2\pi y)} \cos(2\pi x)
\end{align} 
(a linear combination of 1 and $\cos(2\pi x)$ for $y$ fixed). This is crucial for the computations below.   

Let $\psi(x)$ be the digamma function defined by
\begin{align*}
    \psi(x) 
    = \frac{d}{dx}\log(\Gamma(x)) = \frac{\Gamma'(x)}{\Gamma(x)}
    =-\upgamma + \int_0^\infty \frac{e^{-y}-e^{-xy}}{1-e^{-y}}\, dy
\end{align*}
where $\upgamma$ is the Euler constant. 

\begin{theorem}\label{BanKwaMain} 
There exists a kernel $\cP$ such that $\cP(n)\ge 0$ for all $n\in\Z$, $\sum_{n\in\Z}\cP(n)=1$, and $H_{\dis}=T_{\bH}\ast \cP$.  That is, for all $f:\Z\to \R$ of compact support, 
$$H_{\dis}f(n)=(T_{\bH}\ast \cP)f(n),$$ 
where $\ast$ denotes the convolution operation. 
\end{theorem}

An immediate corollary of this is the main result in \cite{BanKwa}. 
\begin{corollary}\label{BanKwaCorollary} 
$$\|H_{\dis}\|_{\ell^p\to \ell^p}\leq \|T_{\bH}\|_{\ell^p\to \ell^p}\leq \cot(\frac{\pi}{2p^*}).$$
\end{corollary} 

\begin{lemma}\label{DPG-forlula1} 
The kernel for the probabilistic discrete Hilbert transform is given by 
\begin{align*}
    \bK_{\bH}(n) = -i\int_\R \frac{\xi}{|\xi|} M(\xi)e^{2\pi i n\xi}\, d\xi 
\end{align*} 
where
\begin{align*}
    M(\xi) =\begin{cases}
    |\xi|^{-1}, & |\xi|\ge 1,\\[1em]
    \ds   1 +(1-|\xi|) \left(2(\psi(1+|\xi|)-\psi(1)\right) +|\xi|(\psi'(1+|\xi|)-\psi'(1)), & |\xi|<1.
    \end{cases}
\end{align*}
\end{lemma}
\begin{proof} 
We begin by observing that 
\begin{align*}
    \Big(\frac{\partial p}{\partial x}(\cdot,y)\Big)^\wedge(\xi)
    &= 2\pi i \xi e^{-2\pi |\xi|y},\\
    \Big(\frac{\partial }{\partial y}(yp_n(\cdot,y))\Big)^\wedge(\xi)
    &= (1-2\pi |\xi|y) e^{-2\pi i n\cdot \xi} e^{-2\pi |\xi|y}. 
\end{align*}
Then, 
\begin{align*}
    \left(\frac{1}{h}\frac{\partial p}{\partial x}\right)^\wedge(\xi)
    &= \coth(2\pi y)\left(\frac{\partial p}{\partial x}\right)^\wedge(\xi) - \frac{1}{\sinh(2\pi y)}\left(\cos(2\pi x)\frac{\partial p}{\partial x}\right)^\wedge(\xi)
\end{align*}
and
\begin{align*}
    \left(\cos(2\pi x)\frac{\partial p}{\partial x}\right)^\wedge(\xi)
    &= \left(\frac12(e^{2\pi i x}+e^{-2\pi i x})\frac{\partial p}{\partial x}\right)^\wedge(\xi)\\
    &= \pi i \left( (\xi-1)e^{-2\pi|\xi-1|y}+(\xi+1)e^{-2\pi|\xi+1|y} \right)\\
    &=\begin{cases}
        \ds 2\pi i \frac{\xi}{|\xi|}e^{-2\pi|\xi|y}(|\xi|\cosh(2\pi y)-\sinh(2\pi y)), & |\xi|\ge 1,\\[1.5em]
        \ds 2\pi i \frac{\xi}{|\xi|}e^{-2\pi y}(|\xi|\cosh(2\pi |\xi|y)-\sinh(2\pi|\xi|y)), & |\xi|<1.
    \end{cases}
\end{align*}
If $|\xi|\ge 1$, then
\begin{align*}
    \left(\frac{1}{h}\frac{\partial p}{\partial x}\right)^\wedge(\xi)
    =2\pi i \frac{\xi}{|\xi|}e^{-2\pi|\xi|y}.
\end{align*}
On the other hand, if $|\xi|<1$, then
\begin{align*}
    \left(\frac{1}{h}\frac{\partial p}{\partial x}\right)^\wedge(\xi)
    =2\pi i \frac{\xi}{|\xi|}e^{-2\pi|\xi|y}\left( |\xi|+(1-|\xi|)\left(\frac{e^{4\pi|\xi|y}-1}{e^{4\pi y}-1}\right) \right).
\end{align*} 
Thus by Plancherel's theorem, 
\begin{align*}
    \bK_{\bH}(n)&=-4\int_0^{\infty}\left(\int_{\R} \left(\frac{1}{h}\frac{\partial p}{\partial x}\right)^\wedge(\xi)\,\ol{\left(\frac{\partial }{\partial y}(yp_n(\cdot,y))\right)^\wedge(\xi)}\, d\xi \right) dy\\
    &=-8\pi i\int_{|\xi|\ge 1}\int_{0}^{\infty} \frac{\xi}{|\xi|} e^{-4\pi |\xi|y}(1-2\pi |\xi|y) e^{2\pi i n\cdot \xi} \,dy d\xi \\
    &\qquad -8\pi i\int_{|\xi|\le 1}\int_{0}^{\infty} \frac{\xi}{|\xi|} e^{-4\pi |\xi|y}(1-2\pi |\xi|y) e^{2\pi i n\cdot \xi} \left( |\xi|+(1-|\xi|)\left(\frac{e^{4\pi|\xi|y}-1}{e^{4\pi y}-1}\right) \right)\,dy d\xi\\
    &=-i\int_{|\xi|\ge 1}\frac{\xi}{|\xi|^2}  e^{2\pi i n\cdot \xi} \,d\xi-i\int_{|\xi|\le 1}\frac{\xi}{|\xi|}  e^{2\pi i n\cdot \xi} \,d\xi \\
    &\qquad- 2i\int_{|\xi|\le 1} \frac{\xi}{|\xi|} e^{2\pi i n\cdot \xi} \left( (1-|\xi|)\int_0^\infty \left(1-\frac{|\xi|y}{2}\right)\frac{e^{-y}-e^{-(1+|\xi|)y}}{1-e^{-y}}\,dy\right)\, d\xi.
\end{align*}
Since
\begin{align*}
    \int_0^\infty (1-\frac{|\xi|y}{2})\frac{e^{-y}-e^{-(1+|\xi|)y}}{1-e^{-y}}\,dy
    =(\psi(1+|\xi|)-\psi(1)) +\frac{|\xi|}{2}(\psi'(1+|\xi|)-\psi'(1)), 
\end{align*}
where $\psi$ is the digamma function, we obtain that 
\begin{align*}
    \bK_{\bH}(n) = -i\int_\R \frac{\xi}{|\xi|} M(\xi)e^{2\pi i n\xi}\, d\xi.
\end{align*} 
\end{proof}

\begin{lemma}\label{DPG-formula2} 
    For $\xi\in Q=[-{\tfrac12,\tfrac12})$ and a compactly supported function $f$ on $\Z$, the Fourier transform of the probabilistic discrete Hilbert transform $T_{\bH}(f)$ is given by 
\begin{align*}
    \cF(T_{\bH}(f))(\xi) = -i\frac{\xi}{|\xi|} \wt{M}(\xi)\cF(f)(\xi)
\end{align*} 
where
\begin{align*}
    \wt{M}(\xi) =1+(1-2|\xi|)(\psi(1+|\xi|)+\psi(1-|\xi|)-2\psi(1))+|\xi|(1-|\xi|)(\psi'(1+|\xi|)-\psi'(1-|\xi|)).
\end{align*}
\end{lemma} 
\begin{proof}
By Lemma~\ref{DPG-forlula1}, we have 
\begin{align*}
    \bK_{\bH}(n) 
    = -i\int_\R \frac{\xi}{|\xi|} M(\xi)e^{2\pi i n\xi}\, d\xi
    = -i\sum_{m\in \Z}\int_Q \frac{\xi+m}{|\xi+m|} M(\xi+m)e^{2\pi i n\xi}\, d\xi.
\end{align*}
Thus, it suffices to show that 
\begin{align*}
    \sum_{n\in\Z}\frac{\xi+n}{|\xi+n|} M(\xi+n) = \frac{\xi}{|\xi|} \wt{M}(\xi),
\end{align*} 
for $\xi\in[0,\tfrac12]$. 
By the series representation for the digamma function
\begin{align*}
    \psi(1+z) = -\upgamma+\sum_{m=1}^\infty\left(\frac{1}{m}-\frac{1}{m+z}\right),
\end{align*}
we have
\begin{align*}
    \sum_{|n|\ge 2}\frac{\xi+n}{|\xi+n|} M(\xi+n)
    &=\sum_{m=2}^\infty\left(\frac{1}{\xi+m}+\frac{1}{\xi-m}\right)\\
    &= \sum_{m=1}^\infty\left(\frac{1}{m+\xi}-\frac{1}{m}+\frac{1}{m}-\frac{1}{m-\xi}\right)+\frac{2\xi}{1-\xi^2}\\
    &= \psi(1-\xi)-\psi(1+\xi) +\frac{2\xi}{1-\xi^2}.
\end{align*}
Using the recurrence relation $\psi(1+z)=\psi(z)+\frac{1}{z}$, we have
\begin{align*}
    &\frac{\xi+1}{|\xi+1|} M(\xi+1)+ \frac{\xi-1}{|\xi-1|} M(\xi-1)\\
    &=\frac{1}{\xi+1}-\left(1+2\xi(\psi(2-\xi)-\psi(1))+\xi(1-\xi)(\psi'(2-\xi)-\psi'(1))\right)\\
    &=1-\frac{2\xi}{1-\xi^2}-\left(1+2\xi(\psi(1-\xi)-\psi(1))+\xi(1-\xi)(\psi'(1-\xi)-\psi'(1))\right).
\end{align*}
Therefore, for $\xi\in[0, \tfrac12]$, we have
\begin{align*}
    \sum_{n\in\Z}\frac{\xi+n}{|\xi+n|} M(\xi+n) 
    &= M(\xi) +\sum_{|n|\ge 2}\frac{\xi+n}{|\xi+n|} M(\xi+n) +\frac{\xi+1}{|\xi+1|} M(\xi+1)+ \frac{\xi-1}{|\xi-1|} M(\xi-1)\\
    &= 1 +(1-\xi) \left(2(\psi(1+\xi)-\psi(1)\right) +\xi(\psi'(1+\xi)-\psi'(1))\\
    &\quad + \psi(1-\xi)-\psi(1+\xi) +\frac{2\xi}{1-\xi^2}+ 1-\frac{2\xi}{1-\xi^2}\\
    &\quad -\left(1+2\xi(\psi(1-\xi)-\psi(1))+\xi(1-\xi)(\psi'(1-\xi)-\psi'(1))\right)\\
    &=\wt{M}(\xi),
\end{align*}
which completes the proof.

\end{proof}

{
\begin{lemma}\label{lem:gconvex}
Let $g(\xi) = \frac{1-2\xi}{\wt{M}(\xi)}$ on $\xi\in[0,\frac12]$. Then, $g$ is strictly positive, decreasing, and convex on $[0,\frac12]$ with $g(0)=1$, $g'(0)=-2$, $g'(\frac12)=0$.
\end{lemma}
\begin{proof}
First, we claim that $\wt{M}(\xi)\ge0$. Since $\wt{M}(0)=1$, and $\wt{M}(1/2)=0$, it suffices to show that  $\wt{M}(\xi)$ is decreasing in $[0,\tfrac12]$. Let $\ph(x):=\psi(1+x)+\psi(1-x)-2\psi(1)$ where $\psi$ is the digamma function. Then, $\wt{M}$ can be written as
\begin{align*}
    \wt{M}(\xi) = 1+(1-2|\xi|)\ph(|\xi|)+|\xi|(1-|\xi|)\ph'(|\xi|).
\end{align*}
By the definition of $\psi$, $\ph(x)$ has the integral representation
\begin{align*}
    \ph(x) = 2\int_0^\infty \frac{1-\cosh(xy)}{e^y - 1}\, dy.
\end{align*}
It follows from this that $\frac{d^n}{dx^n}\ph(x)\le 0$ for all $n=0,1,2,\cdots$ and $x\in[0,\tfrac12]$. Then,
\begin{align*}
    \ph(x) &= -2x^2\sum_{n=1}^\infty \frac{1}{n(n^2-x^2)},\qquad
    \ph'(x) = -4x\sum_{n=1}^\infty \frac{n}{(n^2-x^2)^2}\\
    \ph''(x) &= -4\sum_{n=1}^\infty \frac{n(n^2+3x^2)}{(n^2-x^2)^3}.
\end{align*}
Thus,
\begin{align*}
    \wt{M}'(x)
    &= -2\ph(x)+2(1-2x)\ph'(x)+x(1-x)\ph''(x)\\
    &=-4\sum_{n=1}^\infty\frac{1}{n(n^2-x^2)^3}\left(
    -x^2(n^2-x^2)^2+2(1-2x)n^2(n^2-x^2)+x(1-x)n^2(n^2+3x^2)
\right).
\end{align*}
The numerator in the summand can be written as
\begin{align*}
    &-x^2(n^2-x^2)^2+2(1-2x)n^2(n^2-x^2)+x(1-x)n^2(n^2+3x^2)\\
    &\qquad = (1-2x)(2+x)n^4-x^2(x^2-7x+2)n^2-x^6\\
    &\qquad = (1-2x)(2+x)(n^4-n^2)+((1-2x)(2+x)-x^2(x^2-7x+2)-x^6)n^2+x^6(n^2-1)\\
    &\qquad = (1-2x)(2+x)(n^4-n^2)-(x-1)^2(x^4+2x^3+4x^2-x-2)n^2+x^6(n^2-1).
\end{align*}
Since $x^4+2x^3+4x^2-x-2\le 0$ for all $x\in[0, \tfrac12]$, we conclude that $\wt{M}'(x)\le 0$ for $x\in[0, \tfrac12]$ and so $\wt{M}(x)\ge 0$ on $[0,\frac12]$. Since
\[
    g(1/2) = \lim_{x\to \frac12 -}\frac{1-2x}{\wt{M}(x)}
    = -\lim_{x\to \frac12 -}\frac{2}{\wt{M}'(x)} 
    =\frac{1}{\frac72\zeta(3)-4\log 2}>0,
\]
$g$ is strictly positive.

Let $u(x):=\wt{M}(x)/(1-2|x|)$. We claim that $u$ is increasing and concave on $[0,1/2]$. Denote $\ph(x) = \psi(1 + x) + \psi(1 - x) - 2 \psi(1)$. Since
\begin{align*}
    u(x)
    &=\frac{1}{1-2x}+\ph(x)+\frac{x(1-x)}{1-2x}\ph'(x)\\
    &=\frac{1}{4(1-2x)}(4+\ph'(x))+\ph(x)-\frac{1-2x}{4}\ph'(x),\\
    u'(x)
    &=\frac{1}{2(1-2x)^2}(4+\ph'(x))+\frac{\ph''(x)}{4(1-2x)}+\frac32\ph'(x)-\frac{1-2x}{4}\ph''(x),
\end{align*}
it follows from L'hospital's rule and the recurrence property of polygamma functions that
\begin{align*}
    u'(1/2) = -\frac{1}{16}\ph'''(1/2)+\frac32\ph'(1/2)=0.
\end{align*}
Thus, it suffices to prove that $u''(x)\le 0$ for $x\in[0, \tfrac12]$. Let
\begin{align*}
    v(x)
    &:=(1-2x)^3 u''(x)\\
    &= 2(4+\ph'(x))+(1-2x)(1+2(1-2x)^2)\ph''(x)+x(1-x)(1-2x)^2\ph'''(x).
\end{align*}
Note that 
\begin{align*}
    v'(x) =(1-2x)^2(-12\ph^{(2)}(x)+4(1-2x)\ph^{(3)}(x)+x(1-x)\ph^{(4)}(x)).
\end{align*}
Using the series representation for $\ph(x)$ (see the remark above), we get
\begin{align*}
    &-12\ph^{(2)}(x)+4(1-2x)\ph^{(3)}(x)+x(1-x)\ph^{(4)}(x)\\
    &\qquad = \sum_{n=1}^\infty\frac{n}{(n^2-x^2)^5}(n^6-5x(1-2x)n^4-5x^3(2-x)n^2-x^5)\\
    &\qquad = \sum_{n=1}^\infty\frac{n}{(n^2-x^2)^5}((n^6-n^4)+(1-x)^5n^4+5x^3(2-x)(n^4-n^2)+(n^4-1)x^5)\ge 0
\end{align*}
for all $x\in[0,\tfrac12]$. Since $v'(x)\ge 0$ and $v(1/2)=0$, we obtain that $v(x)\le 0$ and so $u''(x)\le 0$ for $x\in[0,\tfrac12]$ as desired.

\end{proof}

\begin{remark}
    From Lemma~\ref{lem:gconvex}, we see that $\|\wt{M}\|_\infty=1$. This implies that $T_\bH$ is bounded in $\ell^2(\Z)$ and its norm is 1, as we already know. 

\end{remark}

\begin{proof}[Proof of Theorem~\ref{BanKwaMain}]
    The proof is based on P\'olya's criterion for discrete random variables (see for example~\cite{Bill}*{p.353}). 
The Fourier multiplier for the classical discrete Hilbert transform $H_\dis$ with kernel $K_{{H}_{\dis}}(n)=\frac{1}{\pi n}$ is 
\begin{align*}
    \cF(K_{{H}_{\dis}})(\xi)
    = \sum_{n\in\Z\setminus\{0\}} \frac{1}{\pi n}e^{-2\pi i n\xi} 
    = -\sum_{n\in\Z\setminus\{0\}} \frac{2i}{\pi n}\sin(2\pi n\xi) 
    = -i\frac{\xi}{|\xi|}(1-2|\xi|)
\end{align*}
for $\xi\in[-\frac12, -\frac12)$. Thus, we have 
\begin{align*}
    \cF(K_{{H}_{\dis}})(\xi)
    = -i\frac{\xi}{|\xi|}(1-2|\xi|)
    = \cF(\bK_{\bH})(\xi)g(\xi),\qquad g(\xi)=\frac{1-2|\xi|}{\wt{M}(\xi)}.
\end{align*}
Since we know $g(\xi)$ is strictly positive, convex, decreasing on $[0,\frac12]$ and symmetric on $Q$ by Lemma~\ref{lem:gconvex}, there exists a sequence of continuous, symmetric, piecewise linear functions $g_j(\xi)$ on $Q$ converging to $g(\xi)$ pointwise. In particular, for each $j$, there exist $0=t_1<t_2<\cdots<t_k=\frac12$ such that $g_j(0)=1$, $g_j(\frac12)=g(\frac12)=:c$, and $g_j(\xi)=a_i-s_i \xi$ for $\xi\in[t_{i-1}, t_i]$, for some $a_i, s_i\ge 0$, $i=1,2,\cdots,k$. Let $\eta(\xi) = \max\{0, 1-|\xi|\}$, $\eta_t(\xi):=\eta(\xi/t)$, and $\eta_{\infty}(\xi)=1$ on $Q$. 
Since $g_j$ is continuous and $g(0)=1, g(\frac12)=c$,
\[
1
=\sum_{i=1}^k s_i(t_i-t_{i-1})+c
=\sum_{i=1}^k t_i(s_i-s_{i+1})+c
\]
by letting $s_{k+1}=0$. One can see that 
\[
    g_j(\xi) = \sum_{i=1}^k t_i (s_i-s_{i+1})\eta_{t_i}(\xi)+c\eta_\infty(\xi).
\]
Since $\eta_t, \eta_\infty$ for $0<t\le \frac12$ are the characteristic functions of discrete random variables and $g_j$ is the convex combinations of them, $g_j$ is also the characteristic function of a discrete random variable $X_j$ on $\Z$.  Since $g_j$ converges to $g$ pointwise and $g$ is continuous at $0$, we see that $g$ is also the characteristic function of a random variable $X$.
Note that 
\begin{align*}
    i\frac{\xi}{|\xi|}=\cF(\cU)(\xi),\qquad
    \cU(n) = \begin{cases}
        -\frac{1}{2\pi n}, & n\text{ is odd, }\\
        0, & n\text{ is even. }
    \end{cases}
\end{align*}
We know that $\cF(\bK_{\bH})(\xi)=-i\frac{\xi}{|\xi|}\wt{M}(\xi)$, and $\bK_{\bH}(n)$ is decreasing in $n$ with $|\bK_{\bH}(n)|\le \frac{C}{n}$ for some $C>0$. Thus, the convolution $\bK_{\bH}\ast \cU(n)$ is well-defined and the Fourier transform is $\wt{M}$. Since $\wt{M}(0)=1$, we obtain that $\wt{M}(\xi)$ is the characteristic function of a discrete random variable, say $Y$, on $\Z$. Assume $X$ and $Y$ are independent, then the characteristic function of the sum $X+Y$ is $1-2|\xi|$, which is the characteristic function of a discrete random variable. So, we conclude that $X$ is also discrete, supported on $\Z$. That is, there exists a probability kernel $\cP$ such that $\cP(n)\ge 0$ for all $n\in\Z$,  $\sum_{n\in\Z}\cP(n)=1$, and $g(\xi) = \cF(\cP)(\xi)$ on $Q$.
\end{proof}
}

\begin{figure}
\includegraphics[width=10cm]{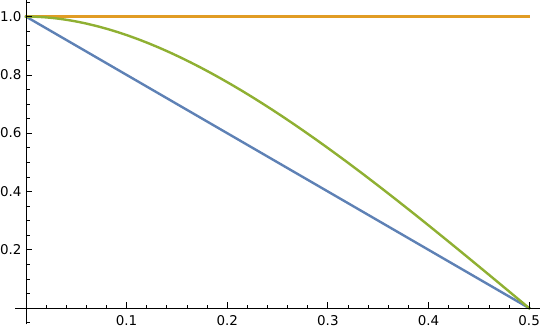}
\caption{Comparison between the Fourier multipliers for the classical Hilbert transform (orange), the discrete Hilbert transform (blue), and the probabilistic discrete Hilbert transform (green).}
\label{fig:fm}
\end{figure}

\begin{question}\label{convZd} 
Does Theorem \ref{BanKwaMain} hold for $d>1$? More precisely, is there a probability kernel $\cP^{(k)}$ on $\Z^d$ such that 
$$
R_{\dis}^{(k)} f(n) =\sum_{m\in \Z^d} \cP^{(k)}(n-m) T_{\bH^{(k)}}f(m) ?    
$$ 
\end{question}

\section{Probabilistic continuous Riesz transforms}\label{ProbContRiesz}

Given that discrete operators obtained from Calder\'on--Zygmund kernels as defined in \eqref{eq:ConK-def} are simply defined by replacing the continuous variable $z$ by the discrete variable $n$ and avoiding the singularity at $\{0\}$ in the sum  are bounded on $\ell^p$ (Propostion \ref{thm:RCZ}),  it is natural to ask if the opposite is also true in the current situation. More precisely, is it true that the kernels obtained from $\bK_{\bH^{(k)}}$ simply by replacing $n \in \Z^d$ with $z\in \R^d$, $|z|\geq 1$, together with the some modification for $|z|<1$,  are Calder\'on--Zygmund  kernels satisfying \eqref{eq:ConK-def}?  In this section we give a formula for such continuous kernels that satisfy \eqref{eq:ConK-def}, with the exception of the $C^1$ property on the sphere $|z|=1$, and are also bounded on $L^p(\R^d)$, $1<p<\infty$ with rather precise norm bounds.   Since we are able to find various explicit constants for the case $d=1$, we consider the cases $d=1$ and $d>1$ separately. 

From formula \eqref{kerneld-1} a natural version of a continuous kernel which gives the probabilistic discrete Hilbert transform \`a la Calder\'on--Zygmund would be
\begin{align}\label{J-0}
\bK_{\bH}(z)=\frac{1}{\pi z}\left(1+\int_0^\infty\frac{2y^3}{(y^2+\pi^2 z^2)\sinh^2(y)}\, dy\right)\,\ind_{\{|z|\geq 1\}}(z)+\frac{1}{\pi z}\ind_{\{|z|<1\}}(z),\quad z\in \R.
\end{align}
Similarly, for $d>1$  from \eqref{byparts} a natural definition of the version of a continuous kernel which gives the probabilistic discrete Riesz transforms for $k=1, 2, \dots, d$ would be 

\begin{align}\label{d>1}
    \bK_{\bH^{(k)}}(z):
	&=\left(-4\int_{\R^{d}}\int_{0}^{\infty}\frac{1}{h(x, y)}\parder{{p_0} }{x_k}\parder{}{y}(yp_z)\,dydx\right)\ind_{\{|z|\geq 1\}}(z)+c_d\frac{z_k}{|z|^{d+1}}\ind_{\{|z|<1\}}(z)\\
&=\left(I_1^{(k)}(z)+I_2^{(k)}(z)\right)\ind_{\{|z|\geq 1\}}(z)+c_d\frac{z_k}{|z|^{d+1}}\ind_{\{|z|<1\}}(z),\nonumber
\end{align}
where
\begin{align*}
	I_1^{(k)}(z)=\int_{\R^d}\int_0^{\infty} \frac{8c_d^2(d+1)x_k y^2}{h(x,y)(|x|^2+y^2)^{\frac{d+3}{2}}(|x-z|^2+y^2)^{\frac{d+1}{2}}}\, dxdy,
	\end{align*} 
	and 
	\begin{align*}
	I_2^{(k)}(z)=-\int_{\R^d}\int_0^{\infty} \frac{4c_d^2(d+1)^2 x_k y^4}{h(x,y)(|x|^2+y^2)^{\frac{d+3}{2}}(|x-z|^2+y^2)^{\frac{d+3}{2}}}\, dxdy 
\end{align*}
and $c_d=\Gamma(\frac{d+1}{2})\pi^{-\frac{d+1}{2}}$. 

Notice that $\bK_{\bH^{(k)}}$ is not continuous on  $|z|=1$ and hence not a Calder\'on--Zygmund kernel requiring \eqref{eq:ConK-def}.  Nevertheless, with these definitions we have 
\begin{theorem}\label{almostCZ} For any $d\geq 1$ and  $k=1, \dots d$,  the kernels 
$\bK_{\bH^{(k)}}(z)$ ($\bK_{\bH}(z)$ when $d=1$) satisfy 
\begin{itemize} 
\item[(i)] 
\begin{align}\label{supK}
|\bK_{\bH^{(k)}}(z)|\leq \frac{C_d}{|z|^d}, \quad z\in \R^d\setminus\{0\}.
\end{align} 
\item[(ii)] 
\begin{align}\label{supgradK}
|\nabla\bK_{\bH^{(k)}}(z)|\leq \frac{C_d}{|z|^{d+1}},\quad z\in \R^d\setminus\{0\},  \quad |z|\ne 1 ,
\end{align} 
where $C_d$ depends only on $d$. Furthermore,  
\item[(iii)] For $d=1$, 
\begin{align}\label{Fourier-d}
\sup_{\xi\in \R}|\wh{\bK_{\bH}}(\xi)|
\le 1+\frac{2}{\pi}\int_0^{\infty}\frac{y\ln(y^2/\pi^2+1)}{\sinh^2(y)} dy 
\approx 1.09956.
\end{align} 
\item[(iv)] For $d>1$ and all $k=1, \dots d$, we have 
\begin{align}\label{Fourier-d>1}
|\wh{\bK_{\bH^{(k)}}}(\xi)|\leq C_d, \quad \text{for all}\quad  \xi\in \R^d\setminus\{0\}
\end{align} 
\end{itemize} 

\end{theorem} 

\begin{proof} 
We first show the case $d=1$ which is computationally much simpler  and will give explicit constants, particularly the bound for the Fourier transform. 
Clearly $|\bK_{\bH}(z)|=\frac{1}{\pi |z|}$,  for $0<|z|<1$. On the other hand, 
\begin{align}\label{zeta} 
    \int_0^\infty\frac{y^3}{(y^2+ \pi^2 z^2)\sinh^2(y)}\, dy
    &\le \frac{1}{\pi^2z^2}\int_0^\infty\frac{y^3}{\sinh^2(y)}\, dy\nonumber\\
    &= \frac{4}{\pi^2z^2}\int_0^\infty\frac{y^3}{e^{2y}(1-e^{-2y})^2}\, dy\nonumber\\
    &= \frac{4}{\pi^2z^2}\sum_{k=1}^{\infty} k\int_0^{\infty} y^3 e^{-2ky}\, dy\nonumber\\
    &= \frac{3}{2\pi^2z^2}\sum_{k=1}^{\infty}\frac{1}{k^3}= \frac{3}{2\pi^2z^2}\zeta(3). 
\end{align}
From this it follows that  
$|\bK_{\bH}(z)|\leq \frac{C_1}{|z|}$,  for all $|z|>0$. 
Similarly, for  $|z|\geq 1$,
\begin{align*}
    |\bK_{\bH}'(z)|
    &=\left|-\frac{1}{\pi z^2}\left(1+\int_0^\infty\frac{2y^3}{(y^2+\pi^2 z^2)\sinh^2(y)}\, dy\right) -\int_0^\infty\frac{4 y^3}{(y^2+\pi^2 z^2)^2\sinh^2(y)}\, dy\right|\\
    &\leq \frac{C_2}{|z|^{2}}, 
\end{align*} 
and again we have $|\bK_{\bH}'(z)|\leq \frac{C}{|z|^2}$  for all $|z|>0$,  $|z|\neq 1$.  Here, $C$ is a universal constant.  Thus $\bK_{\bH}$ satisfies (i) and (ii).   

In addition, we have, in the principal value sense, 
\begin{align}\label{eq:Khat}
    \wh{\bK}_{\bH}(\xi)
    &=\int_{\bR} \bK_{\bH}(z)e^{-2\pi iz\cdot \xi} dz\nonumber\\
    &=\int_{\bR}\frac{1}{\pi z}e^{-2\pi iz\cdot\xi} dz +\frac{1}{\pi} \int_0^\infty\frac{2y^3}{\sinh^2(y)}\, \left(\int_{\{|z|>1\}}\frac{e^{-2\pi iz\cdot \xi}}{z(y^2+\pi^2 z^2)}\, dz\right)\,dy\\
    &=-i\sign(\xi)+\frac{1}{\pi} \int_0^\infty\frac{2y^3}{\sinh^2(y)}\, \left(\int_{\{|z|>1\}}\frac{e^{-2\pi i z\cdot \xi}}{z(y^2+\pi^2 z^2)}\, dz\right)\,dy\nonumber\\
    &=-i\sign(\xi)\left(1+\frac{1}{\pi} \int_0^\infty\frac{2y^3}{\sinh^2(y)}\, 
    \left(\int_{\{|z|>1\}}\frac{\sin(2\pi z|\xi|)}{z(y^2+\pi^2 z^2)}\, dz\right)\,dy\right)\nonumber. 
\end{align} 
By Fubini and the following identity obtained by integration by parts, $$
\int_1^{\infty}\frac{1}{z(y^2+\pi^2 z^2)} dz=\frac{\ln(y^2/\pi^2+1)}{2y^2},
$$ 
we have
\begin{align}\label{eq:JL1}
    \Big|\frac{1}{\pi} \int_0^\infty\frac{y^3}{\sinh^2(y)}\, \left(\int_{\{|z|>1\}}\frac{e^{-2\pi iz\cdot \xi}}{z(y^2+\pi^2 z^2)}\, dz\right)\,dy\Big|
    &\leq
    \frac{2}{\pi} \int_0^\infty\frac{y^3}{\sinh^2(y)}\, \left(\int_1^\infty\frac{1}{z(y^2+\pi^2 z^2)}\, dz\right)\,dy\nonumber\\
    &=\frac{1}{\pi}\int_0^{\infty}\frac{y\ln(y^2/\pi^2+1)}{\sinh^2(y)} dy\nonumber\\
    &\approx 0.0497822.
\end{align}
Hence for all $\xi\in \R$, 
$$|\wh{\bK}_{\bH}(\xi)|
\le 1+\frac{2}{\pi}\int_0^{\infty}\frac{y\ln(y^2/\pi^2+1)}{\sinh^2(y)} dy
\approx 1.09956,
$$
which is the claim in (iii).   

We now suppose $d>1$. By \eqref{eq:Unbound}, \eqref{eq:hupperbound2} and a change of variables we have, for $|z|\geq 1$ and $z=|z|\theta$, that 
\begin{align*}
    |\bK_{\bH^{(k)}}(z)|
    &\leq C_d\int_{0}^{\infty}\int_{\bR^d} \frac{y}{(|x|^2+y^2)^{\frac{d+1}{2}}(|x-z|^2+y^2)^{\frac{d+1}{2}}}\, dxdy\\
    &= \frac{C_d}{|z|^{d}}\int_{0}^{\infty}\int_{\bR^d} \frac{y}{(|x|^2+y^2)^{\frac{d+1}{2}}(|x-\theta|^2+y^2)^{\frac{d+1}{2}}}\, dxdy.
\end{align*}
{
Since $\{|x-\theta|\le \frac12\}\subset\{|x|>\frac12\}$, we have
\begin{align*}
    &\int_0^\infty\int_{\R^d} \frac{y}{(|x|^2+y^2)^{\frac{d+1}{2}}(|x-\theta|^2+y^2)^{\frac{d+1}{2}}}\, dxdy\\
    &=\int_0^\infty\int_{\{|x-\theta|\le \frac12\}} \frac{y}{(|x|^2+y^2)^{\frac{d+1}{2}}(|x-\theta|^2+y^2)^{\frac{d+1}{2}}}\, dxdy\\
    &\qquad +\int_0^\infty\int_{\{|x-\theta|> \frac12\}} \frac{y}{(|x|^2+y^2)^{\frac{d+1}{2}}(|x-\theta|^2+y^2)^{\frac{d+1}{2}}}\, dxdy\\
    &\le \int_0^\infty\int_{\{|x|> \frac12\}} \frac{y}{(\frac14+y^2)^{\frac{d+1}{2}}(|x-\theta|^2+y^2)^{\frac{d+1}{2}}}\, dxdy\\
    &\qquad +\int_0^\infty\int_{\{|x-\theta|> \frac12\}} \frac{y}{(|x|^2+y^2)^{\frac{d+1}{2}}(\frac14+y^2)^{\frac{d+1}{2}}}\, dxdy.
\end{align*}
By change of variables,
\begin{align*}
    &\int_0^\infty\int_{\R^d} \frac{y}{(|x|^2+y^2)^{\frac{d+1}{2}}(|x-\theta|^2+y^2)^{\frac{d+1}{2}}}\, dxdy\\
    &\qquad\le 2 \int_0^\infty\left(\int_{\R^d} \frac{y}{(|x|^2+y^2)^{\frac{d+1}{2}}}\, dx\right)\frac{1}{(\frac14+y^2)^{\frac{d+1}{2}}}dy\\
    &\qquad\le 2 \int_0^\infty\left(\int_{\R^d} \frac{1}{(|x|^2+1)^{\frac{d+1}{2}}}\, dx\right)\frac{1}{(\frac14+y^2)^{\frac{d+1}{2}}}dy<\infty.
\end{align*}
}
This together with the obvious  bound for the second term in \eqref{d>1} gives that $|\bK_{\bH^{(k)}}(z)|\leq \frac{C_d}{|z|^{d}}$,  
for $|z|>0$. Next, for $j=1,2,\cdots,d$, differentiation and \eqref{eq:hupperbound2} gives that for $|z|>1$, 
\begin{align*}
   \Big| \frac{\partial}{\partial z_{j}}I_1^{(k)}(z)\Big|
    &= C_d \Big|\int_0^{\infty}\int_{\R^d} \frac{x_k y^2 (x_j-z_j)}{h(x,y)(|x|^2+y^2)^{\frac{d+3}{2}}(|x-z|^2+y^2)^{\frac{d+3}{2}}}\, dxdy\Big|\\
    &\le  \frac{C_d}{|z|^{(d+1)}}\int_0^{\infty}\int_{\R^d} \frac{|x_k x_j|y}{(|x|^2+y^2)^{\frac{d+2}{2}}(|x-\theta|^2+y^2)^{\frac{d+3}{2}}}\, dxdy\\
    &\qquad + \frac{C_d |z_j|}{|z|^{(d+2)}}\int_0^{\infty}\int_{\R^d} \frac{|x_k z_j| y}{(|x|^2+y^2)^{\frac{d+2}{2}}(|x-\theta|^2+y^2)^{\frac{d+3}{2}}}\, dxdy\\
    &\leq \frac{C_d}{|z|^{d+1}},  
\end{align*}
where $z=|z|\theta$. Similarly, we  can obtain the same upper bound for $\frac{\partial}{\partial z_{j}}I_2^{(k)}(z)$, $|z|>1$, which leads to $\big|\nabla{\bK_{\bH^{(k)}}(z)}\big|\leq\frac{C_d}{|z|^{d+1}}$ for all $|z|>0$ and $|z|\neq 1$.   

It remains to show that the Fourier transform of $\bK_{\bH^{(k)}}(z)$ is bounded. By~\eqref{byparts} and Proposition~\ref{DisRieszKern1}, we have that 
\begin{align}\label{h=1case}
-4\int_{\R^{d}}\int_{0}^{\infty}\parder{p_0(x, y)}{x_k}\parder{}{y}(yp_z(x, y))\,dydx=\frac{c_d z_k}{|z|^{d+1}}, \quad |z|>0.
\end{align} 
This formula can also be easily verified using the Fourier transform.  More precisely, 
{
for a smooth function $\phi$ with compact support, we have
\begin{align*}
    &\int_{\R^d}\left(\int_{\R^{d}}\int_{0}^{\infty}\parder{p}{x_k}(x,y)\parder{}{y}(yp(x-z,y))\,dydx\right)\wh{\phi}(z)\, dz \\
    &=\int_{\R^{d}}\int_{\R^{d}}\int_{0}^{\infty}\parder{p}{x_k}(x,y)\parder{}{y}\left(\int_{\R^d}yp(x-z,y)e^{-2\pi i \xi z}\, dz\right)\phi(\xi)\,dydxd\xi \\
    &=\int_{\R^{d}}\int_{\R^{d}}\int_{0}^{\infty}\parder{p}{x_k}(x,y) \left(e^{-2\pi i \xi x}(1-2\pi|\xi|y)e^{-2\pi |\xi|y}\right) \phi(\xi)\,dydxd\xi \\
    &=\int_{\R^{d}}\int_{0}^{\infty}2\pi i \xi_k (1-2\pi|\xi|y)e^{-4\pi |\xi|y}\phi(\xi)\,dyd\xi \\
    &=\int_{\R^{d}}\frac{i\xi_k}{4|\xi|}\phi(\xi)\,dyd\xi.
\end{align*}
The equation~\eqref{h=1case} follows from the fact that the Fourier transform of $c_d\frac{z_k}{|z|^{d+1}}$ is $-i\frac{\xi_k}{|\xi|}$ in the principal sense.
}
Thus we can write \eqref{d>1} as 
\begin{align}\label{higherCZK2}
    \bK_{\bH^{(k)}}(z)
	=\left(-4\int_{0}^{\infty}\int_{\R^{d}}\left(\frac{1}{h}-1\right)\parder{p_0(x, y)}{x_k}\parder{}{y}(yp_z(x, y))\,dxdy\right) \ind_{\{|z|\ge 1\}}(z) +c_d\frac{\,z_k}{|z|^{d+1}}.
\end{align}
Since the Fourier transform of the second term is $-i\frac{\xi_k}{|\xi|}$ (the Fourier transform of the classical Riesz transforms), it is enough to show that 
\begin{align*}
    \int_{\{|z|\ge 1\}} J_1^{(k)} (z)e^{-2\pi i z\cdot \xi}\, dz,\quad
    \int_{\{|z|\ge 1\}} J_2^{(k)} (z)e^{-2\pi i z\cdot \xi}\, dz 
\end{align*} 
are uniformly bounded in $\xi$, where 
\begin{align}
	J_1^{(k)}(z)&:=\int_0^{\infty}\int_{\R^d}\left(\frac{1}{h(x,y)}-1\right)\frac{x_k y^2}{(|x|^2+y^2)^{\frac{d+3}{2}}(|x-z|^2+y^2)^{\frac{d+1}{2}}}\, dxdy,\label{J-1}\\
	J_2^{(k)}(z)&:=\int_0^{\infty}\int_{\R^d} \left(\frac{1}{h(x, y)}-1\right)\frac{x_k y^4}{(|x|^2+y^2)^{\frac{d+3}{2}}(|x-z|^2+y^2)^{\frac{d+3}{2}}}\, dxdy\label{J-2}.
\end{align}
By the estimate~\eqref{eq:hupperbound2}, $|\frac{1}{h(x,y)}-1|\le C_d\frac{\sqrt{x^2+y^2}}{y}$, and
\begin{align}\label{eq:kernelbd}
\int_{\{|x|\ge \frac12\}}(|x|^2+y^2)^{-\frac{d+1}{2}}\,dx\le
\int_{\{|x|\ge \frac12\}}|x|^{-(d+1)}\,dx\le C_d ,
\end{align}
for $0\le y\le 1$ we have 
\begin{align*}
    &\int_{\{|z|\ge 1\}} \int_0^1 \int_{\{|x|\le \frac12\}}
    \left|\frac{1}{h(x,y)}-1\right|\frac{|x_k| y^2}{(|x|^2+y^2)^{\frac{d+3}{2}}(|x-z|^2+y^2)^{\frac{d+1}{2}}}\, dxdydz\\
    &\le \int_0^1\int_{\{|x|\le \frac12\}}\left(\int_{\{|x-z|\ge \frac12\}}\frac{1}{(|x-z|^2+y^2)^{\frac{d+1}{2}}}\, dz\right)  \left|\frac{1}{h(x,y)}-1\right|\frac{|x_k| y^2}{(|x|^2+y^2)^{\frac{d+3}{2}}}\, dx dy\\
    &\le C_d\int_0^1 \int_{\{|x|\le \frac12\}} \frac{ y}{(|x|^2+y^2)^{\frac{d+1}{2}}}\, dx dy \le C_d.
\end{align*}
Note that in the first inequality, we used the fact that if $|x|\le \frac12$ and $|z|\ge 1$, then $|x-z|\ge\frac12$. Similarly,
\begin{align*}
    &\int_{\{|z|\ge 1\}} \int_1^\infty \int_{\{|x|\le \frac12\}}
    \left|\frac{1}{h(x,y)}-1\right|\frac{|x_k| y^2}{(|x|^2+y^2)^{\frac{d+3}{2}}(|x-z|^2+y^2)^{\frac{d+1}{2}}}\, dxdydz\\
    &\le C_d\int_1^\infty\left(\int_{\{|x|\ge \frac12\}}\frac{1}{(|x|^2+y^2)^{\frac{d+1}{2}}}\, dx\right) \left( \int_{\{|x|\le \frac12\}} \left|\frac{1}{h(x,y)}-1\right|\frac{|x_k| y^2}{(|x|^2+y^2)^{\frac{d+3}{2}}}\, dx \right)\, dy\\
    &\le C_d\int_1^\infty
    \int_{\{|x|\le \frac12\}} y^{-d-1}\, dx dy \le C_d.
\end{align*}
On the other hand, it follows from \eqref{eq:hupperbound2}, \eqref{eq:kernelbd}, and the bound
\begin{align}
    \int_{\{|z|\ge 1\}}\frac{y}{(|x-z|^2+y^2)^{\frac{d+1}{2}}}\, dz\le C_d
\end{align} 
that 
\begin{align*}
    &\int_{\{|z|\ge 1\}} \int_0^1 \int_{\{|x|\ge \frac12\}} \left|\frac{1}{h(x,y)}-1\right|\frac{|x_k| y^2}{(|x|^2+y^2)^{\frac{d+3}{2}}(|x-z|^2+y^2)^{\frac{d+1}{2}}}\, dxdydz\\
    &=\int_0^1 \int_{\{|x|\ge \frac12\}}\left( \int_{\{|z|\ge 1\}}\frac{y}{(|x-z|^2+y^2)^{\frac{d+1}{2}}}\, dz \right) \left|\frac{1}{h(x,y)}-1\right|\frac{|x_k| y}{(|x|^2+y^2)^{\frac{d+3}{2}}}\, dxdy\\
    &\le  C_d\int_0^1 \int_{\{|x|\ge \frac12\}} \left|\frac{1}{h(x,y)}-1\right|\frac{|x_k| y}{(|x|^2+y^2)^{\frac{d+3}{2}}}\, dxdy\\
    &\le C_d\int_0^1 \int_{\{|x|\ge \frac12\}} \frac{1}{(|x|^2+y^2)^{\frac{d+1}{2}}}\, dxdy\le C_d.
\end{align*}
By Lemma~\ref{lem:hlimit}, we know that $|1/h(x,y)-1|\le \frac{C_d}{y}$ for $y\ge 1$. Using this,
\begin{align*}
    &\int_{\{|z|\ge 1\}} \int_1^\infty \int_{\{|x|\ge \frac12\}} \left|\frac{1}{h(x,y)}-1\right|\frac{|x_k| y^2}{(|x|^2+y^2)^{\frac{d+3}{2}}(|x-z|^2+y^2)^{\frac{d+1}{2}}}\, dxdydz\\
    &= \int_1^\infty \int_{\{|x|\ge \frac12\}}\left(\int_{\{|z|\ge 1\}} \frac{y}{(|x-z|^2+y^2)^{\frac{d+1}{2}}}\, dz\right) \left|\frac{1}{h(x,y)}-1\right|\frac{|x_k| y}{(|x|^2+y^2)^{\frac{d+3}{2}}}\, dxdy\\
    &\le  C_d\int_1^\infty \int_{\{|x|\ge \frac12\}} \frac{ 1}{(|x|^2+y^2)^{\frac{d+2}{2}}}\, dxdy \\
    &\le  C_d\int_1^\infty \frac{1}{y^2}\, dy \int_{\R^d} \frac{ 1}{(|w|^2+1)^{\frac{d+2}{2}}}\, dw \le C_d.
\end{align*}
In the last inequality, we have used the change of variable $wy=x$. Thus, we get
\begin{align}\label{L1-J1}
    \left|\int_{\{|z|\ge 1\}} J_1^{(k)} (z)e^{-2\pi i z\cdot \xi}\, dz\right| \le \int_{\{|z|\ge 1\}} |J_1^{(k)} (z)|\, dz\le C_d.
\end{align}
Using the trivial bound $y^2/(|x-z|^2+y^2)\le 1$, it follows from the previous argument that
\begin{align*}
    &\left|\int_{\{|z|\ge 1\}} J_2^{(k)} (z)e^{-2\pi i z\cdot \xi}\, dz\right| \\
    &\le \int_{\{|z|\ge 1\}}\int_0^\infty \int_{\R^d} \left|\frac{1}{h(x,y)}-1\right|\frac{|x_k| y^4}{(|x|^2+y^2)^{\frac{d+3}{2}}(|x-z|^2+y^2)^{\frac{d+3}{2}}}\, dxdydz\\
    &\le \int_{\{|z|\ge 1\}}\int_0^\infty \int_{\R^d} \left|\frac{1}{h(x,y)}-1\right|\frac{|x_k| y^2}{(|x|^2+y^2)^{\frac{d+3}{2}}(|x-z|^2+y^2)^{\frac{d+1}{2}}}\, dxdydz\\
    &\le C_d.
\end{align*}
\end{proof} 

\begin{remark} 
Note that the proof of the boundedness of the Fourier transform for $d=1$ shows that in fact the function 
$$
J(z)=\frac{1}{\pi z}\left(\int_0^\infty\frac{2y^3}{(y^2+\pi^2 z^2)\sinh^2(y)}\, dy\right)\,\ind_{\{|z|\geq 1\}}(z)
$$
is in $L^1(\R)$ with $\|J\|_{L^1}\approx 0.09956$. Similarly, for $d>1$, the proof shows that $J_1^{(k)}(z)\ind_{\{|z|\geq 1\}}(z)$ and $J_2^{(k)}(z)\ind_{\{|z|\geq 1\}}(z)$ are in $L^1(\R^d)$ with bounds depending only on $d$. 
\end{remark} 

This gives the following

\begin{corollary}\label{ContDecom}
For $d=1$, the continuous probabilistic Hilbert transform is given by 
\begin{align}
\bK_{\bH}(z)&=\frac{1}{\pi z}\left(1+\int_0^\infty\frac{2y^3}{(y^2+\pi^2 z^2)\sinh^2(y)}\, dy\right)\,\ind_{\{|z|\geq 1\}}(z)+\frac{1}{\pi z}\ind_{\{|z|< 1\}}(z)\nonumber\\
&=J(z)+\frac{1}{\pi z},
\end{align} 
where $\|J\|_{L^1}\approx 0.09956$.

Similarly for $d>1$,
\begin{align}
\bK_{\bH^{(k)}}(z)&=J_1^{(k)}+J_2^{(k)}+c_d\frac{\,z_k}{|z|^{d+1}},\\
&= J^{(k)}
+c_d\frac{\,z_k}{|z|^{d+1}},\nonumber 
\end{align}
where $\|J^{(k)}\|_{L^1}\leq C_d$ where $C_d$ depends only on $d$.
\end{corollary} 
With $J$ and $J^{(k)}$ as above, the probabilistic continuous Hilbert and Riesz transforms are of the form:
\begin{align}
Tf=\bK_{\bH}*f&=J*f+Hf, \\ 
T^kf =\bK_{\bH^{(k)}}*f&=J^{(k)}*f+R^{(k)}*f, \quad k=1, \dots, d.
\end{align}

Recall that for $\varep > 0$, $1< p<\infty$ and a function $F$ on $\R^{d}$, denote $\tau_{\varep} F(x) = \varep^{d/p} F(\varep x)$. 
We define $(T^k)^\varep = \tau_{1/\varep}T^k\tau_{\varep}$.
Following the proof of Lemma~\ref{lem:DiscreteApproximation}, one has the following
\begin{lemma}
Suppose that $F$ is a smooth and compactly supported function on $\R^d$. Then
$$
	\lim_{\varep \to 0^+} (T^k)^\varep F(x) = R^{(k)} F(x),
$$
for every $x \in \R^{d}$.
As a result, we have $\cot(\pi/(2p^\ast))=\|R^{(k)} F\|_{p\to p}\le \|T^{k}\|_{p\to p}$. The same holds for $T$ as well.
\end{lemma}
\begin{proof}
Note that we have
\begin{align*}
    (T^k)^\varep F(x)
    &= \frac12 \int_{\R^d} \bK_{\bH^{(k)}}(y)(F(x-\varep y)-F(x+\varep y))\, dy\\
    &= \frac12 \int_{\R^d} K^{(k)}(y)(F(x-\varep y)-F(x+\varep y))\, dy\\
    &\qquad + \frac12 \int_{|y|\ge 1} (\bK_{\bH^{(k)}}(y)-K^{(k)}(y))(F(x-\varep y)-F(x+\varep y))\, dy.
\end{align*} 
Since $K^{(k)}(y)= \varep^d K^{(k)}(\varep y)=-K^{(k)}(-y)$, we have
\begin{align*}
    \frac12\int K^{(k)}(y)(F(x-\varep y)-F(x+\varep y))\, dy
    = R^{(k)}F(x).
\end{align*}
Since $F$ is smooth and compactly supported, there exists $R>0$ such that $|F(x-\varep y)-F(x+\varep y)|\le C_1\ind_{B(0,R)}(\varep y)|\varep y|$ for some $C_1>0$. By Lemma~\ref{lem:DiscreteAsymptotics}, for given $\delta>0$, there exists $r>0$ such that $|\bK_{\bH^{(k)}}(y)-K^{(k)}(y)|\le \delta |y|^{-d}$ for $|y|>r$. Let $C_2=\sup_{|y|\ge 1}|\bK_{\bH^{(k)}}(y)-K^{(k)}(y)|$, then
\begin{align*}
    &\int_{|y|\ge 1} (\bK_{\bH^{(k)}}(y)-K^{(k)}(y))(F(x-\varep y)-F(x+\varep y))\, dy\\
    &\qquad \le C_1 \varep \int_{1\le |y|\le R/\varep} |y||\bK_{\bH^{(k)}}(y)-K^{(k)}(y)|\, dy \\
    &\qquad \le C_1 C_2 \varep \int_{1\le |y|\le r} |y|\, dy 
    +C_1 \delta \varep \int_{1\le |y|\le R/\varep} |y|^{1-d}\, dy \\
    &\qquad \le C_3 r^2 \varep +C_4 \delta.
\end{align*}
Choose $\delta$ small enough and letting $\varep\to 0$, we conclude that $(T^k)^\varep(F)(x)$ converges to $R^{(k)}(F)(x)$. By Fatou's lemma, we also have the lower bound of the $L^p$ norm of $T^{k}$ as desired.
\end{proof}

We record the $L^p$-boundedness of the operators in the following  Theorem. 

\begin{theorem}\label{LpCont} 
For $1<p<\infty$, 
\begin{align}\label{ContHoilbert}
\cot(\pi/(2p^\ast))\le \|T\|_{L^p\to L^p}\leq \|J\|_{L^1}+\cot(\pi/(2p^\ast))\approx 0.09956+\cot(\pi/(2p^\ast))
\end{align} 
and 
\begin{align}\label{ContRiesz}
\cot(\pi/(2p^\ast))\le \|T^k\|_{L^p\to L^p}\leq \|J^{(k)}\|_{L^1}+\cot(\pi/(2p^\ast))\leq C_d+\cot(\pi/(2p^\ast)),
\end{align} 
where $C_d$ depends on the dimension $d$.
\end{theorem} 

\begin{conjecture}\label{conticon}
We conjecture that the $L^p$-norm of the operator $T$ should be $cot(\pi/(2p^\ast))$ and similarly for $T^k$.  For the latter, even the conjecture that the norm  is independent of $d$ would be of interest. See the discussion preceding \eqref{vector}. 
\end{conjecture}

\begin{remark}
Recall that the smoothness condition $|\nabla K(z)|\le \kappa |z|^{-(d+1)}$ can be relaxed with H\"ormander's condition
\begin{align}\label{Hor}
    \sup_{y\neq 0}\int_{|x|\ge 2|y|}|K(x-y)-K(x)|\, dx\le B<\infty,
\end{align}
see \cite{Graf}*{Theorem 4.3.3} and \cite{Stein70}*{Corollary on p.34, Theroem 2, p.35}. In particular, 
if $K$ satisfies $|K(x)|\le \kappa|x|^{-d}$ and H\"ormander's condition, then the convolution operator $T$ with kernel $K$ is bounded on $L^p$, $1<p<\infty$. 
\end{remark} 
\begin{theorem}\label{hormander}
The kernels $\bK_{\bH^{(k)}}$ satisfy the H\"ormander's condition.
\end{theorem}
\begin{proof} 
 We have already seen that 
\begin{align*}
    |\nabla\bK_{\bH^{(k)}}(x)|\le \frac{C}{|x|^{d+1}}, \quad \text{for }|x|\neq 0,1.
\end{align*}
Suppose $|y|>1$. If $|x|\ge 2|y|$ then 
$$
|\theta (x-y) +(1-\theta)x|=|x-\theta y|\ge (1-\tfrac{\theta}{2})|x|\ge (2-\theta)|y|
$$
for all $0\le \theta\le 1$. By Taylor's theorem, we have
\begin{align*}
    |\bK_{\bH^{(k)}}(x-y)-\bK_{\bH^{(k)}}(x)| \le |\nabla \bK_{\bH^{(k)}}(x-\theta y)||y| \le \frac{C|y|}{|x-\theta y|^{d+1}} \le \frac{C|y|}{|x|^{d+1}},
\end{align*}
which leads to
\begin{align*}
    \int_{|x|\ge 2|y|}|\bK_{\bH^{(k)}}(x-y)-\bK_{\bH^{(k)}}(x)|\, dx
    \le C|y|\int_{2|y|}^\infty \frac{1}{r^2}\, dr
    \le C<\infty.
\end{align*}
If $|y|\le 1$ and $|x|>|y|+1$, then $|x-\theta y|>1$ for $0\le \theta\le 1$. Thus, the same argument yields
\begin{align*}
    \int_{|x|\ge |y|+1}|\bK_{\bH^{(k)}}(x-y)-\bK_{\bH^{(k)}}(x)|\, dx \le C<\infty.
\end{align*}
Let $\tfrac14<|y|\le 1$ and $2|y|\le |x|\le |y|+1$. Using $|\bK_{\bH^{(k)}}(x)|\le \kappa|x|^{-d}$ and $|x-y|\ge \tfrac12 |x|$, we get
\begin{align*}
    \int_{2|y|\le|x|\le |y|+1}|\bK_{\bH^{(k)}}(x-y)-\bK_{\bH^{(k)}}(x)|\, dx 
    \le C\int_{2|y|\le|x|\le |y|+1}|x|^{-d} \, dx 
    = C \left|\log\left(\frac{|y|+1}{2|y|}\right)\right|,
\end{align*}
which is bounded for $|y|\in (\tfrac14,1]$.
Suppose $|y|\le\tfrac14$ and $\tfrac34\le |x|\le |y|+1$, then the same argument gives 
\begin{align*}
    \int_{\tfrac34\le|x|\le |y|+1}|\bK_{\bH^{(k)}}(x-y)-\bK_{\bH^{(k)}}(x)|\, dx 
    \le C\int_{\tfrac34\le|x|\le |y|+1}|x|^{-d} \, dx 
    = C |\log(\tfrac43(|y|+1))|<\infty.
\end{align*}
If $|y|\le \tfrac14$ and $2|y|\le|x|<\tfrac34$, then $|x-\theta y|<1$. Thus it follows from the gradient bound that
\begin{align*}
    \int_{2|y|\le |x|< \frac34}|\bK_{\bH^{(k)}}(x-y)-\bK_{\bH^{(k)}}(x)|\, dx
    \le C|y|\int_{2|y|}^{\frac34} \frac{1}{r^2}\, dr
    \le C<\infty.
\end{align*}
Therefore  the kernel $\bK_{\bH^{(k)}}$ satisfies H\"ormander's condition and the $L^p(\R^d)$-boundedness of the operators for $1<p<\infty$ also follows from the Calder\'on--Zygmund theory. 
\end{proof}

\section{A method of rotations for discrete Riesz trasforms}\label{MethodRotations}
Given the fact that the classical method of rotations can be used to show that the Riesz transforms (and other singular integrals) in $\R^d$ have norms bounded above by the norm of the Hilbert transform, as discussed in Section \ref{Conjecture}, it is natural to ask if there is a discrete version of such a technique that would reduce the boundedness of operators on $\ell^p(\Z^d)$ (with some assumptions on their kernel) to the boundedness of $H_{\dis}$ on $\ell^p(\Z)$. While  this does not seem to be the case for the setting of the discrete Riesz transform as defined in \eqref{DRT-1}, we can define closely related operators for which such a procedure is possible.

\subsection{Two-dimensional case}
We first consider the $d=2$ case  where a particularly simple expression for the discrete transform is available. For $j=1, 2$, from \eqref{classicalCRT-1} we have 
\begin{align*}
    R^{(i)}f(x) & = \frac{1}{2 \pi} \, \mathrm{p.v.\!} \int_{\R^2} f(y) \, \frac{x_i - y_i}{|x - y|^3} \, dy .
\end{align*}
Note that
$$ \frac{\partial^2 |y|}{\partial y_j^2} = \frac{1}{|y|} - \frac{y_j^2}{|y|^3} = \frac{|y|^2 - y_j^2}{|y|^3} = \frac{y_i^2}{|y|^3} \, ,$$
where $i = 1$, if $j = 2$ and $i = 2$, if  $j = 1$. Hence, the kernel of $R^{(i)}$ is given by
$$\frac{1}{y_i} \frac{\partial^2 |y|}{\partial y_j^2} \, .$$

Although not necessarily natural, this motivates the following  definition for a different variant of discrete Riesz transforms
$$ \drt^{(i)} f(n) = \frac{1}{2 \pi} \sum_{m \in \Z^2} f(m) \frac{|n-m+e_j| + |n-m-e_j| - 2 |n-m|}{{n_i - m_i} } \ind_{\{{m_i \ne n_i} \}}. $$

For simplicity, we consider ${i}= 1$. Fix $a, b \in \R$ and define the directional discrete Hilbert transform via the formula
\begin{align*}
    \dht_{a,b} f(n) & = \frac{1}{\pi} \sum_{m \in \Z^2} \frac{f(m)}{n_1 - m_1} \, \ind_{\{n_1 \ne m_1, n_2 - \tfloor{a n_1 + b} = m_2 - \tfloor{a m_1 + b}\}} .
\end{align*}
The intuition behind this definition is as follows. We split $\Z^2$ into an infinite family of ``one-dimensional'' sets
\begin{align*}
    F_{a,b,l} & = \{ (k, \tfloor{a k + b} + l) : k \in \Z \} = \{ n \in \Z^2 : n_2 - \tfloor{a n_1 + b} = l \} ,
\end{align*}
where $l$ takes arbitrary integer values. Then $\dht_{a,b}$ acts as a (one-dimensional) discrete Hilbert transform on each of the fibers $F_{a,b,l}$. In particular, the above interpretation combined with Corollary~\ref{BanKwaCorollary} immediately gives that 
\begin{align*}
    \|\dht_{a,b}\|_{p \to p} & = \cot(\tfrac{\pi}{2p^*}),
\end{align*}
which is the norm of the continuous Hilbert transform. 

\begin{theorem}\label{d=2}
For compactly supported $f : \Z^2 \to \R$, we have
\begin{align}\label{eq:mor}
    \drt^{({1})} f(n) & = \frac{1}{2} \int_0^1 \int_{-\infty}^\infty \frac{1}{(1 + a^2)^{3/2}} \, \dht_{a,b} f(n) da db, 
\end{align}
\end{theorem}
\begin{proof}
Formula~\eqref{eq:mor} is equivalent to
\begin{align*}
    & \frac{1}{2 \pi} \, \frac{|n - m + e_2| + |n - m - e_2| - 2 |n - m|}{n_1 - m_1} \\
    & \qquad = \frac{1}{2} \int_0^1 \int_{-\infty}^\infty \frac{1}{(1 + a^2)^{3/2}} \, \frac{1}{\pi} \, \frac{1}{m_1 - n_1} \, \ind_{\{m_2 - \tfloor{a m_1 + b} = n_2 - \tfloor{a n_1 + b}\}} \, da db,
\end{align*}
whenever $m_1 \ne n_1$. After elementary simplification, we need to prove that
\begin{align*}
    & |n - m + e_2| + |n - m - e_2| - 2 |n - m| \\
    & \qquad = \int_0^1 \int_{-\infty}^\infty \frac{1}{(1 + a^2)^{3/2}} \, \ind_{\{m_2 - \tfloor{a m_1 + b} = n_2 - \tfloor{a n_1 + b}\}} \, da db .
\end{align*}
We denote the right-hand side of the above equality by $I$.

The integrand in $I$ is a periodic function of $b$, with period $1$. Therefore, we may integrate with respect to $b$ over an arbitrary interval of unit length. For convenience, we choose this to be $[-a n_1, -a n_1 + 1)$, so that $\tfloor{a n_1 + b} = 0$, and we substitute $c = a n_1 + b$. It follows that
\begin{align*}
    I & = \int_{-\infty}^\infty \int_{-a n_1}^{-a n_1 + 1} \frac{1}{(1 + a^2)^{3/2}} \, \ind_{\{m_2 - \tfloor{a m_1 + b} = n_2\}} \, db da \\
    & = \int_{-\infty}^\infty \int_0^1 \frac{1}{(1 + a^2)^{3/2}} \, \ind_{\{m_2 - \tfloor{a (m_1 - n_1) + c} = n_2\}} \, dc da .
\end{align*}
We consider the case $m_1 > n_1$, the remaining case $m_1 < n_1$ being very similar. We have
\begin{align*}
    I & = \int_0^1 \int_{-\infty}^\infty \frac{1}{(1 + a^2)^{3/2}} \, \ind_{\{m_2 - \tfloor{a (m_1 - n_1) + c} = n_2\}} \, da dc \\
    & = \int_0^1 \int_{(m_2 - n_2 - c) / (m_1 - n_1)}^{(m_2 - n_2 - c + 1) / (m_1 - n_1)} \frac{1}{(1 + a^2)^{3/2}} \, da dc \\
    & = \int_0^1 \expr{\frac{m_2 - n_2 - c + 1}{\sqrt{(m_2 - n_2 - c + 1)^2 + (m_1 - n_1)^2}} - \frac{m_2 - n_2 - c}{\sqrt{(m_2 - n_2 - c)^2 + (m_1 - n_1)^2}}} dc \\
    & = \expr{-\sqrt{(m_2 - n_2)^2 + (m_1 - n_1)^2} + \sqrt{(m_2 - n_2 - 1)^2 + (m_1 - n_1)^2}} \\
    & \qquad - \expr{-\sqrt{(m_2 - n_2 + 1)^2 + (m_1 - n_1)^2} + \sqrt{(m_2 - n_2)^2 + (m_1 - n_1)^2}} \\
    & = |m - n - e_2| + |m - n + e_2| - 2 |m - n| ,
\end{align*}
as desired.
\end{proof}

Note that
\begin{align*}
    \int_{-\infty}^\infty \frac{1}{(1 + a^2)^{3/2}} \, da & = 2 .
\end{align*}
This, as in the classical method of rotations, immediately leads to the following estimate.

\begin{corollary}
We have $$\|\drt^{(1)}\|_{\ell^p \to \ell^p} \leq  \cot(\tfrac{\pi}{2p^*}).$$
\end{corollary}

On the other hand, we have the following perfect analogue of Lemma~\ref{lem:DiscreteAsymptotics} which gives the opposite inequality. 

\begin{lemma}\label{lem:RotationAsymptotics} 
If we denote by $K_{\drt^{(k)}}(n)$ the kernel of $\drt^{(k)}$, then
\begin{align*}
    \lim_{|n| \to \infty} |n|^2 |K_{\drt^{(k)}}(n) - K_{R_{\dis}^{(k)}}(n) | = 0.
\end{align*}
\end{lemma}
\begin{proof}
The argument boils down to an application of Taylor's theorem and elementary estimates. If $n = (n_1, n_2) \in \Z^{2}$, $|n| \geq 2$ and $n_1 \ne 0$, then
\begin{align*}
	|n - e_2| + |n + e_2| - 2 |n|
	& = \sqrt{n_1^2 + (n_2 - 1)^2} + \sqrt{n_1^2 + (n_2 + 1)^2} - 2 \sqrt{n_1^2 + n_2^2} \\
	& = \int_{-1}^1 (1 - |x|) \frac{n_1^2}{(n_1^2 + (n_2 - x)^2)^{3/2}} \, dx ,
\end{align*}
and hence
\begin{align*}
	& \biggl|\frac{|n - e_2| + |n + e_2| - 2 |n|}{n_1} - \frac{n_1}{(n_1^2 + n_2^2)^{3/2}}\biggr| \\
	& \qquad = \biggl| \int_{-1}^1 (1 - |x|) \frac{n_1}{(n_1^2 + (n_2 + x)^2)^{3/2}} \, dx - \int_{-1}^1 (1 - |x|) \frac{n_1}{(n_1^2 + n_2^2)^{3/2}} \, dx\biggr| \\
	& \qquad \leq \int_{-1}^1 (1 - |x|) \biggl| \frac{n_1}{(n_1^2 + (n_2 + x)^2)^{3/2}} - \frac{n_1}{(n_1^2 + n_2^2)^{3/2}} \biggr| dx .
\end{align*}
However, $|n + y e_2| \geq |n| - |y| \geq |n| - 1 \geq \frac{1}{2} |n|$ when $|y| \leq 1$ and $|n| \geq 2$, so that
\begin{align*}
	\biggl| \frac{n_1}{(n_1^2 + (n_2 + x)^2)^{3/2}} - \frac{n_1}{(n_1^2 + n_2^2)^{3/2}} \biggr|
	& = \biggl| \int_0^x \frac{3 n_1 (n_2 + y)}{(n_1^2 + (n_2 + y)^2)^{5/2}} \, dy \biggr| \\
	& \leq \frac{96 |n_1| (|n_2| + 1)}{|n|^5}
	\leq \frac{192}{|n|^3}
\end{align*}
when $|x| \leq 1$ and $|n| \ge 2$. It follows that
\begin{align*}
	\biggl|\frac{|n - e_2| + |n + e_2| - 2 |n|}{n_1} - \frac{n_1}{n_1^2 + n_2^2}\biggr|
	\leq \frac{192}{|n|^3} \int_{-1}^1 (1 - |x|) \, dx
	= \frac{192}{|n|^3}
\end{align*}
when $|n| \ge 2$ and $n_1 \ne 0$, and the desired result follows.
\end{proof}

With the above result at hand, we can follow the proof of Lemma~\ref{lem:DiscreteApproximation} and show that appropriately rescaled operators $\drt^{(k)}$ can be used to approximate (in the point-wise sense) the continuous Riesz transforms $R^{(k)}$, and consequently
\begin{align*}
	\|\drt^{(k)}\|_{\ell^p \to \ell^p} \geq \|R^{(k)}\|_{L^p \to L^p} = \cot(\tfrac{\pi}{2p^*}).
\end{align*}
We have thus proved the following result.

\begin{theorem}\label{thm:rot:two}
The two-dimensional discrete Riesz transforms, defined for $i = 1, 2$ by
\begin{align*}
    \drt^{(i)} f(n) & = \frac{1}{2 \pi} \sum_{m \in \Z^2} f(m) \frac{|n-m+e_j| + |n-m-e_j| - 2 |n-m|}{{n_i - m_i}
} \ind_{\{m_i \ne n_i\}} ,
\end{align*}
where $j = 2$ if $i = 1$ and $j = 1$ if $i = 2$, have norms on $\ell^p$ equal to the norms on $L^p$ of the corresponding continuous Riesz transforms: when $1 < p < \infty$, we have
\begin{align*}
	\|\drt^{(i)}\|_{\ell^p \to \ell^p} & = \|R^{(i)}\|_{L^p \to L^p} = \cot(\tfrac{\pi}{2p^*}).
\end{align*}
\end{theorem}

\subsection{Higher dimensions}
The same approach works in higher dimensions, too, but a closed-form expression for the corresponding kernel does not seem available. When $d \ge 2$, we define
\begin{align}\label{rotationRiesz}
    \drt^{(k)} f(n) & = \sum_{m \in \Z^d} K_{\drt^{(k)}}(n - m) f(m) ,
\end{align}
where the kernel for $k = 1$ is given in an integral form as follows. If $n = (n_1, \tilde{n}) \in \Z^d$ with $n_1 \in \Z$ and $\tilde{n} = (n_2, \ldots, n_d) \in \Z^{d-1}$, and if $n_1 > 0$, then
\begin{align*}
    K_{\drt^{(1)}}(n) & = \frac{1}{\pi n_1} \times C_d \int_{[0, 1]^{d - 1}} \int_{\frac{\tilde{n} - b}{n_1} + [0, \frac{1}{n_1})^{d - 1}} \frac{1}{(1 + |a|^2)^{(d + 1) / 2}} \, da db ,
\end{align*}
where $C_d$ is related to the constant $c_d$ in \eqref{RieszK} via
\begin{align*}
    C_d & = \biggl( \int_{\R^{d - 1}} \frac{1}{(1 + |a|^2)^{(d + 1) / 2}} \, da \biggr)^{-1} = \pi c_d .
\end{align*}
Furthermore, when $n_1 < 0$, then $K_{\drt^{(1)}}(n) = -K_{\drt^{(1)}}(-n)$. For a general $k$, the kernel $K_{\drt^{(k)}}(n)$ is equal to $K_{\drt^{(1)}}(n')$, where $n'$ is obtained from $n$ by swapping the first and $k$-th coordinate.

By definition, as in the two-dimensional case, for compactly supported $f : \Z^d \to \R$, we have
\begin{align}\label{MethodRot}
    \drt^{(1)} f(n) = C_d \int_{[0, 1]^{d - 1}} \int_{\R^{d - 1}} \frac{1}{(1 + |a|^2)^{(d + 1)/2}} \, \dht_{a,b} f(n) da db, 
\end{align}
where $\dht_{a, b}$ acts as the discrete Hilbert transform with respect to $n_1$ on each of the fibers
\begin{align*}
    F_{a,b,l} & = \{(n_1, \lfloor a n_1 + b \rfloor + l) : n_1 \in \Z\} = \{ n \in \Z^d : n_j - \lfloor a_j n_1 + b_j \rfloor = l_j, \, j = 2, 3, \ldots, d \} ,
\end{align*}
with $l \in \Z^{d - 1}$ (here we understand that the floor function in $\lfloor a n_1 + b \rfloor$ acts component-wise). Therefore,
\begin{align*}
    \|\drt^{(1)}\|_{\ell^p \to \ell^p} & \leq \cot(\tfrac{\pi}{2p^*}).
\end{align*}
On the other hand, below we prove that (as in Lemma \ref{lem:RotationAsymptotics} for $d=2$)
\begin{align}\label{eq:rot:higher}
	\lim_{|n| \to \infty} |n|^d |K_{\drt^{(1)}}(n) - K_{R_{\dis}^{(1)}}(n) | = \lim_{|n| \to \infty} |n|^d \biggl|K_{\drt^{(1)}}(n) - c_d \frac{n_1}{|n|^{(d + 1)/2}} \biggr| = 0 .
\end{align}
Once this is shown, by the same argument as in the case of $d=2$, we find that
\begin{align*}
	\|\drt^{(1)}\|_{\ell^p \to \ell^p} & \geq \|R^{(1)}\|_{L^p \to L^p} = \cot(\tfrac{\pi}{2p^*}).
\end{align*}
Thus, we conclude that in fact the norms are equal. We state this as a theorem.

\begin{theorem}\label{thm:rot:higher}
The discrete Riesz transforms $\drt^{(k)}$ introduced above have norms on $\ell^p$ equal to the norms on $L^p$ of the corresponding continuous Riesz transforms: when $1 < p < \infty$, we have
\begin{align*}
	\|\drt^{(k)}\|_{\ell^p \to \ell^p} & = \|R^{(k)}\|_{L^p \to L^p} = \cot(\tfrac{\pi}{2p^*}).
\end{align*}
\end{theorem}
\begin{proof}
We only need to prove~\eqref{eq:rot:higher}. As before, we write $n = (n_1, \tilde{n})$, where $\tilde{n} = (n_2, \ldots, n_d)$, and since both kernels are odd functions of $n_1$, without loss of generality we assume that $n_1 > 0$. We have
\begin{align*}
    & K_{\drt^{(1)}}(n) - K_{R_{\dis}^{(1)}}(n) \\
    & \qquad = \frac{c_d}{n_1} \int_{[0, 1]^{d - 1}} \int_{\frac{\tilde{n} - b}{n_1} + [0, \frac{1}{n_1})^{d - 1}} \frac{1}{(1 + |a|^2)^{(d + 1) / 2}} \, da db - c_d \, \frac{n_1}{(n_1^2 + |\tilde{n}|^2)^{(d + 1) / 2}} \\
    & \qquad = \frac{c_d}{n_1^d} \int_{[0, 1]^{d - 1}} \int_{[0, 1)^{d - 1}} \frac{n_1^{d + 1}}{(n_1^2 + |\tilde{n} - b + v|^2)^{(d + 1) / 2}} \, dv db - c_d \, \frac{n_1}{(n_1^2 + |\tilde{n}|^2)^{(d + 1) / 2}} \\
    & \qquad = c_d n_1 \int_{[0, 1]^{d - 1}} \int_{[0, 1)^{d - 1}} \biggl( \frac{1}{(n_1^2 + |\tilde{n} - b + v|^2)^{(d + 1) / 2}} - \frac{1}{(n_1^2 + |\tilde{n}|^2)^{(d + 1) / 2}} \biggr) dv db .
\end{align*}
Since in the given region of integration we have $-b + v \in [-1, 1]^{d - 1}$, it follows that
\begin{align*}
    |K_{\drt^{(1)}}(n) - K_{R_{\dis}^{(1)}}(n) | & \leq c_d n_1 \sup_{w \in [-1, 1]^{d - 1}} \biggl| \frac{1}{(n_1^2 + |\tilde{n} + w|^2)^{(d + 1) / 2}} - \frac{1}{(n_1^2 + |\tilde{n}|^2)^{(d + 1) / 2}} \biggr| .
\end{align*}
We now simply use the mean value theorem for the function $t \mapsto t^{-d - 1}$ evaluated at $t_1 = \sqrt{n_1^2 + |\tilde{n} + w|^2}$ and $t_2 = \sqrt{n_1^2 + |\tilde{n}|^2} = |n|$: we have
\begin{align*}
    \biggl| \frac{1}{t_1^{d + 1}} - \frac{1}{t_2^{d + 1}} \biggr| \leq |t_1 - t_2| \times \frac{d + 1}{\min\{t_1^{d + 2}, t_2^{d + 2}\}} \, .
\end{align*}
Since $|t_1 - t_2| \leq |w| \leq \sqrt{d - 1}$, we have $t_1 \geq t_2 - \sqrt{d - 1} = |n| - \sqrt{d - 1} \ge \frac{1}{2} |n|$ when $|n|$ is large enough, and thus
\begin{align*}
    \biggl| \frac{1}{(n_1^2 + |\tilde{n} + w|^2)^{(d + 1) / 2}} - \frac{1}{(n_1^2 + |\tilde{n}|^2)^{(d + 1) / 2}} \biggr| & = \biggl| \frac{1}{t_1^{d + 1}} - \frac{1}{t_2^{d + 1}} \biggr| \leq \sqrt{d - 1} \times \frac{2^{d + 2} (d + 1)}{|n|^{d + 2}}
\end{align*}
when $|n|$ is large enough. We thus conclude that when $|n|$ is large enough, then
\begin{align*}
    |K_{\drt^{(1)}}(n) - K_{R_{\dis}^{(1)}}(n) | & \leq 2^{d + 1} (d + 1) c_d \sqrt{d - 1} \, \frac{n_1}{|n|^{d + 2}} \, .
\end{align*}
The right-hand side multiplied by $|n|^d$ goes to zero as $|n| \to \infty$, and the proof is complete.
\end{proof}

We remark that a similar construction of the discrete Riesz transform using the method of rotations can be carried out using the probabilistic discrete Hilbert transform $T_{\bH}$ instead of the discrete Hilbert transform $H_\dis$ applied above. This procedure will lead to a transform with the same norm on $\ell^p$, but with a kernel which is greater in absolute value than the kernel of $\drt^{(k)}$ (in the point-wise sense). However, we did not pursue this direction.

\section{Numerical comparison of kernels}\label{Numerical}
We end with some remarks on numerical comparisons on the kernels for the discrete operators $R_\dis^{(k)}$, $\drt^{(k)}$, and $T_{\bH^{(k)}}$. Numerical evaluation of the kernels for $R_\dis^{(k)}$ and $\drt^{(k)}$ when $d=2$ presents no difficulties. The situation is quite different for $T_{\bH^{(k)}}$, which is given by a triple integral involving the periodic Poisson kernel $h(x, y)$.

In the following numerical simulations we used {\it Wolfram Mathematica 10} and a relatively naive approach, which may lead to significant errors. That said, the outcome turned out to be relatively stable when we varied the parameters, so we believe that our approximations are correct to roughly fourth significant digit.

The periodic Poisson kernel $h(x, y)$ was approximated using the definition~\eqref{doob-h} when $y \leq \tfrac{1}{4}$ and using the expression~\eqref{FourierInv} based on the Poisson summation formula when $y \ge \tfrac{1}{4}$. Additionally, since $h(x, y)$ converges to $1$ exponentially fast as $y \to \infty$, for $y \ge 10$ we simply approximated $h(x, y)$ by a constant $1$. To speed up numerical integration, we evaluated the above numerical approximation to $h(x, y)$ in a limited number of points, and then we used appropriate interpolation to find the values of $h(x, y)$ between these points.

Numerical integration was done using standard methods available in {\it Mathematica}. Although {\it Mathematica} warned about slow convergence, the estimated error of numerical integration appears to be less significant than the errors in approximation of the periodic Poisson kernel.

\begin{figure}
\includegraphics[width=10cm]{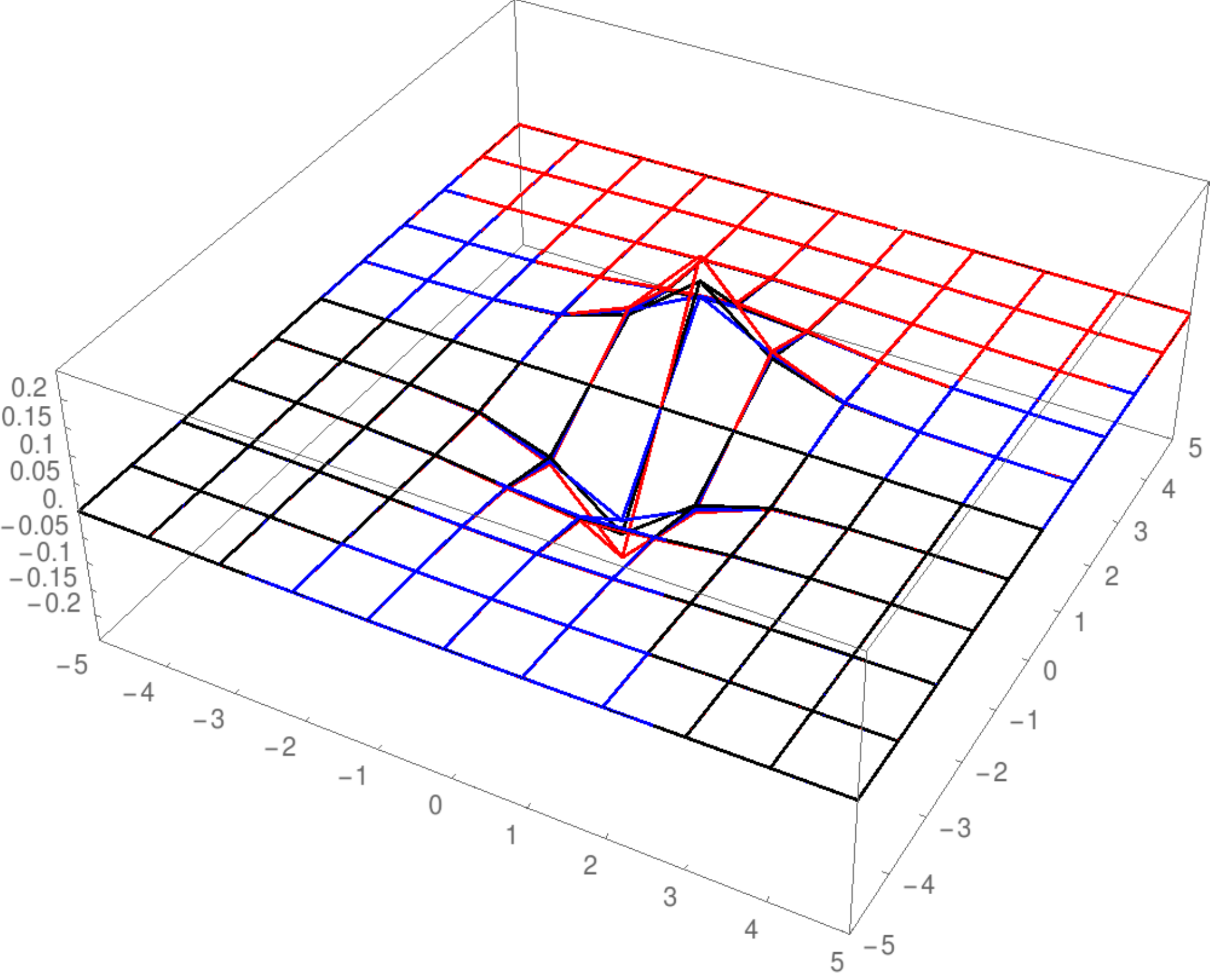}
\caption{Values of the kernels of the three transforms: $\bK_{\bH^{(1)}}(n_1, n_2)$ (red), $K_{\drt^{(1)}}(n_1, n_2)$ (blue), and $ K_{R_\dis^{(1)}} (n_1, n_2)$ (black) for $n_1, n_2 \in \{-5, -4, \ldots, 5\}$.}
\label{fig:values}
\end{figure}

\begin{figure}
\includegraphics[width=10cm]{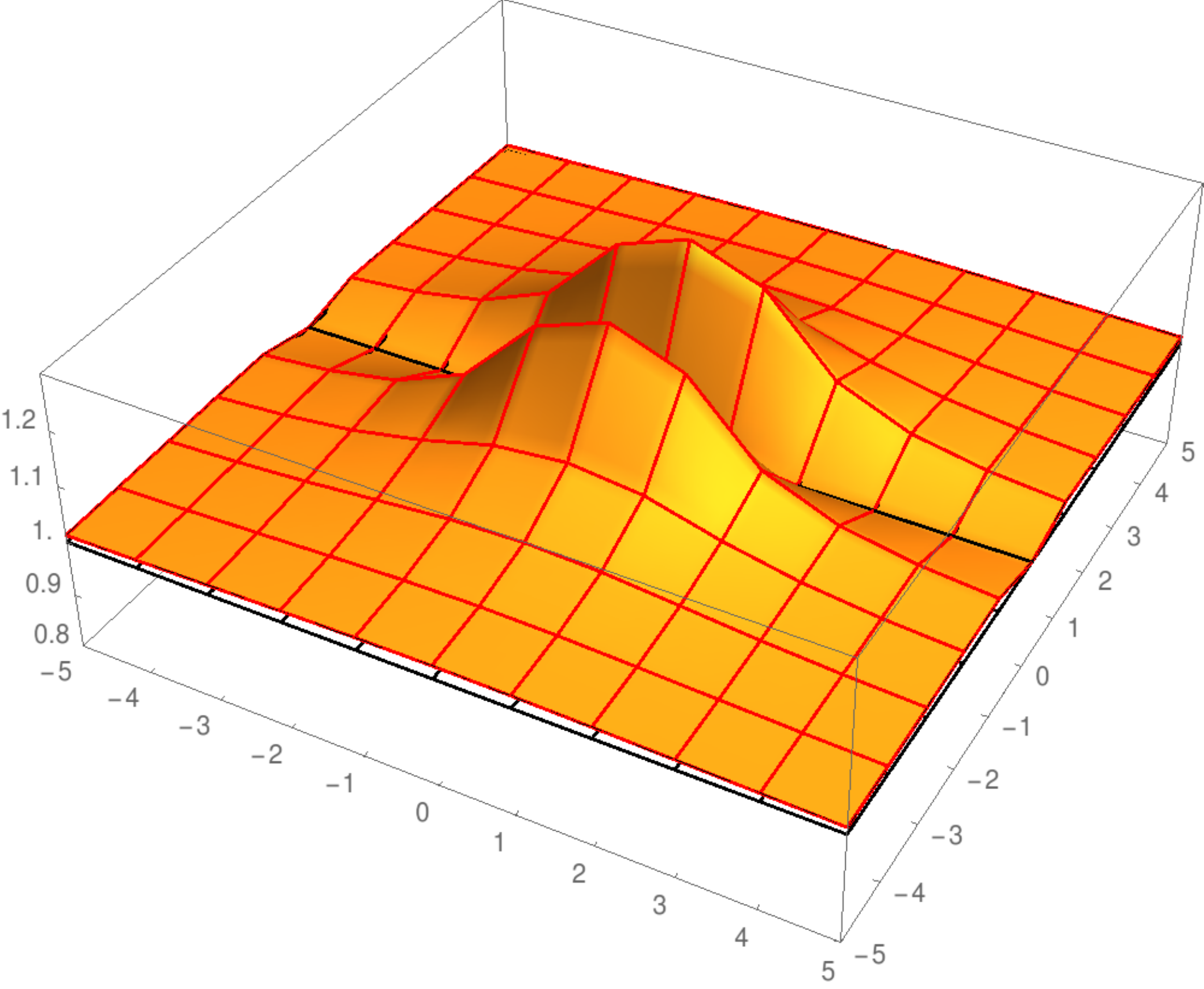}
\caption{Ratio between the kernels of two transforms: $\bK_{\bH^{(1)}}(n_1, n_2) / K_{R_\dis^{(1)}} (n_1, n_2)$ (red) compared with constant $1$ (black) for $n_1, n_2 \in \{-5, -4, \ldots, 5\}$. When $n_1 = 0$, we set $0 / 0 = 1$.}
\label{fig:rel:prob}
\bigskip
\includegraphics[width=10cm]{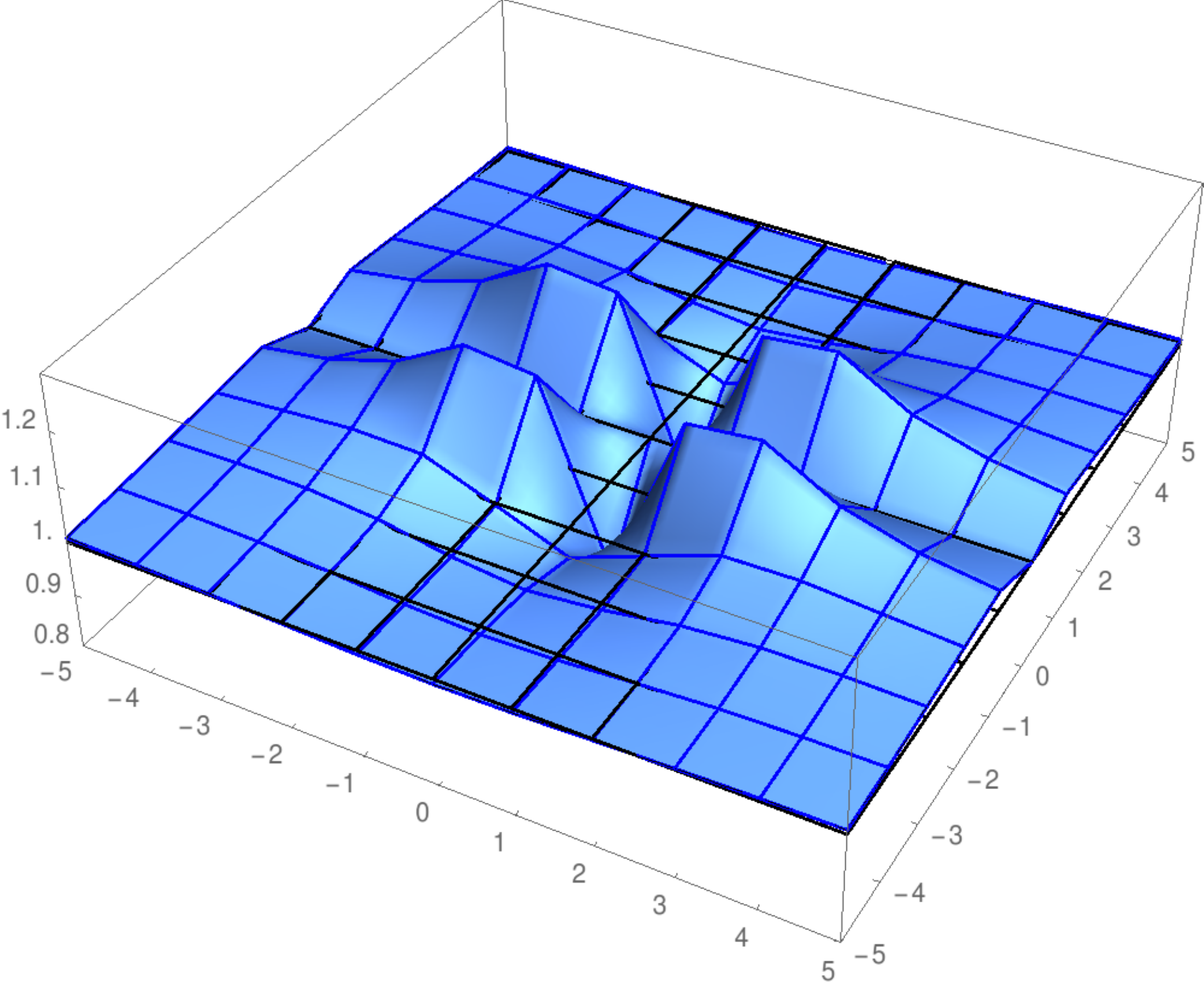}
\caption{Ratio between the kernels of two transforms: $K_{\drt^{(1)}}(n_1, n_2) / K_{R_\dis^{(1)}} (n_1, n_2)$ (blue) compared with constant $1$ (black) for $n_1, n_2 \in \{-5, -4, \ldots, 5\}$. When $n_1 = 0$, we set $0 / 0 = 1$.}
\label{fig:rel:rot}
\end{figure}

The values of the three kernels are shown in Figure~\ref{fig:values}. The ratio between the kernels of the probabilistic and the CZ discrete Riesz transform are shown in Figure~\ref{fig:rel:prob}, while a similar plot for the method of rotations and the Riesz transform $R_{\dis}^1$ is shown in Figure~\ref{fig:rel:rot}. Numerical results are presented in Tables~\ref{tbl:values}, \ref{tbl:rel1} and~\ref{tbl:rel2}.

Our simulations suggest that there is no general point-wise relation between the kernels of $\drt^{(k)}$ and $R_\dis^{(k)}$, nor there is one between the kernels of $T_{\bH^{(k)}}$ and $\drt^{(k)}$. However, it seems that the kernel of $T_{\bH^{(k)}}$ is always greater (in the absolute value) than the kernel of $R_\dis^{(k)}$. This leads to the following conjecture which we know is true for $d = 1$ by~\eqref{kerneld-1}.

\begin{conjecture}\label{conjPointwise}
For all $d\geq 2$, we have $|\bK_{\bH^{(k)}}(n)| \ge | K_{R_\dis^{(k)}}(n) |$ for every $n \in \Z^d$. 
\end{conjecture} 

The above numerical findings give a little insight into Question~\ref{convZd}, which asks whether $K_{R_\dis^{(k)}}$ is the convolution of $\bK_{\bH^{(k)}}$ with some probability kernel. Indeed, although intuitively point-wise domination asserted in Conjecture~\ref{conjPointwise} appears to be a necessary condition for a positive answer to Question~\ref{convZd}, neither of these statements implies the other one.

On the other hand, our calculations strongly suggest that in dimension $d = 2$ the maximum of $K_{\drt^{(1)}}(n_1, n_2)$ is strictly smaller than the maximum of $ K_{R_\dis^{(1)}} (n_1, n_2)$; both maxima are attained at $(n_1, n_2) = (1, 0)$. If this is indeed the case, then $ K_{R_\dis^{(1)}} $ is clearly {\it not} a convolution of $K_{\drt^{(1)}}$ and a probability kernel. Thus, we expect that the analogue of Question~\ref{convZd} for $\drt^{(k)}$ instead of $T_{\bH^{(k)}}$ has a negative answer.

Finally, one can ask if the analogue of Question~\ref{convZd} holds for the discrete Riesz transform obtained with the method of rotations, but using the probabilistic discrete Hilbert transform $T_{\bH}$ instead of the usual discrete Hilbert transform $H_\dis$. We did not attempt to answer this question.

\begin{table}
\footnotesize
\caption{{Values of the kernels of the three transforms: $\bK_{\bH^{(1)}}(n_1, n_2)$ (top row), $K_{\drt^{(1)}}(n_1, n_2)$ (middle row) and $ K_{R_\dis^{(1)}} (n_1, n_2)$ (bottom row) for $n_1, n_2 \in \{0, 1, \ldots, 5\}$, $n_1 \ne 0$. The largest value in each cell is set in bold, while the smallest one is given in slanted type.}}
\label{tbl:values}
\newcommand{\tabledata}[3]{\begin{tabular}{c}#1\\#2\\#3\end{tabular}}
\newcommand{\one}[1]{\bf #1}
\newcommand{\two}[1]{#1}
\newcommand{\thr}[1]{\sl #1}
\begin{tabular}{cc|c|c|c|c|c|c}
&& \multicolumn{6}{c}{$n_2$} \\
&& 0 & 1 & 2 & 3 & 4 & 5 \\
\cline{1-8}
& 1 &
\tabledata{\one{0.2051}}{\thr{0.1318}}{\two{0.1592}} &
\tabledata{\one{0.0698}}{\two{0.0649}}{\thr{0.0563}} &
\tabledata{\two{0.0158}}{\one{0.0166}}{\thr{0.0142}} &
\tabledata{\two{0.0053}}{\one{0.0055}}{\thr{0.0050}} &
\tabledata{\two{0.0024}}{\one{0.0024}}{\thr{0.0023}} &
\tabledata{\two{0.0012}}{\one{0.0012}}{\thr{0.0012}} \\
\cline{2-8}
& 2 &
\tabledata{\one{0.0446}}{\thr{0.0376}}{\two{0.0398}} &
\tabledata{\one{0.0315}}{\thr{0.0284}}{\two{0.0285}} &
\tabledata{\one{0.0151}}{\two{0.0147}}{\thr{0.0141}} &
\tabledata{\two{0.0071}}{\one{0.0071}}{\thr{0.0068}} &
\tabledata{\two{0.0037}}{\one{0.0037}}{\thr{0.0036}} &
\tabledata{\two{0.0021}}{\one{0.0021}}{\thr{0.0020}} \\
\cline{2-8}
$n_1$ & 3 &
\tabledata{\one{0.0188}}{\thr{0.0172}}{\two{0.0177}} &
\tabledata{\one{0.0160}}{\thr{0.0149}}{\two{0.0151}} &
\tabledata{\one{0.0106}}{\two{0.0103}}{\thr{0.0102}} &
\tabledata{\one{0.0065}}{\two{0.0064}}{\thr{0.0063}} &
\tabledata{\one{0.0039}}{\two{0.0039}}{\thr{0.0038}} &
\tabledata{\one{0.0025}}{\two{0.0025}}{\thr{0.0024}} \\
\cline{2-8}
& 4 &
\tabledata{\one{0.0103}}{\thr{0.0098}}{\two{0.0099}} &
\tabledata{\one{0.0094}}{\thr{0.0090}}{\two{0.0091}} &
\tabledata{\one{0.0073}}{\thr{0.0071}}{\two{0.0071}} &
\tabledata{\one{0.0052}}{\two{0.0051}}{\thr{0.0051}} &
\tabledata{\one{0.0036}}{\two{0.0036}}{\thr{0.0035}} &
\tabledata{\one{0.0025}}{\two{0.0025}}{\thr{0.0024}} \\
\cline{2-8}
& 5 &
\tabledata{\one{0.0065}}{\thr{0.0063}}{\two{0.0064}} &
\tabledata{\one{0.0061}}{\thr{0.0060}}{\two{0.0060}} &
\tabledata{\one{0.0052}}{\thr{0.0051}}{\two{0.0051}} &
\tabledata{\one{0.0041}}{\two{0.0040}}{\thr{0.0040}} &
\tabledata{\one{0.0031}}{\two{0.0030}}{\thr{0.0030}} &
\tabledata{\one{0.0023}}{\two{0.0023}}{\thr{0.0023}}
\end{tabular}
\end{table}

\begin{table}
\footnotesize
\caption{{Values of the ratios $\bK_{\bH^{(1)}}(n_1, n_2) / K_{R_\dis^{(1)}} (n_1, n_2)$ of kernels of two transforms for $n_1, n_2 \in \{0, 1, \ldots, 5\}$, $n_1 \ne 0$.}}
\label{tbl:rel1}
\newcommand{\one}[1]{\bf #1}
\newcommand{\two}[1]{#1}
\begin{tabular}{cc|c|c|c|c|c|c}
&& \multicolumn{6}{c}{$n_2$} \\
&& 0 & 1 & 2 & 3 & 4 & 5 \\
\cline{1-8}
& 1 &
\one{1.2885} &
\one{1.2413} &
\one{1.1127} &
\one{1.0593} &
\one{1.0356} &
\one{1.0235} \\
\cline{2-8}
& 2 &
\one{1.1200} &
\one{1.1067} &
\one{1.0717} &
\one{1.0458} &
\one{1.0303} &
\one{1.0211} \\
\cline{2-8}
$n_1$ & 3 &
\one{1.0615} &
\one{1.0567} &
\one{1.0450} &
\one{1.0333} &
\one{1.0243} &
\one{1.0180} \\
\cline{2-8}
& 4 &
\one{1.0364} &
\one{1.0345} &
\one{1.0298} &
\one{1.0241} &
\one{1.0191} &
\one{1.0150} \\
\cline{2-8}
& 5 &
\one{1.0239} &
\one{1.0230} &
\one{1.0208} &
\one{1.0179} &
\one{1.0149} &
\one{1.0123}
\end{tabular}
\bigskip
\footnotesize
\caption{{Values of the ratios $K_{\drt^{(1)}}(n_1, n_2) / K_{R_\dis^{(1)}} (n_1, n_2)$ of kernels of two transforms for $n_1, n_2 \in \{0, 1, \ldots, 5\}$, $n_1 \ne 0$.}}
\label{tbl:rel2}
\begin{tabular}{cc|c|c|c|c|c|c}
&& \multicolumn{6}{c}{$n_2$} \\
&& 0 & 1 & 2 & 3 & 4 & 5 \\
\cline{1-8}
& 1 &
\two{0.8284} &
\one{1.1530} &
\one{1.1667} &
\one{1.0947} &
\one{1.0574} &
\one{1.0379} \\
\cline{2-8}
& 2 &
\two{0.9443} &
\two{0.9959} &
\one{1.0452} &
\one{1.0483} &
\one{1.0385} &
\one{1.0292} \\
\cline{2-8}
$n_1$ & 3 &
\two{0.9737} &
\two{0.9873} &
\one{1.0094} &
\one{1.0205} &
\one{1.0221} &
\one{1.0199} \\
\cline{2-8}
& 4 &
\two{0.9848} &
\two{0.9897} &
\two{0.9997} &
\one{1.0077} &
\one{1.0116} &
\one{1.0125} \\
\cline{2-8}
& 5 &
\two{0.9902} &
\two{0.9923} &
\two{0.9972} &
\one{1.0022} &
\one{1.0057} &
\one{1.0075}
\end{tabular}
\end{table}

\newpage

\subsection*{Acknowledgments}  
We express our thanks to Renming Song for a helpful  conversation on positive-definite functions and reference \cite{Jacob1}, and to Tomasz Szarek for reference \cite{Kov}.  
We are grateful to  Mark Ashbaugh for useful conversations on special functions. The research for this paper was conducted  while the second author was a J. L. Doob research assistant professor at the University of Illinois at Urbana-Champaign.  With deep appreciation and respect,  R. Ba\~nuelos acknowledges the countless conversations he had for almost 40 years with the late Richard Gundy on topics related to those of this paper.

\bibliography{ref.bib}

\begin{bibdiv}
\begin{biblist}

\bib{ArcDomPet2}{article}{
      author={Arcozzi, Nicola},
      author={Domelevo, Komla},
      author={Petermichl, Stefanie},
       title={Second order {R}iesz transforms on multiply-connected {L}ie
  groups and processes with jumps},
        date={2016},
        ISSN={0926-2601},
     journal={Potential Anal.},
      volume={45},
      number={4},
       pages={777\ndash 794},
  url={https://doi-org.ezproxy.lib.purdue.edu/10.1007/s11118-016-9566-x},
      review={\MR{3558360}},
}

\bib{ADP20}{article}{
      author={Arcozzi, Nicola},
      author={Domelevo, Komla},
      author={Petermichl, Stefanie},
       title={Discrete {H}ilbert transform \`a la {G}undy-{V}aropoulos},
        date={2020},
        ISSN={0002-9939},
     journal={Proc. Amer. Math. Soc.},
      volume={148},
      number={6},
       pages={2433\ndash 2446},
      review={\MR{4080886}},
}

\bib{Ban86}{article}{
      author={Ba\~{n}uelos, Rodrigo},
       title={Martingale transforms and related singular integrals},
        date={1986},
        ISSN={0002-9947},
     journal={Trans. Amer. Math. Soc.},
      volume={293},
      number={2},
       pages={547\ndash 563},
         url={https://doi.org/10.2307/2000021},
      review={\MR{816309}},
}

\bib{BanGoodLamb}{article}{
      author={Ba\~{n}uelos, Rodrigo},
       title={A sharp good-{$\lambda$} inequality with an application to
  {R}iesz transforms},
        date={1988},
        ISSN={0026-2285,1945-2365},
     journal={Michigan Math. J.},
      volume={35},
      number={1},
       pages={117\ndash 125},
         url={https://doi.org/10.1307/mmj/1029003685},
      review={\MR{931943}},
}

\bib{Ban}{article}{
      author={Ba\~{n}uelos, Rodrigo},
       title={The foundational inequalities of {D}. {L}. {B}urkholder and some
  of their ramifications},
        date={2010},
        ISSN={0019-2082},
     journal={Illinois J. Math.},
      volume={54},
      number={3},
       pages={789\ndash 868 (2012)},
         url={http://projecteuclid.org/euclid.ijm/1336049979},
      review={\MR{2928339}},
}

\bib{BBL20}{article}{
      author={Ba\~{n}uelos, Rodrigo},
      author={Baudoin, Fabrice},
      author={Chen, Li},
       title={Gundy-{V}aropoulos martingale transforms and their projection
  operators on manifolds and vector bundles},
        date={2020},
        ISSN={0025-5831},
     journal={Math. Ann.},
      volume={378},
      number={1-2},
       pages={359\ndash 388},
      review={\MR{4150921}},
}

\bib{BBLS2021}{article}{
      author={Ba\~{n}uelos, Rodrigo},
      author={Baudoin, Fabrice},
      author={Chen, Li},
      author={Sire, Yannick},
       title={Multiplier theorems via martingale transforms},
        date={2021},
        ISSN={0022-1236},
     journal={J. Funct. Anal.},
      volume={281},
      number={9},
       pages={Paper No. 109188, 37},
      review={\MR{4295971}},
}

\bib{BanBog}{article}{
      author={Ba\~{n}uelos, Rodrigo},
      author={Bogdan, Krzysztof},
       title={L\'{e}vy processes and {F}ourier multipliers},
        date={2007},
        ISSN={0022-1236},
     journal={J. Funct. Anal.},
      volume={250},
      number={1},
       pages={197\ndash 213},
         url={https://doi.org/10.1016/j.jfa.2007.05.013},
      review={\MR{2345912}},
}

\bib{BanJan}{article}{
      author={Ba\~{n}uelos, Rodrigo},
      author={Janakiraman, Prabhu},
       title={{$L^p$}-bounds for the {B}eurling-{A}hlfors transform},
        date={2008},
        ISSN={0002-9947},
     journal={Trans. Amer. Math. Soc.},
      volume={360},
      number={7},
       pages={3603\ndash 3612},
  url={https://doi-org.ezproxy.lib.purdue.edu/10.1090/S0002-9947-08-04537-6},
      review={\MR{2386238}},
}

\bib{BanKwa}{article}{
      author={Ba\~{n}uelos, Rodrigo},
      author={Kwa\'{s}nicki, Mateusz},
       title={On the {$\ell^p$}-norm of the discrete {H}ilbert transform},
        date={2019},
        ISSN={0012-7094},
     journal={Duke Math. J.},
      volume={168},
      number={3},
       pages={471\ndash 504},
      review={\MR{3909902}},
}

\bib{BanKwa1}{article}{
      author={Ba\~{n}uelos, Rodrigo},
      author={Kwa\'{s}nicki, Mateusz},
       title={The {$\ell^p$} norm of the {R}iesz-{T}itchmarsh transform for
  even integer {$p$}},
        date={2024},
        ISSN={0024-6107,1469-7750},
     journal={J. Lond. Math. Soc. (2)},
      volume={109},
      number={4},
       pages={Paper No. e12888, 21},
         url={https://doi.org/10.1112/jlms.12888},
      review={\MR{4727420}},
}

\bib{BanMen}{article}{
      author={Ba\~{n}uelos, Rodrigo},
      author={M\'{e}ndez-Hern\'{a}ndez, P.~J.},
       title={Space-time {B}rownian motion and the {B}eurling-{A}hlfors
  transform},
        date={2003},
        ISSN={0022-2518},
     journal={Indiana Univ. Math. J.},
      volume={52},
      number={4},
       pages={981\ndash 990},
      review={\MR{2001941}},
}

\bib{BanOse1}{article}{
      author={Ba\~{n}uelos, Rodrigo},
      author={Os\c{e}kowski, Adam},
       title={Martingales and sharp bounds for {F}ourier multipliers},
        date={2012},
        ISSN={1239-629X},
     journal={Ann. Acad. Sci. Fenn. Math.},
      volume={37},
      number={1},
       pages={251\ndash 263},
  url={https://mathscinet-ams-org.libezp.lib.lsu.edu/mathscinet-getitem?mr=2920438},
      review={\MR{2920438}},
}

\bib{BO15}{article}{
      author={Ba\~{n}uelos, Rodrigo},
      author={Os\c{e}kowski, Adam},
       title={Sharp martingale inequalities and applications to {R}iesz
  transforms on manifolds, {L}ie groups and {G}auss space},
        date={2015},
        ISSN={0022-1236},
     journal={J. Funct. Anal.},
      volume={269},
      number={6},
       pages={1652\ndash 1713},
         url={https://mathscinet.ams.org/mathscinet-getitem?mr=3373431},
      review={\MR{3373431}},
}

\bib{BO18}{article}{
      author={Ba\~{n}uelos, Rodrigo},
      author={Os\c{e}kowski, Adam},
       title={Stability in {B}urkholder's differentially subordinate
  martingales inequalities and applications to {F}ourier multipliers},
        date={2018},
        ISSN={0021-7824},
     journal={J. Math. Pures Appl. (9)},
      volume={119},
       pages={1\ndash 44},
  url={https://doi-org.ezproxy.lib.purdue.edu/10.1016/j.matpur.2018.09.001},
      review={\MR{3862142}},
}

\bib{BanWang}{article}{
      author={Ba\~{n}uelos, Rodrigo},
      author={Wang, Gang},
       title={Sharp inequalities for martingales with applications to the
  {B}eurling-{A}hlfors and {R}iesz transforms},
        date={1995},
        ISSN={0012-7094},
     journal={Duke Math. J.},
      volume={80},
      number={3},
       pages={575\ndash 600},
      review={\MR{1370109}},
}

\bib{Bak1}{incollection}{
      author={Bakry, Dominique},
       title={The {R}iesz transforms associated with second order differential
  operators},
        date={1989},
   booktitle={Seminar on {S}tochastic {P}rocesses, 1988 ({G}ainesville, {FL},
  1988)},
      series={Progr. Probab.},
      volume={17},
   publisher={Birkh\"{a}user Boston, Boston, MA},
       pages={1\ndash 43},
         url={https://doi-org.ezproxy.lib.purdue.edu/10.1214/aop/1176991490},
      review={\MR{990472}},
}

\bib{Bass}{book}{
      author={Bass, Richard~F.},
       title={Probabilistic techniques in analysis},
      series={Probability and its Applications (New York)},
   publisher={Springer-Verlag, New York},
        date={1995},
        ISBN={0-387-94387-0},
      review={\MR{1329542}},
}

\bib{Bill}{book}{
      author={Billingsley, Patrick},
       title={Probability and measure},
     edition={Third},
      series={Wiley Series in Probability and Mathematical Statistics},
   publisher={John Wiley \& Sons, Inc., New York},
        date={1995},
        ISBN={0-471-00710-2},
        note={A Wiley-Interscience Publication},
      review={\MR{1324786}},
}

\bib{Pie4}{article}{
      author={Bober, Jonathan},
      author={Carneiro, Emanuel},
      author={Hughes, Kevin},
      author={Pierce, Lillian~B.},
       title={On a discrete version of {T}anaka's theorem for maximal
  functions},
        date={2012},
        ISSN={0002-9939},
     journal={Proc. Amer. Math. Soc.},
      volume={140},
      number={5},
       pages={1669\ndash 1680},
  url={https://doi-org.ezproxy.lib.purdue.edu/10.1090/S0002-9939-2011-11008-6},
      review={\MR{2869151}},
}

\bib{Bur84}{article}{
      author={Burkholder, D.~L.},
       title={Boundary value problems and sharp inequalities for martingale
  transforms},
        date={1984},
        ISSN={0091-1798},
     journal={Ann. Probab.},
      volume={12},
      number={3},
       pages={647\ndash 702},
         url={https://mathscinet.ams.org/mathscinet-getitem?mr=744226},
      review={\MR{744226}},
}

\bib{CZ}{article}{
      author={Calderon, A.~P.},
      author={Zygmund, A.},
       title={On the existence of certain singular integrals},
        date={1952},
        ISSN={0001-5962},
     journal={Acta Math.},
      volume={88},
       pages={85\ndash 139},
      review={\MR{52553}},
}

\bib{CarDra2}{article}{
      author={Carbonaro, Andrea},
      author={Dragi\v{c}evi\'{c}, Oliver},
       title={Bellman function and linear dimension-free estimates in a theorem
  of {B}akry},
        date={2013},
        ISSN={0022-1236},
     journal={J. Funct. Anal.},
      volume={265},
      number={7},
       pages={1085\ndash 1104},
  url={https://doi-org.ezproxy.lib.purdue.edu/10.1016/j.jfa.2013.05.031},
      review={\MR{3073250}},
}

\bib{CarDra1}{article}{
      author={Carbonaro, Andrea},
      author={Dragi\v{c}evi\'{c}, Oliver},
       title={Convexity of power functions and bilinear embedding for
  divergence-form operators with complex coefficients},
        date={2020},
        ISSN={1435-9855},
     journal={J. Eur. Math. Soc. (JEMS)},
      volume={22},
      number={10},
       pages={3175\ndash 3221},
         url={https://doi-org.ezproxy.lib.purdue.edu/10.4171/jems/984},
      review={\MR{4153106}},
}

\bib{Car1982}{article}{
      author={Cartwright, M.~L.},
       title={Manuscripts of {H}ardy, {L}ittlewood, {M}arcel {R}iesz and
  {T}itchmarsh},
        date={1982},
        ISSN={0024-6093},
     journal={Bull. London Math. Soc.},
      volume={14},
      number={6},
       pages={472\ndash 532},
      review={\MR{679927}},
}

\bib{Choi}{article}{
      author={Choi, K.~P.},
       title={Some sharp inequalities for martingale transforms},
        date={1988},
        ISSN={0002-9947},
     journal={Trans. Amer. Math. Soc.},
      volume={307},
      number={1},
       pages={279\ndash 300},
         url={https://doi.org/10.2307/2000763},
      review={\MR{936817}},
}

\bib{Tor}{article}{
      author={Ciaurri, \'{O}scar},
      author={Gillespie, T.~Alastair},
      author={Roncal, Luz},
      author={Torrea, Jos\'{e}~L.},
      author={Varona, Juan~Luis},
       title={Harmonic analysis associated with a discrete {L}aplacian},
        date={2017},
        ISSN={0021-7670},
     journal={J. Anal. Math.},
      volume={132},
       pages={109\ndash 131},
         url={https://doi.org/10.1007/s11854-017-0015-6},
      review={\MR{3666807}},
}

\bib{CarSam}{article}{
      author={De~Carli, Laura},
      author={Samad, Gohin~Shaikh},
       title={One-parameter groups of operators and discrete {H}ilbert
  transforms},
        date={2016},
        ISSN={0008-4395},
     journal={Canad. Math. Bull.},
      volume={59},
      number={3},
       pages={497\ndash 507},
         url={https://doi.org/10.4153/CMB-2016-028-7},
      review={\MR{3563731}},
}

\bib{Pet2}{incollection}{
      author={Domelevo, K.},
      author={Os\c{e}kowski, A.},
      author={Petermichl, S.},
       title={Various sharp estimates for semi-discrete {R}iesz transforms of
  the second order},
        date={2018},
   booktitle={50 years with {H}ardy spaces},
      series={Oper. Theory Adv. Appl.},
      volume={261},
   publisher={Birkh\"{a}user/Springer, Cham},
       pages={229\ndash 255},
      review={\MR{3792098}},
}

\bib{Pet3}{article}{
      author={Domelevo, Komla},
      author={Petermichl, Stefanie},
       title={Sharp {$L^p$} estimates for discrete second order {R}iesz
  transforms},
        date={2014},
        ISSN={0001-8708},
     journal={Adv. Math.},
      volume={262},
       pages={932\ndash 952},
         url={https://doi.org/10.1016/j.aim.2014.06.003},
      review={\MR{3228446}},
}

\bib{DraVol4}{article}{
      author={Dragi\v{c}evi\'{c}, Oliver},
      author={Volberg, Alexander},
       title={Bellman function, {L}ittlewood-{P}aley estimates and asymptotics
  for the {A}hlfors-{B}eurling operator in {$L^p(\Bbb C)$}},
        date={2005},
        ISSN={0022-2518},
     journal={Indiana Univ. Math. J.},
      volume={54},
      number={4},
       pages={971\ndash 995},
  url={https://doi-org.ezproxy.lib.purdue.edu/10.1512/iumj.2005.54.2554},
      review={\MR{2164413}},
}

\bib{DraVol3}{article}{
      author={Dragi\v{c}evi\'{c}, Oliver},
      author={Volberg, Alexander},
       title={Bellman functions and dimensionless estimates of
  {L}ittlewood-{P}aley type},
        date={2006},
        ISSN={0379-4024},
     journal={J. Operator Theory},
      volume={56},
      number={1},
       pages={167\ndash 198},
      review={\MR{2261616}},
}

\bib{DraVol2}{article}{
      author={Dragi\v{c}evi\'{c}, Oliver},
      author={Volberg, Alexander},
       title={Bellman function and bilinear embedding theorem for
  {S}chr\"{o}dinger-type operators},
        date={2009},
        ISSN={1631-073X},
     journal={C. R. Math. Acad. Sci. Paris},
      volume={347},
      number={9-10},
       pages={537\ndash 540},
  url={https://doi-org.ezproxy.lib.purdue.edu/10.1016/j.crma.2009.02.004},
      review={\MR{2576904}},
}

\bib{DraVol1}{article}{
      author={Dragi\v{c}evi\'{c}, Oliver},
      author={Volberg, Alexander},
       title={Linear dimension-free estimates in the embedding theorem for
  {S}chr\"{o}dinger operators},
        date={2012},
        ISSN={0024-6107},
     journal={J. Lond. Math. Soc. (2)},
      volume={85},
      number={1},
       pages={191\ndash 222},
         url={https://doi-org.ezproxy.lib.purdue.edu/10.1112/jlms/jdr036},
      review={\MR{2876316}},
}

\bib{Gam}{book}{
      author={Gamelin, T.~W.},
       title={Uniform algebras and {Jensen} measures},
    language={English},
      series={Lond. Math. Soc. Lect. Note Ser.},
   publisher={Cambridge University Press, Cambridge. London Mathematical
  Society, London},
        date={1978},
      volume={32},
}

\bib{GesMonSak}{article}{
      author={Geiss, Stefan},
      author={Montgomery-Smith, Stephen},
      author={Saksman, Eero},
       title={On singular integral and martingale transforms},
        date={2010},
        ISSN={0002-9947},
     journal={Trans. Amer. Math. Soc.},
      volume={362},
      number={2},
       pages={553\ndash 575},
      review={\MR{2551497}},
}

\bib{Graf}{book}{
      author={Grafakos, Loukas},
       title={Classical {F}ourier analysis},
     edition={Third},
      series={Graduate Texts in Mathematics},
   publisher={Springer, New York},
        date={2014},
      volume={249},
        ISBN={978-1-4939-1193-6; 978-1-4939-1194-3},
      review={\MR{3243734}},
}

\bib{GV79}{article}{
      author={Gundy, R.~F.},
      author={Varopoulos, N.~Th.},
       title={Les transformations de {R}iesz et les int{\'e}grales
  stochastiques},
        date={1979},
        ISSN={0151-0509},
     journal={C. R. Acad. Sci. Paris S{\'e}r. A-B},
      volume={289},
      number={1},
       pages={A13\ndash A16},
      review={\MR{545671}},
}

\bib{Gun89}{book}{
      author={Gundy, Richard~F.},
       title={Some topics in probability and analysis},
      series={CBMS Regional Conference Series in Mathematics},
   publisher={Published for the Conference Board of the Mathematical Sciences,
  Washington, DC; by the American Mathematical Society, Providence, RI},
        date={1989},
      volume={70},
        ISBN={0-8218-0721-8},
      review={\MR{995461}},
}

\bib{HarLitPol}{book}{
      author={Hardy, G.~H.},
      author={Littlewood, J.~E.},
      author={P\'{o}lya, G.},
       title={Inequalities},
      series={Cambridge Mathematical Library},
   publisher={Cambridge University Press, Cambridge},
        date={1988},
        ISBN={0-521-35880-9},
        note={Reprint of the 1952 edition},
      review={\MR{944909}},
}

\bib{HunMukWhe}{article}{
      author={Hunt, Richard},
      author={Muckenhoupt, Benjamin},
      author={Wheeden, Richard},
       title={Weighted norm inequalities for the conjugate function and
  {H}ilbert transform},
        date={1973},
        ISSN={0002-9947},
     journal={Trans. Amer. Math. Soc.},
      volume={176},
       pages={227\ndash 251},
         url={https://doi-org.ezproxy.lib.purdue.edu/10.2307/1996205},
      review={\MR{312139}},
}

\bib{IwaKow}{book}{
      author={Iwaniec, Henryk},
      author={Kowalski, Emmanuel},
       title={Analytic number theory},
      series={American Mathematical Society Colloquium Publications},
   publisher={American Mathematical Society, Providence, RI},
        date={2004},
      volume={53},
        ISBN={0-8218-3633-1},
         url={https://doi.org/10.1090/coll/053},
      review={\MR{2061214}},
}

\bib{Iwa82}{article}{
      author={Iwaniec, T.},
       title={Extremal inequalities in {S}obolev spaces and quasiconformal
  mappings},
        date={1982},
        ISSN={0232-2064},
     journal={Z. Anal. Anwendungen},
      volume={1},
      number={6},
       pages={1\ndash 16},
         url={https://doi-org.ezproxy.lib.purdue.edu/10.4171/ZAA/37},
      review={\MR{719167}},
}

\bib{IwaMar}{article}{
      author={Iwaniec, T.},
      author={Martin, G.},
       title={Riesz transforms and related singular integrals},
        date={1996},
        ISSN={0075-4102},
     journal={J. Reine Angew. Math.},
      volume={473},
       pages={25\ndash 57},
      review={\MR{1390681}},
}

\bib{Jacob1}{book}{
      author={Jacob, N.},
       title={Pseudo differential operators and {M}arkov processes. {V}ol.
  {I}},
   publisher={Imperial College Press, London},
        date={2001},
        ISBN={1-86094-293-8},
        note={Fourier analysis and semigroups},
      review={\MR{1873235}},
}

\bib{Kov}{article}{
      author={Kovrizhkin, Oleg},
       title={On the norms of discrete analogues of convolution operators},
        date={2012},
        ISSN={0002-9939},
     journal={Proc. Amer. Math. Soc.},
      volume={140},
      number={4},
       pages={1349\ndash 1352},
  url={https://doi-org.ezproxy.lib.purdue.edu/10.1090/S0002-9939-2011-11043-8},
      review={\MR{2869118}},
}

\bib{Laeng}{article}{
      author={Laeng, Enrico},
       title={Remarks on the {H}ilbert transform and on some families of
  multiplier operators related to it},
        date={2007},
        ISSN={0010-0757},
     journal={Collect. Math.},
      volume={58},
      number={1},
       pages={25\ndash 44},
      review={\MR{2310545}},
}

\bib{Laeng09}{article}{
      author={Laeng, Enrico},
       title={Sharp norm inequalities for the truncated {H}ilbert transform},
        date={2009},
        ISSN={1846-579X},
     journal={J. Math. Inequal.},
      volume={3},
      number={1},
       pages={123\ndash 127},
      review={\MR{2502547}},
}

\bib{LiX}{article}{
      author={Li, Xiang~Dong},
       title={Riesz transforms for symmetric diffusion operators on complete
  {R}iemannian manifolds},
        date={2006},
        ISSN={0213-2230},
     journal={Rev. Mat. Iberoam.},
      volume={22},
      number={2},
       pages={591\ndash 648, loose erratum},
         url={https://doi.org/10.4171/RMI/467},
      review={\MR{2294791}},
}

\bib{Lus-Piq3}{article}{
      author={Lust-Piquard, Fran\c{c}oise},
       title={Dimension free estimates for discrete {R}iesz transforms on
  products of abelian groups},
        date={2004},
        ISSN={0001-8708},
     journal={Adv. Math.},
      volume={185},
      number={2},
       pages={289\ndash 327},
  url={https://doi-org.ezproxy.lib.purdue.edu/10.1016/j.aim.2003.07.002},
      review={\MR{2060471}},
}

\bib{Lus-Piq4}{article}{
      author={Lust-Piquard, Fran\c{c}oise},
       title={Riesz transforms on generalized {H}eisenberg groups and {R}iesz
  transforms associated to the {CCR} heat flow},
        date={2004},
        ISSN={0214-1493},
     journal={Publ. Mat.},
      volume={48},
      number={2},
       pages={309\ndash 333},
         url={https://doi-org.ezproxy.lib.purdue.edu/10.5565/PUBLMAT_48204_02},
      review={\MR{2091008}},
}

\bib{MSW}{article}{
      author={Magyar, A.},
      author={Stein, E.~M.},
      author={Wainger, S.},
       title={Discrete analogues in harmonic analysis: spherical averages},
        date={2002},
        ISSN={0003-486X},
     journal={Ann. of Math. (2)},
      volume={155},
      number={1},
       pages={189\ndash 208},
      review={\MR{1888798}},
}

\bib{Mey1}{incollection}{
      author={Meyer, P.-A.},
       title={Transformations de {R}iesz pour les lois gaussiennes},
        date={1984},
   booktitle={Seminar on probability, {XVIII}},
      series={Lecture Notes in Math.},
      volume={1059},
   publisher={Springer, Berlin},
       pages={179\ndash 193},
         url={https://doi-org.ezproxy.lib.purdue.edu/10.1007/BFb0100043},
      review={\MR{770960}},
}

\bib{Nao}{article}{
      author={Naor, Assaf},
       title={Discrete {R}iesz transforms and sharp metric {$X_p$}
  inequalities},
        date={2016},
        ISSN={0003-486X},
     journal={Ann. of Math. (2)},
      volume={184},
      number={3},
       pages={991\ndash 1016},
  url={https://doi-org.ezproxy.lib.purdue.edu/10.4007/annals.2016.184.3.9},
      review={\MR{3549628}},
}

\bib{Ose2}{article}{
      author={Os\c{e}kowski, Adam},
       title={Inequalities for second-order {R}iesz transforms associated with
  {B}essel expansions},
        date={2020},
        ISSN={0239-7269},
     journal={Bull. Pol. Acad. Sci. Math.},
      volume={68},
      number={1},
       pages={75\ndash 88},
         url={https://doi-org.ezproxy.lib.purdue.edu/10.4064/ba200302-5-10},
      review={\MR{4165504}},
}

\bib{OY21}{article}{
      author={Os\c{e}kowski, Adam},
      author={Yaroslavtsev, Ivan},
       title={The {H}ilbert transform and orthogonal martingales in {B}anach
  spaces},
        date={2021},
        ISSN={1073-7928},
     journal={Int. Math. Res. Not. IMRN},
      number={15},
       pages={11670\ndash 11730},
      review={\MR{4294130}},
}

\bib{Pic72}{article}{
      author={Pichorides, S.~K.},
       title={On the best values of the constants in the theorems of {M}.
  {R}iesz, {Z}ygmund and {K}olmogorov},
        date={1972},
        ISSN={0039-3223},
     journal={Studia Math.},
      volume={44},
       pages={165\ndash 179. (errata insert)},
      review={\MR{312140}},
}

\bib{Pierce}{book}{
      author={Pierce, Lillian~B.},
       title={Discrete analogues in harmonic analysis},
   publisher={ProQuest LLC, Ann Arbor, MI},
        date={2009},
        ISBN={978-1109-13748-4},
        note={Thesis (Ph.D.)--Princeton University},
      review={\MR{2713096}},
}

\bib{Pie5}{article}{
      author={Pierce, Lillian~B.},
       title={On discrete fractional integral operators and mean values of
  {W}eyl sums},
        date={2011},
        ISSN={0024-6093},
     journal={Bull. Lond. Math. Soc.},
      volume={43},
      number={3},
       pages={597\ndash 612},
         url={https://doi-org.ezproxy.lib.purdue.edu/10.1112/blms/bdq127},
      review={\MR{2820148}},
}

\bib{Pie3}{article}{
      author={Pierce, Lillian~B.},
       title={Discrete fractional {R}adon transforms and quadratic forms},
        date={2012},
        ISSN={0012-7094},
     journal={Duke Math. J.},
      volume={161},
      number={1},
       pages={69\ndash 106},
         url={https://doi-org.ezproxy.lib.purdue.edu/10.1215/00127094-1507288},
      review={\MR{2872554}},
}

\bib{Pie2}{article}{
      author={Pierce, Lillian~B.},
       title={A note on discrete fractional integral operators on the
  {H}eisenberg group},
        date={2012},
        ISSN={1073-7928},
     journal={Int. Math. Res. Not. IMRN},
      number={1},
       pages={17\ndash 33},
         url={https://doi-org.ezproxy.lib.purdue.edu/10.1093/imrn/rnr008},
      review={\MR{2874926}},
}

\bib{Pie1}{article}{
      author={Pierce, Lillian~B.},
       title={Correction to ``{D}iscrete fractional {R}adon transforms and
  quadratic forms,'' {D}uke {M}ath. {J}. 161 (2012), 69--106 [mr2872554]},
        date={2013},
        ISSN={0012-7094},
     journal={Duke Math. J.},
      volume={162},
      number={6},
       pages={1203\ndash 1204},
         url={https://doi-org.ezproxy.lib.purdue.edu/10.1215/00127094-2210146},
      review={\MR{3053569}},
}

\bib{Pis}{incollection}{
      author={Pisier, Gilles},
       title={Riesz transforms: a simpler analytic proof of {P}.-{A}. {M}eyer's
  inequality},
        date={1988},
   booktitle={S\'{e}minaire de {P}robabilit\'{e}s, {XXII}},
      series={Lecture Notes in Math.},
      volume={1321},
   publisher={Springer, Berlin},
       pages={485\ndash 501},
         url={https://doi-org.ezproxy.lib.purdue.edu/10.1007/BFb0084154},
      review={\MR{960544}},
}

\bib{Riesz}{article}{
      author={Riesz, Marcel},
       title={Sur les fonctions conjugu\'{e}es},
        date={1928},
        ISSN={0025-5874},
     journal={Math. Z.},
      volume={27},
      number={1},
       pages={218\ndash 244},
      review={\MR{1544909}},
}

\bib{Stein70}{book}{
      author={Stein, E.~M.},
       title={Singular integrals and differentiability properties of
  functions},
      series={Princeton Mathematical Series, No. 30},
   publisher={Princeton University Press, Princeton, N.J.},
        date={1970},
      review={\MR{0290095}},
}

\bib{SteSome}{article}{
      author={Stein, E.~M.},
       title={Some results in harmonic analysis in {${\bf R}^{n}$}, for
  {$n\rightarrow \infty $}},
        date={1983},
        ISSN={0273-0979},
     journal={Bull. Amer. Math. Soc. (N.S.)},
      volume={9},
      number={1},
       pages={71\ndash 73},
  url={https://doi-org.ezproxy.lib.purdue.edu/10.1090/S0273-0979-1983-15157-1},
      review={\MR{699317}},
}

\bib{SteWai99}{article}{
      author={Stein, E.~M.},
      author={Wainger, S.},
       title={Discrete analogues in harmonic analysis. {I}. {$\ell^2$}
  estimates for singular {R}adon transforms},
        date={1999},
        ISSN={0002-9327},
     journal={Amer. J. Math.},
      volume={121},
      number={6},
       pages={1291\ndash 1336},
  url={http://muse.jhu.edu/journals/american_journal_of_mathematics/v121/121.6stein.pdf},
      review={\MR{1719802}},
}

\bib{SW2000}{article}{
      author={Stein, E.~M.},
      author={Wainger, S.},
       title={Discrete analogues in harmonic analysis. {II}. {F}ractional
  integration},
        date={2000},
        ISSN={0021-7670},
     journal={J. Anal. Math.},
      volume={80},
       pages={335\ndash 355},
      review={\MR{1771530}},
}

\bib{Titc26}{article}{
      author={Titchmarsh, E.~C.},
       title={Reciprocal formulae involving series and integrals},
        date={1926},
        ISSN={0025-5874},
     journal={Math. Z.},
      volume={25},
      number={1},
       pages={321\ndash 347},
      review={\MR{1544814}},
}

\bib{Titc27}{article}{
      author={Titchmarsh, E.~C.},
       title={Reciprocal formulae involving series and integrals (correction)},
        date={1927},
        ISSN={0025-5874},
     journal={Math. Z.},
      volume={26},
      number={1},
       pages={496},
      review={\MR{1544871}},
}

\bib{NazVol}{article}{
      author={Volberg, A.},
      author={Nazarov, F.},
       title={Heat extension of the {B}eurling operator and estimates for its
  norm},
        date={2003},
        ISSN={0234-0852},
     journal={Algebra i Analiz},
      volume={15},
      number={4},
       pages={142\ndash 158},
         url={https://doi.org/10.1090/S1061-0022-04-00822-2},
      review={\MR{2068982}},
}

\end{biblist}
\end{bibdiv}

\end{document}